\documentclass[12pt]{article}
\usepackage{color}
\usepackage{bbm}
\usepackage{cite}
\usepackage{psfrag}
\usepackage{ifpdf}

\usepackage{enumerate}
\usepackage{amsmath,amssymb,amsthm,mathrsfs}

\DeclareSymbolFont{largesymbol}{OMX}{yhex}{m}{n}
\DeclareMathAccent{\Widehat}{\mathord}{largesymbol}{"62}

\newtheorem{thm}{Theorem}[section]
\newtheorem{lemma}[thm]{Lemma}

\newtheorem{remark}[]{Remark}

\allowdisplaybreaks \allowdisplaybreaks[2]
\allowdisplaybreaks[4]
\newcounter{RomanNumber}

\numberwithin{equation}{section}

\begin{document}
\bibliographystyle{amsplain}
\title{Nondispersive solutions to the mass critical  half-wave equation in two dimensions\\ }
\author{Vladimir Georgiev{$^b$}\footnote{Corresponding author} ~Yuan Li{$^a$ $^b$}~
\renewcommand\thefootnote{}
\footnote{{E-mail addresses}: liyuan2014@lzu.edu.cn (Y. Li), georgiev@dm.unipi.it}
\setcounter{footnote}{0}
 \\\small ${^a}$ School of Mathematics and Statistics, Lanzhou
 University, Lanzhou, 730000, PR China
 \\\small ${^b}$Dipartimento di Matematica, Universit\`{a} di Pisa, Largo B. Pontecorvo 5, 56100 Pisa, Italy}
\date{ }
\maketitle
\begin{abstract}
We consider the half-wave equation with mass critical in two dimension
\begin{eqnarray*}
\begin{cases}
iu_t=Du-|u|u,\,\,\, \\
u(0,x)=u_0(x),
\end{cases}
\end{eqnarray*}
First, we prove the existence of a family of traveling solitary waves. We then show the existence of finite-time blowup solutions with minimal mass $\|u_0\|_2=\|Q\|_2$, where $Q$ is the ground state solution of equation $DQ+Q=Q^2$.


\end{abstract}

\section{Introduction}\label{section:1}
In this paper, we consider the half-wave equation in two dimension
\begin{equation}\label{equ-1-hf-2}
\begin{cases}
i\partial_tu=Du-|u|u,\\
u(t_0,x)=u_0(x),\ u:I\times\mathbb{R}\rightarrow\mathbb{C}.
\end{cases}
\end{equation}
Here, $I\subset\mathbb{R}$ is an interval containing the initial time $t_0\in\mathbb{R}$, and
\begin{align*}
\widehat{(Df)}(\xi)=|\xi|\hat{f}(\xi)
\end{align*}
denotes the first-order nonlocal fractional derivative. Let us mention that nonlinear half-wave equation have recently attracted some attention in the area of dispersive nonlinear PDE. The evolution problems  like \eqref{equ-1-hf-2} arise in various physical settings, which include equations range from turbulence phenomena \cite{majda1997,majda2001}, wave propagation \cite{Weinstein1987}, continuum limits of lattice system \cite{k-lenzmann2013} and models for gravitational collapse in astrophysics \cite{Ionescu2014,frank-lenzmann2013,eckhaus1983}. We also refer to \cite{Elgart2007,FJL2007,klein2014,cho2013} and the references therein for the background of the fractional Schr\"{o}dinger model in mathematics, numerics and physics.

Let us review some basic properties of this equation. The Cauchy problem
\eqref{equ-1-hf-2} is an infinite-dimensional Hamiltonian system, which has the following three conservation laws:
\begin{align}\label{mass-hf-2}
\text{Mass}\ \ M(u)&=\int_{\mathbb{R}^2}|u(t,x)|^2dx=M(u_0),\\
\label{momentum-hf-2}
\text{Momentum}\ \ P(u)&=\int_{\mathbb{R}^2}-i\nabla u(t,x)\bar{u}(t,x)dx=P(u_0),\\\label{energy-hf-2}
\text{Energy}\ \ E(\psi)&=\frac{1}{2}\int_{\mathbb{R}^2}\bar{u}(t,x)Du(t,x)dx
-\frac{1}{3}\int_{\mathbb{R}^2}|u(t,x)|^{3}dx=E(u_0).
\end{align}

The equation \eqref{equ-1-hf-2} also has the following symmetries:
\begin{itemize}
\item  Phase: if $u(t, x)$ is a solution, then for all $\theta\in\mathbb{R}$, $u(t, x)e^{i\theta}$ is also a solution.
\item Translation: if $u(t,x)$ is a solution, then for all $t_0\in\mathbb{R}$, $x_0\in\mathbb{R}^2$, $u(x-x_0,t-t_0)$ is also a solution.
\item Scaling: if $u(t,x)$ is a solution, then for all $\lambda>0$
\begin{align}\label{Intro-1-2}
u_{\lambda}(t,x)=\frac{1}{\lambda}u\left(\frac{t}{\lambda},\frac{x}{\lambda}\right)
\end{align}
is also a solution.
\end{itemize}

The Cauchy problem \eqref{equ-1-hf-2} is $L^2$-critical since the $L^2$-norm is invariant under the scaling rule \eqref{Intro-1-2}:
\begin{align}\notag
\|u_{\lambda}\|_2=\|u\|_2,\ \text{for all}\ \lambda>0.
\end{align}
From \cite{BGV2018} we known that the Cauchy problem \eqref{equ-1-hf-2} is locally well-posed in energy space $H^{1/2}(\mathbb{R}^2)$. More precisely, for all $u_0\in H^{1/2}(\mathbb{R}^2)$, there exists a unique solution $u(t)\in C([0,T);H^{1/2}(\mathbb{R}^2))$ to \eqref{equ-1-hf-2}. Moreover, we have the blowup alternative that if $u(t)$ is the unique solution with its maximal time of existence $t_0<T\leq\infty$, then
\begin{align}\label{Intro-2-2}
T<+\infty\  \text{implies}\ \lim_{t\rightarrow T^{-}}\|u(t)\|_{H^{1/2}}=+\infty.
\end{align}

A classical criterion of global-in-time existence for $H^{1/2}$ initial data is derived by using the Gagliardo-Nirenberg inequality with best constant
\begin{align}
\|u\|_{3}^3\leq C_{opt}\|D^{\frac{1}{2}}u\|_2^2\|u\|_2, \text{for}\ u\in H^{1/2}(\mathbb{R}^2),
\end{align}
where $C_{opt}=\frac{3}{2\|Q\|_2}$ and $Q$ is the unique ground state solution to
\begin{equation}\label{equ-s=1-hf-2}
D Q+Q=Q^2,\ Q(x)>0,\ Q(x)\in H^{1/2}(\mathbb{R}^2).
\end{equation}
Note that the existence of this equation follows from standard variational techniques, while the uniqueness of $Q$ follows from the result of Frank, Lenzmann and  Silvestre in \cite{frank-lenzmann2013,FrankLS2016}. A combination of  the mass and energy conservation  and the blowup criterion \eqref{Intro-2-2} implies that initial data $u_0\in H^{1/2}(\mathbb{R}^2)$ with
\begin{align}\notag
\|u_0\|_2<\|Q\|_2
\end{align}
generate global-in-time solution.

In this paper, we study  the  following two nondispersive phenomena connected with the focusing 2D half - wave equation \eqref{equ-1-hf-2}.

1) \textbf{Traveling solitary waves of the form}
\begin{align}\notag
u(t,x)=e^{it\mu}Q_c(x-vt)
\end{align}
with some $\mu\in\mathbb{R}$ and traveling velocity $v\in\mathbb{R}^2$.
Bellazzini, Georgiev, Lenzmann and Visciglia \cite{BGLV2019} proved that traveling solitary waves for speed $|v|>1$ does not exist and small data scattering failed in any space dimension. We also refer to \cite{BGV,FJL2007,KLR2013,HS2017,NR2018,Wang-Li-2020} and the references therein for the traveling solitary waves of the fractional Schr\"{o}dinger operator, square root Klein-Gordon operator $\sqrt{-\Delta+m^2}$ and other nonlinearities.
In what follows, let $Q\in H^{1/2}(\mathbb{R}^2)$ be the unique ground state solution of \eqref{equ-s=1-hf-2}.
We can obtain the existence of traveling solitary waves by using a variational approach and  adapting the proof in \cite{FJL2007}. For the half-wave equation \eqref{equ-1-hf-2}, we have the following result.

\begin{thm}\label{Theorem-1-hf-2}
For any $v\in\mathbb{R}^2$ with $0<|v|<1$, there exists a profile $Q_v\in H^{1/2}(\mathbb{R}^2)$ such that
\begin{align}\notag
u(t,x)=e^{it}Q_v(x-vt)
\end{align}
is a traveling solitary waves solution to \eqref{equ-1-hf-2}. Moreover, if $v = e \lambda$ with $|e|=1$ and $\lambda >0$, then the mass $\|Q_v\|_2$ is strictly decreasing with respect to $\lambda$, and for any $\lambda \in (0,1)$, the profile $Q_v$ has strictly subcritical mass:
\begin{align}
\|Q_v\|_2<\|Q\|_2.
\end{align}
Moreover,  the following  limits hold:
\begin{align}\notag
\begin{cases}
\|Q_v\|_2\rightarrow\|Q\|_2\ &\text{as}\ |v|=\lambda \rightarrow 0,\\
\|Q_v\|_2\rightarrow0\ &\text{as}\ |v|=\lambda \rightarrow 1.
\end{cases}
\end{align}
\end{thm}

2) \textbf{Minimal mass blowup solutions.}
There is no general criterion for blowup  solutions in $\mathbb{R}^2$ for $L^2$-critical and $L^2$-supercritical half-wave equation. This is still an open problem ( see \cite{lenzmann-2016blowup} for more details).

For the classical $L^2$-critical nonlinear Schr\"{o}dinger equation, we have the Variance-Virial Laws, which can be expressed as

\begin{align}\notag
\frac{1}{2}\frac{d^2}{dt^2}\left(\int_{\mathbb{R}^2}|x|^2|u(t)|^2dx\right)
=2\frac{d}{dt}\left(\Im\int\bar{u}x\cdot\nabla udx\right)=8E(u_0).
\end{align}
Unlike the $L^2$-critical NLS, for the $L^2$-critical half-wave equation, we only have
\begin{align}\notag
\frac{d}{dt}\left(\Im\int\bar{u}x\cdot\nabla udx\right)=2E(u_0).
\end{align}
However, it seems difficult to represent  the term $\Im\int\bar{u}x\cdot\nabla udx$ as the derivative of some nonnegative one. Possible analogue  of the variance for half-wave equation was suggested in \cite{lenzmann-2016blowup}
\begin{align}\notag
V(u(t)):=\int\bar{u}(t)x\cdot(-\Delta)^{\frac{1}{2}}xu(t)dx=\|x(-\Delta)^{\frac{1}{4}}u(t)\|_2^2.
\end{align}
However, the identity
\begin{align}\notag
\frac{d}{dt}V(u(t))=8\Im\left(\int\bar{u}(t)x\cdot\nabla u(t)dx\right)
\end{align}
is true only if $u(t)$ is a solution of free half-wave equation $i\partial_tu=\sqrt{-\Delta}u$. This observation shows the difficulty to use viral type identity and prove a blow - up result in the mass critical case.

Another difficulty arises, when one tries to construct a minimal blow up solution, following the approach for NLS. This difficulty is connected with the lack of  pseudo-conformal symmetry, that is an essential advantage of NLS.  However, Krieger, Lenzmann and Rapha\"{e}l \cite{KLR2013} constructed a minimal mass blow-up solutions to the mass critical Half-wave equation in one dimension and they obtained that the blowup speed is
\begin{align}\notag
\|D^{1/2}u(t)\|_2\sim\frac{C(u_0)}{|t|}\ \text{as}\ t\rightarrow0^{-}.
\end{align}
But unlike the mass critical NLS \cite{Raphael2011-Jams}, the uniqueness for this minimal mass blow-up solution is still not known.
Our main result is the following.
\begin{thm}(Existence of minimal mass blowup elements)\label{Theorem-minimal-hf-2}
For all $(E_0,P_0)\in \mathbb{R}_+^*\times\mathbb{R}^2$, there exists $t^*<0$ independent of $E_0$, $P_0$ and a minimal mass solution $u\in C^0([t^*,0);H^{1/2}(\mathbb{R}^2))$ of equation \eqref{equ-1-hf-2} with
\begin{align}\notag
\|u\|_2=\|Q\|_2,\ E(u)=E_0,\ P(u)=P_0,
\end{align}
which blow up at time $T=0$. More precisely, it holds that
\begin{align}
u(t,x)-\frac{1}{\lambda(t)}Q\left(\frac{x-\alpha(t)}{\lambda(t)}\right)
e^{i\gamma(t)}\rightarrow0\ \text{in}\ L^2(\mathbb{R}^2)\ \text{as}\ t\rightarrow0^-,
\end{align}
where
\begin{align}\notag
\lambda(t)=\lambda^*t^2+\mathcal{O}(t^5),\ \alpha(t)=\mathcal{O}(t^3),\ \gamma(t)=\frac{1}{\lambda^*|t|}+\mathcal{O}(t),
\end{align}
with some constant $\lambda^*>0$, and the blowup speed is given by:
\begin{align}
\|D^{\frac{1}{2}}u(t)\|_2\sim\frac{C(u_0)}{|t|^2}\ \text{as}\ t\rightarrow0^{-},
\end{align}
where $C(u_0)>0$ is constant  depending only  on the initial data $u_0$.
\end{thm}
This paper is organized as follows: in Section 2, we prove the Theorem \ref{Theorem-1-hf-2}; in Section 3, we construct the high order approximation $Q_{\mathcal{P}}$ solution of the renormalized equation; in Section 4, we decompose the solution and estimate the modulation parameters; in Section 5, we establish a refined energy/virial type estimate; in Section 6, we apply the energy estimate to establish a bootstrap argument that will be needed in the construction of minimal mass blowup solutions; in Section 7, we prove the Theorem \ref{Theorem-minimal-hf-2}; The Section
8 is Appendix.\\
\textbf{Notations}\\
- $(f,g)=\int \bar{f}g$ as the inner product on $L^2(\mathbb{R}^2)$.\\
- $\|\cdot\|_{L^p}$ denotes the $L^p(\mathbb{R}^2)$ norm for $p\geq 1$.\\
- $\widehat{f}$ denotes the Fourier transform of function $f$.\\
- We shall use $X\lesssim Y$ to denote that $X\leq CY$ holds, where the constant $C>0$ may change from line to line, but $C$ is allowed to depend on universally fixed quantities only.\\
- Likewise, we use $X\sim Y$ to denote that both $X\lesssim Y$ and $Y\lesssim X$ hold.
\section{Proof of Theorem \ref{Theorem-1-hf-2}}
In this section we prove Theorem \ref{Theorem-1-hf-2}, which establishes the existence and properties of traveling solitary waves for \eqref{equ-1-hf-2}.

Let $v\in\mathbb{R}^2$ with $|v|<1$ be given. By making the ansatz $u(t,x)=e^{it}Q_v(x-vt)$ for equation \eqref{equ-1-hf-2}, we find that the profile $Q_v\in H^{1/2}(\mathbb{R}^2)$ has to satisfy
\begin{equation}\label{equ-2elliptic-hf-N}
DQ_v+Q_v+i(v\cdot\nabla)Q_v=|Q_v|Q_v.
\end{equation}
Following an idea in \cite{FJL2007}, we obtain nontrivial solutions $Q_v\in H^{1/2}(\mathbb{R}^2)$ as optimizers for the interpolation inequality
\begin{align}\label{interpolation-inequality-hf-N}
\int|u|^{3}\leq C_v\left(\int\bar{u}Du+\bar{u}(iv\cdot\nabla u)\right)\left(\int|u|^2\right)^{1/2}.
\end{align}
Here $C_v>0$ denotes the optimal constant given by Weinstein functional
\begin{align}\label{min-hf-N}
\frac{1}{C_v}=\inf_{u\in H^{1/2}(\mathbb{R}^2)\backslash\{0\}}
\frac{\left(\int\bar{u}Du+\bar{u}(iv\cdot\nabla u)\right)\left(\int|u|^2\right)^{1/2}}{\int|u|^{3}}.
\end{align}
By Sobolev inequalities, we see that the infimum on the right is strictly positive (and hence $C_v<+\infty$). Furthermore, the fact that this infimum is, indeed, attained can be deduced from the concentration-compactness arguments, which is our case follow from a direct adaption of the proof given in \cite{FJL2007}. In particular, optimizers $Q_v\in H^{1/2}(\mathbb{R}^2)$ for \eqref{interpolation-inequality-hf-N} exist, and after a suitable rescaling $Q_v(x)\mapsto aQ(bx)$ with $a,b>0$ they are found to satisfy equation \eqref{equ-2elliptic-hf-N}. Following the terminology introduced in \cite{FJL2007}, we refer to optimizers such as $Q_v$ that
solve Equation \eqref{equ-2elliptic-hf-N} as boosted ground states (with velocity $v$) in what follows. In particular, the unboosted ground state
$Q_{v=0}(x)=Q(x)$ is the unique (modulo symmetries) radial ground state solving \eqref{equ-s=1-hf-2} above. Finally, we observe that
\begin{align}\label{best-constant-hf-N}
C_v=\frac{3}{2}\|Q_v\|_2^{-1},
\end{align}
which follows from the fact that $Q_v$ is an optimizer \eqref{interpolation-inequality-hf-N} and satisfy equation \eqref{equ-2elliptic-hf-N}. In particular, the relation \eqref{best-constant-hf-N} shows that two different boosted ground states $Q_v$ and $\tilde{Q}_v$ with the same velocity $v$ must satisfy $\|Q_v\|_2=\|\tilde{Q}_v\|_2$.

We may reformulate \eqref{best-constant-hf-N} as follows. Let the energy functional
\begin{align}
E_v(u)=\frac{1}{2}\int\bar{u}Du+\frac{1}{2}\int\bar{u}(iv\cdot\nabla u)-\frac{1}{3}\int|u|^{3},
\end{align}
then by the standard Pohozaev identity
\begin{align}\label{energy=0-hf-N}
E_v(Q_v)=0.
\end{align}
Using \eqref{best-constant-hf-N} and the sharp Gagliardo-Nirenberg interpolation inequality:
\begin{align}\label{GN-hf-N}
E_v(u)\geq\frac{1}{2}\left(\int\bar{u}Du+\bar{u}(iv\cdot\nabla u)\right)\left(1-\frac{\|u\|_2}{\|Q_v\|_2}\right).
\end{align}

From the previous paragraph we know that boosted ground states $Q_v$ satisfying equation \eqref{equ-2elliptic-hf-N} exist. Now we prove the behaviour of the boosted ground states.

$\mathbf{Step~1}$ Sign of the momentum. Let $0\leq|v|<1$. We claim:
\begin{align}\label{momentum=0-hf-N}
v\cdot\int\bar{Q}_v(i\nabla Q_v)\leq0.
\end{align}
Indeed, assume on the contrary that $v\cdot\int\bar{Q}_v(i\nabla Q_v)>0$ holds. We define the reflected function $\tilde{Q_v}(x):=Q_v(-x)$. Note that $\int|\tilde{Q}_v|^2=\int|Q_v|^2$  and $v\cdot\int\bar{\tilde{Q}}_v(i\nabla \tilde{Q}_v)<0$. Since the remaining terms in $E_v(u)$ are invariant with respect to space reflections, we find that $E_v(\tilde{Q}_v)<E_v(Q_v)=0$. But $\|\tilde{Q}_v\|_2=\|Q_v\|_2$ implies $E_v(\tilde{Q}_v)\geq0$ from \eqref{GN-hf-N}, a contradiction. We conclude that \eqref{momentum=0-hf-N} holds. In particular, by a suitable (possibly improper) rotation in $\mathbb{R}^2$, we can henceforth assume that
\begin{align}\notag
v=|v|e_1=(|v|,0)\in\mathbb{R}^2
\end{align}
points in (positive) $x_1-$direction. We can see
\begin{align}\label{existence1-hf-N}
\int\bar{Q}_v(i\partial_1Q_v)\leq0\ \ \text{for}\ \ 0<|v|<1.
\end{align}
For the case $v=0$, we recall that the fact from \cite{FrankLS2016} that (after translation and shift by a complex constant phase) the functions $Q_{v=0}(x)=Q(|x|)$ are radial. Hence, in this special case, we have
\begin{align}\label{intergal=0-hf-N}
\int\bar{Q}_{v=0}i\nabla Q_{v=0}=0
\end{align}
$\mathbf{Step~2}$ The mass is non-increasing. We claim the monotonicity:
\begin{align}\label{non-increasing-hf-N}
\|Q_{v_2}\|_2\leq\|Q_{v_1}\|_2\ \text{for}\ 0\leq |v_1|<|v_2|<1,
\end{align}
where $v_j=|v_j|e_1=(|v_j|,0)\in\mathbb{R}^2$, $j=1,2$.

Note that this implies, in particular, the subcritical mass property:
\begin{align}\notag
\|Q_v\|_2<\|Q\|_2\ \text{for}\ 0<|v|<1.
\end{align}
Indeed, let $Q_{v_1}$ and $Q_{v_2}$ be two boosted ground states satisfying \eqref{equ-2elliptic-hf-N} with $v=v_1$ and $v=v_2$, respectively. Since $E_{v_1}(Q_{v_1})=0$ by \eqref{energy=0-hf-N}, we find using \eqref{existence1-hf-N}, if $|v_1|>0$ and \eqref{momentum=0-hf-N} if $|v_1|=0$, that
\begin{align}\notag
E_{v_2}(Q_{v_1})=E_{v_1}(Q_{v_1})+(v_2-v_1)\cdot\int\bar{Q}_{v_1}(i\nabla Q_{v_1}),
\end{align}
since $v_2-v_1=(|v_2|-v_1)e_1=((|v_2|-|v_1|)e_1,0)$ and  $(v_2-v_1)\cdot\int\bar{Q}_{v_1}(i\nabla Q_{v_1})=(|v_2|-v_1)e_1\int\bar{Q}_{v_1}(i\partial_1Q_{v_1})\leq0$. Hence $E_{v_2}(Q_{v_1})\leq0$, which together with \eqref{GN-hf-N} implies $\|Q_{v_1}\|_2\geq\|Q_{v_2}\|_2$. In the case of equality, $\|Q_{v_1}\|_2=\|Q_{v_2}\|_2$, $Q_{v_1}$ attains the minimization problem \eqref{min-hf-N} with $v_2$. In particular, the function $Q_{v_1}$ satisfies the equation
\begin{equation}\notag
DQ_{v_1}+\lambda Q_{v_1}+v_2\cdot\nabla Q_{v_1}-|Q_{v_1}|Q_{v_1}=0
\end{equation}
with the Lagrange multiplier $\lambda\in\mathbb{R}$. On the other hand, by assumption, the boosted ground state $Q_{v_1}$ also satisfies equation \eqref{equ-2elliptic-hf-N} with $v=v_1$. By subtracting the equations satisfied by $Q_{v_1}$, we obtain that
\begin{align}\notag
(\lambda-1)Q_{v_1}+({v_2-v_1})\cdot\nabla Q_{v_1}=0.
\end{align}
Since $v_1\neq v_2$ by assumption and $Q_{v_1}\rightarrow0$ as $|x|\rightarrow\infty$, we deduce from this equation that $Q_{v_1}\equiv0$ holds, which is absurd.

$\mathbf{Step~3}$ Limits. We claim:
\begin{align}\notag
\begin{cases}
\|Q_{v}\|_2\rightarrow\|Q\|_2\ &\text{as}\ |v|\rightarrow0,\\
\|Q_v\|_2\rightarrow0\ &\text{as}\ |v|\rightarrow 1.
\end{cases}
\end{align}
To show this, we argue as follows. From $|\xi|-v\cdot\xi\geq(1-|v|)|\xi|$ for $\xi\in\mathbb{R}^2$ and Plancherel's identity, we deduce that $C_v\leq(1-|v|)^{-1}C_{v=0}$ for the optimal constants in \eqref{interpolation-inequality-hf-N}. From this simple bound and rescalling \eqref{best-constant-hf-N} and the monotonicity \eqref{non-increasing-hf-N}, we deduce that the bounds
\begin{align}\notag
\sqrt{1-|v|}\ \|Q\|_2\leq\|Q_v\|_2\leq\|Q\|_2.
\end{align}
Hence it follows that $\|Q_v\|_2\rightarrow\|Q\|_2$ as $v\rightarrow0$.

It remain to show $\|Q_v\|_2\rightarrow0$ as $|v|\rightarrow 1$. To prove this, from \cite{BGLV2019}, we know that for $|v|<1$, we have the estimate
\begin{align}\notag
C_v\sim(1-|v|)^{-1}\ \text{and}\ \|Q_v\|_2\sim(1-|v|)^2.
\end{align}
Hence, we can easily obtain our result.
\section{Approximate Blowup Profile}
This section is devoted to the construction of the approximate blowup profile.
For a sufficiently regular function $f:\mathbb{R}^2\rightarrow\mathbb{C}$, we define the generator of $L^2$ scaling given by
\begin{align}\notag
\Lambda f:=f+x\cdot\nabla f.
\end{align}
Note that the operator $\Lambda$ is skew-adjoint on $L^2(\mathbb{R}^2)$, that is, we have
\begin{align}\notag
(\Lambda f, g)=-(f,\Lambda g).
\end{align}
We write $\Lambda^kf$, with $k\in\mathbb{N}$, for the iterates of $\Lambda$ with the convention that $\Lambda^0f\equiv f$.

In some parts of this paper, it will be convenient to identity any complex-valued function $f:\mathbb{R}^2\rightarrow\mathbb{C}$ with the function $\mathbf{f}:\mathbb{R}^2\rightarrow\mathbb{R}^2$ by setting
\begin{align}\notag
\mathbf{f}={\left[ \begin{array}{c}
f_1\\
f_2
\end{array}
\right ]}={\left[ \begin{array}{c}
\Re f\\
\Im f
\end{array}
\right ]}.
\end{align}
We also define
$$\mathbf{f}\cdot\mathbf{g}=f_1g_1+f_2g_2.$$
Corresponding, we will identity the multiplication by $i$ in $\mathbb{C}$ with the multiplication by the real $2\times 2$-matrix defined as
\begin{align}\notag
J={\left[\begin{array}{cc}
0 & -1\\
1 & 0
\end{array}\right]}.
\end{align}

We start with a general observation: If $u=u(t,x)$ solves \eqref{equ-1-hf-2}, then we define the function $v=v(s,y)$ by setting
\begin{align}
u(t,x)=\frac{1}{\lambda(t)}v\left(s,\frac{x-\alpha(t)}{\lambda(t)}\right)
e^{i\gamma(t)},\ \ \frac{ds}{dt}=\frac{1}{\lambda(t)}.
\end{align}
It is easy to check that $v=v(s,y)$ with $y=\frac{x-\alpha}{\lambda}$ satisfies
\begin{equation}\label{equ-transform-hf-N}
i\partial_sv-Dv-v+|v|v=i\frac{\lambda_s}{\lambda}\Lambda v+i\frac{\alpha_s}{\lambda}\cdot\nabla v+\tilde{\gamma_s}v,
\end{equation}
where we set $\tilde{\gamma_s}=\gamma_s-1$. Here the operators $D$ and $\nabla$ are understood as $D=D_y$ and $\nabla=\nabla_y$, respectively. Following the slow modulated ansatz strategy developed in \cite{Raphael-2009-cpam,Raphael2011-Jams,KLR2013}, we freeze the modulation
\begin{align}\label{mod1}
-\frac{\lambda_s}{\lambda}=a,\ \ \frac{\alpha_s}{\lambda}=b.
\end{align}
And we look for an approximate solution of the form
\begin{align}\label{eq.der1}
v(s,y) = Q_{\mathcal{P}(s)}(y),\ \mathcal{P}(s)=(a(s),b(s)),
\end{align}
where
\begin{align}\notag
Q_{\mathcal{P}}(y)=Q(y)+ \left(\sum_{k \geq 1}a^{k}R_{k,0}(y) \right) + \left( \sum_{\ell\geq 1} a^{k}\sum_{j=1}^2b^{l}_j R_{k,\ell,j}(y)\right),
\end{align}
where $\mathcal{P} = (a,b) \in \mathbb{R} \times \mathbb{R}^2.$

We shall define ODE for
$a(s), b(s) $ of type
$$ a_s = P_1(a,b), b_s = P_2(a,b),$$
where $P_1,P_2$ are appropriate polynomials in $a,b$.

Using the heuristic asymptotic expansions $$\lambda (t) \sim t^2, \alpha(t) \sim t^3$$ from $\frac{ds}{dt}=\frac{1}{\lambda(t)}$ we see that $s= s_0-1/t$ goes to $\infty$ as $t \nearrow 0$ and $t=1/(s_0-s) \sim -1/s$ as $s \to + \infty.$  Moreover, the modulation relations \eqref{mod1} show that
$$ a(s) = -\frac{\lambda_s}{\lambda} \sim \frac{1}{s} , \ |b(s)| = \frac{|\alpha_s|}{\lambda} \sim \frac{1}{s^2} $$

These asymptotic expansions suggests to define $a(s), b(s)$ so that
\begin{align}\label{as11}
a_s=-\frac{a^2}{2},\ b_s=-ab.
\end{align}

Moreover the asymptotic expansions for $a(s),b(s)$ show that we can consider $\mathcal{P} =(a,b)$ close to the origin with norn
$$ \|\mathcal{P}\|^2 \sim a^2 + |b|.$$

The terms $R_{k,0}(y)$, $R_{k,\ell,j}(y)$ are decomposed in real and imaginary parts as follows
\begin{align}\notag
R_{k,0}(y)=T_{k,0}(y)+iS_{k,0}(y),\,R_{k,\ell,j}(y)=T_{k,\ell,j}(y)+iS_{k,\ell,j}(y).
\end{align}
We also use the notation
\begin{align}\notag
T_{k,\ell}=(T_{k,\ell,1},T_{k,\ell,2})\,\text{and}\,S_{k,\ell}=(S_{k,\ell,1},S_{k,\ell,2}).
\end{align}
We adjust the modulation equation for $(a(s),b(s))$ to ensure the solvability of the obtained system, and a specific algebra leads to the laws to leading order:
\begin{align}\notag
a_s=-\frac{a^2}{2},\ b_s=-ab.
\end{align}

From \eqref{eq.der1} we have
$$ \partial_s v = -\frac{a^2}{2}\partial_aQ_{\mathcal{P}}-a\sum_{j=1}^{2}b_j\partial_{b_j}Q_{\mathcal{P}}$$
Therefore, our purpose is  to construct a high order approximation $Q(y,a,b)=Q_{\mathcal{P}}$ that is solution to
\begin{align}\notag
-i\frac{a^2}{2}\partial_aQ_{\mathcal{P}}-ia\sum_{j=1}^{2}b_j\partial_{b_j}Q_{\mathcal{P}}-DQ_{\mathcal{P}}
-Q_{\mathcal{P}}+ia\Lambda Q_{\mathcal{P}}-i\sum_{j=1}^2b_j\partial_jQ_{\mathcal{P}}+
|Q_{\mathcal{P}}|Q_{\mathcal{P}}=-\Phi_{\mathcal{P}},
\end{align}
where $\mathcal{P}=(a,b)$ is close to $0$ and $\Phi_{\mathcal{P}}$ is some small term of order $\mathcal{O}(\|P\|^5)=\mathcal{O}(a^5+|b|^{5/2})$.

We have the following result about an approximate blowup profile $\mathbf{Q}_{\mathcal{P}}$, parameterized by $\mathcal{P}=(a,b)$, around the ground state $\mathbf{Q}=[Q,0]^{\top}$.
\begin{lemma}(Approximate Blowup Profile)\label{lemma-3app}
Let $\mathcal{P}=(a,b)$. There exists a smooth function $\mathbf{Q}_{\mathcal{P}}=\mathbf{Q}_{\mathcal{P}}(x)$ of the form
\begin{align}\notag
\mathbf{Q}_{\mathcal{P}}=&\mathbf{Q}+a\mathbf{R}_{1,0}+\sum_{j=1}^2b_j\mathbf{R}_{0,1,j}+a\sum_{j=1}^2b_j\mathbf{R}_{1,1,j}
+a^2\mathbf{R}_{2,0}\notag\\&+\sum_{j=1}^2b_j^2\mathbf{R}_{0,2,j}+a^3\mathbf{R}_{3,0}+a^2\sum_{j=1}^2b_j\mathbf{R}_{2,1,j}+a^4\mathbf{R}_{4,0}
\end{align}
that satisfies the equation
\begin{equation}\label{equ-3approxiamte}
-J\frac{1}{2}a^2\partial_a\mathbf{Q}_{\mathcal{P}}-Ja\sum_{j=1}^{2}b_j\partial_{b_j}\mathbf{Q}_{\mathcal{P}}
-D\mathbf{Q}_{\mathcal{P}}-\mathbf{Q}_{\mathcal{P}}+Ja\Lambda\mathbf{Q}_{\mathcal{P}}-J\sum_{j=1}^{2}b_j\partial_{b_j}\mathbf{Q}_{\mathcal{P}}
+|\mathbf{Q}_{\mathcal{P}}|\mathbf{Q}_{\mathcal{P}}
=-\mathbf{\Phi}_{\mathcal{P}}.
\end{equation}
Here the functions $\{\mathbf{R}_{k,l}\}_{0\leq k\leq3,0\leq l\leq 1}$ satisfy the following regularity and decay bounds:
\begin{align}\label{3app-regulairty}
\|\mathbf{R}_{k,l}\|_{H^m}+\|\Lambda\mathbf{R}_{k,l}\|_{H^m}
+\|\Lambda^2\mathbf{R}_{k,l}\|_{H^m}\lesssim1,\ &\text{for}\ m\in\mathbb{N},\\\label{3app-decay}
|\mathbf{R}_{k,l}|+|\Lambda\mathbf{R}_{k,l}|+|\Lambda^2\mathbf{R}_{k,l}|\lesssim
\langle x\rangle^{-3},\ &\text{for}\ x\in\mathbb{R}^2.
\end{align}
Moreover, the term on the right-hand side of \eqref{equ-3approxiamte} satisfies
\begin{align}\label{3app-regu-decay}
\|\mathbf{\Phi}_{\mathcal{P}}\|_{H^m}\lesssim\mathcal{O}(a^5+b^2\mathcal{P}),\ |\nabla\mathbf{\Phi}_{\mathcal{P}}|\lesssim\mathcal{O}(a^5+b^2\mathcal{P})\langle x\rangle^{-3},
\end{align}
for $m\in\mathbb{N}$ and $x\in\mathbb{R}^2$.
\end{lemma}
\begin{proof}
We recall that the definition of the linear operator
\begin{equation}\notag
L={\left[ \begin{array}{cc}
L_+ & 0\\
0 & L_{-}
\end{array}
\right ]}
\end{equation}
acting on $L^2(\mathbb{R}^2,\mathbb{R}^{2})$, where $L_+$ and $L_-$ denote the unbounded operators acting on $L^2(\mathbb{R}^2,\ \mathbb{R}^{2})$ given by
\begin{align}\notag
L_+=D+1-2Q,\ L_-=D+1-Q.
\end{align}
From \cite{FJL2007} we have the key property that the kernel of $L$ is given by
\begin{equation}\notag
\ker L=span\left\{{\left[ \begin{array}{c}
(\partial_{x_1}Q, \partial_{x_2}Q)\\ (0, 0)
\end{array}
\right ]},{\left[ \begin{array}{c}
 0\\Q
\end{array}
\right ]}\right\}.
\end{equation}
Note that the bounded inverse $L^{-1}=diag\{L_+^{-1},L_-^{-1}\}$ exists on the orthogonal complement $\{\ker L\}^{-1}=\{(\partial_{x_1}Q,\partial_{x_2}Q)\}^{\bot}\bigoplus\{Q\}^{\bot}$.

$\mathbf{Step~1}$ Determining the functions $\mathbf{R}_{k,l}$.
We discuss our ansatz for $\mathbf{Q}_{\mathcal{P}}$ to solve \eqref{equ-3approxiamte} order by order. The proof of the regularity and decay bounds for the functions  $\mathbf{R}_{k,l}$ will be given further below.

$\mathbf{Order}$ $\mathcal{O}(1)$: Clearly, we have that
\begin{align}\notag
D\mathbf{Q}+\mathbf{Q}-|\mathbf{Q}|\mathbf{Q}=0.
\end{align}
Since $\mathbf{Q}=[Q,0]^{\top}$, with $Q=Q(|x|)>0$ being the ground state solution.

$\mathbf{Order}$ $\mathcal{O}(a)$: We note that
\begin{align}\notag
|\mathbf{Q}_{\mathcal{P}}|\mathbf{Q}_{\mathcal{P}}
=\mathbf{Q}^2+2a\mathbf{Q}\Re(\mathbf{R}_{1,0})+a\mathbf{Q}\Im(\mathbf{R}_{1,0})+\mathcal{O}(a^2).
\end{align}
Hence, we obtain the equation
\begin{equation}\notag
L\mathbf{R}_{1,0}=J\Lambda\mathbf{Q}.
\end{equation}
Note that $J\Lambda\mathbf{Q}=[0,\Lambda Q]^{\top}$ satisfies $J\Lambda\mathbf{Q}\bot\ker L$ due to the fact that $(\Lambda Q,Q)=0$, which can be easily seen by  using the $L^2$-criticality. Hence we can find a unique solution $\mathbf{R}_{1,0}\bot\ker L$ to the equation above. In what follows, we denote
\begin{align}\notag
\mathbf{R}_{1,0}=L^{-1}J\Lambda\mathbf{Q}={\left[\begin{array}{c}
0\\ L_{-}^{-1}\Lambda Q \end{array}\right]}.
\end{align}

$\mathbf{Order}$ $\mathcal{O}(b_j),\,j=1,2$: Similar to the above discussion. Here we need to solve
\begin{align}\notag
L\mathbf{R}_{0,1,j}=-J\partial_j\mathbf{Q},\,j=1,2.
\end{align}
We observe the orthogonality $J\partial_{x_j}\mathbf{Q}=[0,\partial_{x_j} Q]^{\top}\bot\ker L$, since $(\partial_{x_j}Q,Q)=0$ holds for $j=1,2$. Thus there is a unique solution $\mathbf{R}_{0,1,j}\perp\ker L$, which we denote as
\begin{align}\notag
\mathbf{R}_{0,1,j}=-L^{-1}J\partial_{x_j}\mathbf{Q}=
{\left[\begin{array}{c}
0 \\ -L_{-}^{-1}\partial_{x_j} Q
\end{array}\right]},\,j=1,2.
\end{align}

$\mathbf{Order}$ $\mathcal{O}(ab_j),\,j=1,2$:
\begin{align*}
&|\mathbf{Q}+a\mathbf{R}_{1,0}+b_j\mathbf{R}_{0,1,j}+ab_j\mathbf{R}_{1,1,j}|
(\mathbf{Q}+a\mathbf{R}_{1,0}+b_j\mathbf{R}_{0,1,j}+ab_j\mathbf{R}_{1,1,j})\\
=&Q\Big(1+\frac{ab_j(\mathbf{R}_{1,0}\mathbf{\bar{R}}_{0,1,j}
+\mathbf{\bar{R}}_{1,0}\mathbf{{R}}_{0,1,j})}{Q^2}+\frac{ab_j(\mathbf{R}_{1,1,j}+\mathbf{\bar{R}}_{1,1,j})}{Q}\Big)^{1/2}\\
&(\mathbf{Q}+a\mathbf{R}_{1,0}+b_j\mathbf{R}_{0,1,j}+ab_j\mathbf{R}_{1,1,j})\\
=&\mathbf{Q}Q+2ab_j\mathbf{Q}\Re{\mathbf{R}_{1,1,j}}+ab_j\mathbf{Q}\Im{\mathbf{R}_{1,1,j}}+
+ab_j(\mathbf{R}_{1,0}\mathbf{\bar{R}}_{0,1,j})Q\mathbf{Q}+\mathcal{O}(a^2b_j+ab_j^2).
\end{align*}
We find that $\mathbf{R}_{1,1,j}$ has to solve the equation
\begin{align}\label{3app-1}
L\mathbf{R}_{1,1,j}=-J\mathbf{R}_{0,1,j}+J\Lambda\mathbf{R}_{0,1,j}-J\partial_{x_j}\mathbf{R}_{1,0}
+(\mathbf{R}_{1,0}\mathbf{\bar{R}}_{0,1,j})Q\mathbf{Q},
\end{align}
where we use the fact that $\mathbf{Q}\cdot\mathbf{R}_{1,0}=\mathbf{Q}\cdot\mathbf{R}_{0,1,j}=0$. Now, we need to prove
\begin{align}\label{3app-3}
\text{the right-hand side of the above \eqref{3app-1} is}\ \bot \ker L.
\end{align}
Indeed, we note that
\begin{align}\label{3definition-S}
\mathbf{R}_{1,0}={\left[\begin{array}{c}
0\\S_{1,0}
\end{array}
\right]},\ &\text{with}\ L_{-}S_{1,0}=\Lambda Q,\\\label{3definition-G}
\mathbf{R}_{0,1,j}={\left[\begin{array}{c}
0\\S_{0,1,j}
\end{array}
\right]},\ &\text{with}\ L_{-}S_{0,1,j}=-\partial_{x_j} Q.
\end{align}
Let $S_{0,1}=(S_{0,1,1},S_{0,1,2})$ is a  vector.
Therefore the orthogonality condition \eqref{3app-3} is equivalent to
\begin{align}\label{3app-2}
(\nabla Q,S_{0,1})-(\nabla Q,\Lambda S_{0,1})+(\nabla Q,\nabla S_{1,0})+(\nabla Q,S_{1,0}S_{0,1})=0.
\end{align}
To see that this holds true, we argue as follows. Using the  commutator formula $[\Lambda,\nabla]=-\nabla$ and intergrating by part, we obtain
\begin{align*}
-(\nabla Q,\Lambda S_{0,1})&=(\Lambda\nabla Q,S_{0,1})=(\nabla\Lambda Q,S_{0,1})-(\nabla Q,S_{0,1})\notag\\
&=(\nabla L_{-}S_{1,0},S_{0,1})-(\nabla Q,S_{0,1}).
\end{align*}
Next, since $L_{-}$ is self-adjoint and the definition of $S_{0,1}$, for any function $F$, we have
\begin{align*}
(\nabla L_{-}F,S_{0,1})+(\nabla Q, \nabla F)&=-(L_{-}F,\nabla S_{0,1})-(L_{-}S_{0,1},\nabla F)=(F,[\nabla,L_{-}]S_{0,1})\\
&=-(F,\nabla QS_{0,1}),
\end{align*}
where we use the commutator formulate $[\nabla,L_{-}]=-\nabla Q $. By combining the above equalities, we conclude that \eqref{3app-2} holds. This shows that the \eqref{3app-3} holds, and hence there is a unique solution $\mathbf{R}_{1,1,j}\bot\ker L$ of the equation \eqref{3app-1}. Moreover, we note that
\begin{align}\notag
\mathbf{R}_{1,1,j}={\left[\begin{array}{c}T_{1,1,j}\\0\end{array}\right]},\,j=1,2.
\end{align}

$\mathbf{Order}$ $\mathcal{O}(a^2)$:
\begin{align*}
&|\mathbf{Q}+a\mathbf{R}_{1,0}+a^2\mathbf{R}_{2,0}|(\mathbf{Q}+a\mathbf{R}_{1,0}+a^2\mathbf{R}_{2,0})\\
=&Q\left(1+\frac{a^2(\mathbf{R}_{2,0}+\mathbf{\bar{R}}_{2,0})}{Q}+\frac{a^2|\mathbf{R}_{1,0}|^2}{Q^2}\right)^{1/2}
(\mathbf{Q}+a\mathbf{R}_{1,0}+a^2\mathbf{R}_{2,0})\\
=&\mathbf{Q}Q+2a^2\mathbf{Q}\Re(\mathbf{R}_{2,0})+a^2\mathbf{Q}\Im(\mathbf{R}_{2,0})
+\frac{a^2}{2}|\mathbf{R}_{1,0}|^2Q^{-1}\mathbf{Q}+\mathcal{O}(a^3).
\end{align*}
We find the equation
\begin{align}\notag
L\mathbf{R}_{2,0}=-\frac{1}{2}J\mathbf{R}_{1,0}+J\Lambda\mathbf{R}_{1,0}+
\frac{1}{2}|\mathbf{R}_{1,0}|^2Q^{-1}\mathbf{Q}.
\end{align}
Since $\mathbf{R}_{1,0}=[0,S_{1,0}]^{\top}$ with $L_{-}S_{1,0}=\Lambda Q$, the solvability condition for $\mathbf{R}_{2,0}$ reduces to
\begin{align}\notag
\frac{1}{2}(\nabla Q, S_{1,0})-(\nabla Q,\Lambda S_{1,0})+\frac{1}{2}(\nabla Q, S_{1,0}^2)=0.
\end{align}
However, this is true, since $S_{1,0}$ and $Q$ are the radial functions. Thus there exists a unique solution $\mathbf{R}_{2,0}\bot\ker L$, which is given by
\begin{align}\notag
\mathbf{R}_{2,0}={\left[\begin{array}{c}L_{+}^{-1}
\left(\frac{1}{2}S_{1,0}+\Lambda S_{1,0}-
\frac{1}{2}|S_{1,0}|^2\right)\\0\end{array}\right]}.
\end{align}

$\mathbf{Order}$ $\mathcal{O}(b_j^2)$:
\begin{align*}
&|\mathbf{Q}+b_j\mathbf{R}_{0,1,j}+b_j^2\mathbf{R}_{0,2,j}|
(\mathbf{Q}+b_j\mathbf{R}_{0,1,j}+b_j^2\mathbf{R}_{0,2,j})\\
=&Q\left(1+\frac{b_j^2(\mathbf{R}_{0,2,j}+\mathbf{\bar{R}}_{0,2,j})}{Q}
+\frac{b_j^2|\mathbf{R}_{0,1,j}|^2}{Q^2}\right)^{1/2}(\mathbf{Q}+b_j\mathbf{R}_{0,1,j}+b_j^2\mathbf{R}_{0,2,j})\\
=&\mathbf{Q}Q+2b_j^2\mathbf{Q}\Re\mathbf{R}_{0,2,j}+b_j^2\mathbf{Q}\Im\mathbf{R}_{0,2,j}
+\frac{1}{2}b_j^2|\mathbf{R}_{0,1,j}|^2Q^{-1}\mathbf{Q}+\mathcal{O}(b_j^3).
\end{align*}

We find the equation
\begin{align}\notag
L\mathbf{R}_{0,2}=-J\partial_{x_j} \mathbf{R}_{0,1,j}+\frac{1}{2}|\mathbf{R}_{0,1,j}|^2Q\mathbf{Q}.
\end{align}
Since $\mathbf{R}_{0,1,j}=[0,S_{0,1,j}]^{\top}$ with $L_{-}S_{0,1,j}=-\partial_{x_j} Q$ and $\mathbf{Q}=[Q,0]^{\top}$, the solvability condition reads
\begin{align}\notag
(\partial_{x_j} Q,\nabla S_{0,1,j})+\frac{1}{2}(\partial_{x_j} Q,S_{0,1,j}^2)=0.
\end{align}
Obviously, this is true, since $Q$ is radial function and $S_{0,1,j}$ is antisymmetry function. Hence there exists a unique solution $\mathbf{R}_{0,2,j}\bot\ker L$, and we have
\begin{align}\notag
\mathbf{R}_{0,2,j}={\left[\begin{array}{c}L_{+}^{-1}\left(\partial_{x_j} S_{0,1,j}+\frac{1}{2}|S_{0,1,j}|^2\right)\\0\end{array}\right]}.
\end{align}

$\mathbf{Order}$ $\mathcal{O}(a^3)$:
\begin{align*}
&|\mathbf{Q}+a\mathbf{R}_{1,0}+a^2\mathbf{R}_{2,0}+a^3\mathbf{R}_{3,0}|
(\mathbf{Q}+a\mathbf{R}_{1,0}+a^2\mathbf{R}_{2,0}+a^3\mathbf{R}_{3,0})\\
=&Q\Big(1+\frac{a(\mathbf{R}_{1,0}+\mathbf{\bar{R}}_{1,0})}{Q}+
\frac{a^2(\mathbf{R}_{2,0}+\mathbf{\bar{R}}_{2,0})}{Q}+
\frac{a^3(\mathbf{R}_{3,0}+\mathbf{\bar{R}}_{3,0})}{Q}
+\frac{a^3(\mathbf{R}_{1,0}\mathbf{\bar{R}}_{2,0}
+\mathbf{\bar{R}}_{1,0}\mathbf{R}_{2,0})}{Q^2}\\
&+\frac{a^2|\mathbf{R}_{1,0}|^2}{Q^2}\Big)^{1/2}
(\mathbf{Q}+a\mathbf{R}_{1,0}+a^2\mathbf{R}_{2,0}+a^3\mathbf{R}_{3,0})\\
=&Q\left(1+\frac{2a^2\mathbf{R}_{2,0}}{Q}+
\frac{a^3(\mathbf{R}_{3,0}+\mathbf{\bar{R}}_{3,0})}{Q}+\frac{a^2|\mathbf{R}_{1,0}|^2}{Q^2}\right)^{1/2}
(\mathbf{Q}+a\mathbf{R}_{1,0}+a^2\mathbf{R}_{2,0}+a^3\mathbf{R}_{3,0})\\
=&\mathbf{Q}Q+2a^3\mathbf{Q}\Re(\mathbf{R}_{3,0})+a^3\mathbf{Q}\Im(\mathbf{R}_{3,0})+a^3\Re \mathbf{R}_{2,0}\mathbf{R}_{1,0}+\frac{a^3}{2}Q^{-1}|\mathbf{R}_{1,0}|^2\mathbf{R}_{1,0}+\mathcal{O}(a^4).
\end{align*}

We notice that $\mathbf{R}_{1,0}\cdot\mathbf{R}_{2,0}=0$ and we obtain the equation
\begin{align}
L\mathbf{R}_{3,0}=-J\mathbf{R}_{2,0}+J\Lambda\mathbf{R}_{2,0}
+\Re\mathbf{R}_{2,0}\mathbf{R}_{1,0}+\frac{1}{2}Q^{-1}|\mathbf{R}_{1,0}|^2\mathbf{R}_{1,0}.
\end{align}
Note that the right side is of the form $[0,f]^{\top}$ with some nontrivial $f$. Hence the solvability condition for $\mathbf{R}_{3,0}$ is equivalent to
\begin{align}\label{3app-4}
-(Q,T_{2,0})+(Q,\Lambda T_{2,0})+(Q,T_{2,0}S_{1,0})+\frac{1}{2}(Q,Q^{-1}S_{1,0}^2S_{1,0})=0,
\end{align}
where the functions $S_{1,0}$ and $T_{2,0}$ satisfy
\begin{align}\label{3definition-T}
L_{-}S_{1,0}=\Lambda Q,\ \ L_{+}T_{2,0}=\frac{1}{2}S_{1,0}-\Lambda S_{1,0}+
\frac{1}{2}|S_{1,0}|^2.
\end{align}
To see that \eqref{3app-4} holds, we first note that
\begin{align*}
&\text{Left-hand side of \eqref{3app-4}}\\
&=-(Q,T_{2,0})-(\Lambda Q, T_{2,0})+(Q,T_{2,0}S_{1,0})+\frac{1}{2}(Q,Q^{-1}S_{1,0}^2S_{1,0})\\
&=-(Q,T_{2,0})-(L_{-}S_{1,0},T_{2,0})+(Q,T_{2,0}S_{1,0})+\frac{1}{2}(Q,Q^{-1}S_{1,0}^2S_{1,0})\\
&=-(Q,T_{2,0})-(L_{+}S_{1,0},T_{2,0})+\frac{1}{2}(Q,Q^{-1}S_{1,0}^2S_{1,0})\\
&=-(Q,T_{2,0})-\frac{1}{2}(S_{1,0},S_{1,0})+(S_{1,0},\Lambda S_{1,0})-\frac{1}{2}(S_{1,0},S_{1,0}^2)+\frac{1}{2}(Q,Q^{-1}S_{1,0}^2S_{1,0})\\
&=-(Q,T_{2,0})-\frac{1}{2}(S_{1,0},S_{1,0}),
\end{align*}
where in the last step we used that $(S_{1,0},\Lambda S_{1,0})=0$, since $\Lambda^*=-\Lambda$. Thus it remains to show that
\begin{align}\label{3app-relation}
-(Q,T_{2,0})=\frac{1}{2}(S_{1,0},S_{1,0}).
\end{align}
Indeed, by using $L_{+}\Lambda Q=-Q$ and the equations for $T_{2,0}$ and $S_{1,0}$ above, we deduce
\begin{align}\label{3app-5}
-(Q,T_{2,0})&=(\Lambda Q,\frac{1}{2}S_{1,0}-\Lambda S_{1,0}+
\frac{1}{2}|S_{1,0}|^2)\notag\\
&=\frac{1}{2}(L_{-}S_{1,0},S_{1,0})-(L_{-}S_{1,0},\Lambda S_{1,0})+\frac{1}{2}(\Lambda Q,S_{1,0}^2)\notag\\
&=\frac{1}{2}(S_{1,0},DS_{1,0})+\frac{1}{2}(S_{1,0},S_{1,0})
-\frac{1}{2}(S_{1,0},QS_{1,0})\notag\\
&-(L_{-}S_{1,0},\Lambda S_{1,0})+\frac{1}{2}(\Lambda Q,S_{1,0}^2).
\end{align}
Next, we apply the commutator formula $(L_{-}f,\Lambda f)=\frac{1}{2}(f,[L_{-},\Lambda]f)$, which shows that
\begin{align*}
(L_{-}S_{1,0},\Lambda S_{1,0})&=\frac{1}{2}(S_{1,0},[L_{-},\Lambda]S_{1,0})
=\frac{1}{2}(S_{1,0},[D,\Lambda]S_{1,0})-\frac{1}{2}(S_{1,0},[Q,\Lambda]S_{1,0})\\
&=\frac{1}{2}(S_{1,0},DS_{1,0})+\frac{1}{2}(S_{1,0},(x\cdot\nabla Q)S_{1,0}),
\end{align*}
where we use the commutator $[D,x\cdot\nabla]=D$ and $[Q,\Lambda]=-x\cdot\nabla Q$ holds. Furthermore, we have the pointwise identity
\begin{align*}
-x\cdot\nabla Q+\Lambda Q=Q.
\end{align*}
If now we insert the above two equalities into \eqref{3app-5}, we obtain the desired relation \eqref{3app-relation}, and thus the solvability condition \eqref{3app-4} holds. Note that $\mathbf{R}_{3,0}=[0,S_{3,0}]^{\top}$ with some radial function $S_{3,0}$.

$\mathbf{Order}$ $\mathcal{O}(a^4)$:
\begin{align*}
&|\mathbf{Q}+a\mathbf{R}_{1,0}+a^2\mathbf{R}_{2,0}+a^3\mathbf{R}_{3,0}+a^4\mathbf{R}_{4,0}|
(\mathbf{Q}+a\mathbf{R}_{1,0}+a^2\mathbf{R}_{2,0}+a^3\mathbf{R}_{3,0}+a^4\mathbf{R}_{4,0})\\
=&Q\Big(1+\frac{a(\mathbf{R}_{1,0}+\mathbf{\bar{R}}_{1,0})}{Q}+
\frac{a^2(\mathbf{R}_{2,0}+\mathbf{\bar{R}}_{2,0})}{Q}+
\frac{a^3(\mathbf{R}_{3,0}+\mathbf{\bar{R}}_{3,0})}{Q}+
\frac{a^4(\mathbf{R}_{4,0}+\mathbf{\bar{R}}_{4,0})}{Q}\\
&+\frac{a^4(\mathbf{R}_{1,0}\mathbf{\bar{R}}_{3,0}+\mathbf{\bar{R}}_{1,0}\mathbf{R}_{3,0})}{Q^2}
+\frac{a^3(\mathbf{R}_{1,0}\mathbf{\bar{R}}_{2,0}+\mathbf{\bar{R}}_{1,0}\mathbf{R}_{2,0})}{Q^2}
+\frac{a^2|\mathbf{R}_{1,0}|^2}{Q^2}+\frac{a^4|\mathbf{R}_{2,0}|^2}{Q^2}\Big)^{1/2}\\
&(\mathbf{Q}+a\mathbf{R}_{1,0}+a^2\mathbf{R}_{2,0}+a^3\mathbf{R}_{3,0}+a^4\mathbf{R}_{4,0})\\
=&Q\Big(1+\frac{1}{2}\Big(\frac{2a^2\Re\mathbf{R}_{2,0}}{Q}+
\frac{a^4(\mathbf{R}_{4,0}+\mathbf{\bar{R}}_{4,0})}{Q}
+\frac{2a^4\Re(\mathbf{R}_{1,0}\mathbf{\bar{R}}_{3,0})}{Q^2}
+\frac{a^2|\mathbf{R}_{1,0}|^2}{Q^2}+\frac{a^4|\mathbf{R}_{2,0}|^2}{Q^2}\Big)\\
&-\frac{1}{8}\Big(\frac{2a^2\Re\mathbf{R}_{2,0}}{Q}+\frac{a^2|\mathbf{R}_{1,0}|^2}{Q^2}\Big)^2\Big)
(\mathbf{Q}+a\mathbf{R}_{1,0}+a^2\mathbf{R}_{2,0}+a^3\mathbf{R}_{3,0}+a^4\mathbf{R}_{4,0})\\
=&\mathbf{Q}Q+2a^4\mathbf{Q}\Re\mathbf{R}_{4,0}+a^4\mathbf{Q}\Im\mathbf{R}_{4,0}
+a^4\Re(\mathbf{R}_{1,0}\mathbf{\bar{R}}_{3,0})Q^{-1}\mathbf{Q}
+\frac{a^4}{2}|\mathbf{R}_{2,0}|^2Q^{-1}\mathbf{Q}+
a^4|\mathbf{R}_{2,0}|^2Q^{-1}\mathbf{Q}\\
&-\frac{a^4}{2}|\mathbf{R}_{2,0}|^2Q^{-1}\mathbf{Q}
-\frac{a^4}{8}\frac{|\mathbf{R}_{1,0}|^4}{Q^3}\mathbf{Q}
-\frac{1}{2}\frac{a^4\Re\mathbf{R}_{2,0}|\mathbf{R}_{1,0}|^2}{Q^2}\mathbf{Q}+\mathcal{O}(a^5)\\
=&\mathbf{Q}Q+2a^4\mathbf{Q}\Re\mathbf{R}_{4,0}+a^4\mathbf{Q}\Im\mathbf{R}_{4,0}
+a^4\Re(\mathbf{R}_{1,0}\mathbf{\bar{R}}_{3,0})Q^{-1}\mathbf{Q}+a^4|\mathbf{R}_{2,0}|^2Q^{-1}\mathbf{Q}\\
&-\frac{a^4}{8}\frac{|\mathbf{R}_{1,0}|^4}{Q^3}\mathbf{Q}
-\frac{1}{2}\frac{a^4\Re\mathbf{R}_{2,0}|\mathbf{R}_{1,0}|^2}{Q^2}\mathbf{Q}+\mathcal{O}(a^5).
\end{align*}

We obtain the equation
\begin{align}\label{3app-6}
L\mathbf{R}_{4,0}=&-\frac{3}{2}J\mathbf{R}_{3,0}+J\Lambda\mathbf{R}_{3,0}+
+a^4\Re(\mathbf{R}_{1,0}\mathbf{\bar{R}}_{3,0})Q^{-1}\mathbf{Q}+|\mathbf{R}_{2,0}|^2Q^{-1}\mathbf{Q}\notag\\
&-\frac{1}{8}\frac{|\mathbf{R}_{1,0}|^4}{Q^3}\mathbf{Q}
-\frac{1}{2}\frac{\Re\mathbf{R}_{2,0}|\mathbf{R}_{1,0}|^2}{Q^2}\mathbf{Q},
\end{align}
where we use the fact that $\mathbf{R}_{1,0}\cdot\mathbf{Q}=\mathbf{R}_{3,0}\cdot\mathbf{Q}=0$. Moreover, we easily  see that
\begin{align}\notag
\text{Right-hand side of \eqref{3app-6}}\ \bot\ker L,
\end{align}
since the right-hand side of \eqref{3app-6} is radial and $(F,\nabla Q)=0$ for any radial function $F\in L^2({\mathbb{R}^2})$. Hence there is a unique solution $\mathbf{R}_{4,0}\bot\ker L$ of equation \eqref{3app-6}, and we have that $\mathbf{R}_{4,0}=[T_{4,0},0]^{\top}$ holds with some radial function $T_{4,0}$.

$\mathbf{Order}$ $\mathcal{O}(a^2b_j),\,j=1,2$:
\begin{align*}
&|\mathbf{Q}+a\mathbf{R}_{1,0}+b_j\mathbf{R}_{0,1,j}+ab_j\mathbf{R}_{1,1,j}+a^2\mathbf{R}_{2,0}+a^2b_j\mathbf{R}_{2,1,j}|\\
&(|\mathbf{Q}+a\mathbf{R}_{1,0}+b_j\mathbf{R}_{0,1,j}+ab_j\mathbf{R}_{1,1,j}+a^2\mathbf{R}_{2,0}+a^2b_j\mathbf{R}_{2,1,j})\\
=&Q\Big(1+\frac{a(\mathbf{R}_{1,0}+\mathbf{\bar{R}}_{1,0})}{Q}+
\frac{b_j(\mathbf{R}_{0,1,j}+\mathbf{\bar{R}}_{0,1,j})}{Q}+
\frac{ab_j(\mathbf{R}_{1,1,j}+\mathbf{\bar{R}}_{1,1,j})}{Q}+
\frac{a^2(\mathbf{R}_{2,0}+\mathbf{\bar{R}}_{2,0})}{Q}\\
&+\frac{a^2b_j(\mathbf{R}_{2,1,j}+\mathbf{\bar{R}}_{2,1,j})}{Q}
+\frac{a^2|\mathbf{R}_{1,0}|^2}{Q^2}
+\frac{ab_j(\mathbf{R}_{1,0}\mathbf{\bar{R}}_{0,1,j})+\mathbf{\bar{R}}_{1,0}\mathbf{R}_{0,1,j})}{Q^2}\\
&+\frac{a^2b_j(\mathbf{R}_{1,0}\mathbf{\bar{R}}_{1,1,j})+\mathbf{\bar{R}}_{1,0}\mathbf{R}_{1,1,j})}{Q^2}
+\frac{a^2b_j(\mathbf{R}_{2,0}\mathbf{\bar{R}}_{0,1,j})+\mathbf{\bar{R}}_{2,0}\mathbf{R}_{0,1,j})}{Q^2}\Big)^{1/2}\\
&(|\mathbf{Q}+a\mathbf{R}_{1,0}+b_j\mathbf{R}_{0,1,j}+ab_j\mathbf{R}_{1,1,j}+a^2\mathbf{R}_{2,0}+a^2b_j\mathbf{R}_{2,1,j})\\
=&Q\Big(1+
\frac{2ab_j\Re\mathbf{R}_{1,1,j}}{Q}+
\frac{2a^2\Re\mathbf{R}_{2,0}}{Q}
+\frac{a^2b_j(\mathbf{R}_{2,1,j}+\mathbf{\bar{R}}_{2,1,j})}{Q}
+\frac{a^2|\mathbf{R}_{1,0}|^2}{Q^2}
+\frac{2ab_j\Re(\mathbf{R}_{1,0}\mathbf{\bar{R}}_{0,1,j})}{Q^2}\Big)^{1/2}\\
&(|\mathbf{Q}+a\mathbf{R}_{1,0}+b_j\mathbf{R}_{0,1,j}+ab_j\mathbf{R}_{1,1,j}+a^2\mathbf{R}_{2,0}+a^2b_j\mathbf{R}_{2,1,j})\\
=&Q\Big(1+\frac{1}{2}\Big(
\frac{2ab_j\Re\mathbf{R}_{1,1,j}}{Q}+
\frac{2a^2\Re\mathbf{R}_{2,0}}{Q}
+\frac{a^2b_j(\mathbf{R}_{2,1,j}+\mathbf{\bar{R}}_{2,1,j})}{Q}
+\frac{a^2|\mathbf{R}_{1,0}|^2}{Q^2}
+\frac{2ab_j\Re(\mathbf{R}_{1,0}\mathbf{\bar{R}}_{0,1,j})}{Q^2}\Big)\Big)\\
&(|\mathbf{Q}+a\mathbf{R}_{1,0}+b_j\mathbf{R}_{0,1,j}+ab_j\mathbf{R}_{1,1,j}
+a^2\mathbf{R}_{2,0}+a^2b_j\mathbf{R}_{2,1,j})+\mathcal{O}(a^2b^2+a^3b+ a^3+b^2)\\
=&\mathbf{Q}Q+2a^2b_j\mathbf{Q}\Re\mathbf{R}_{2,1,j}+a^2b_j\mathbf{Q}\Im\mathbf{R}_{2,1,j}
+a^2b_j\Re\mathbf{R}_{1,1,j}\mathbf{R}_{1,0}\\
&+a^2b_j(\mathbf{R}_{1,0}\cdot\mathbf{\bar{R}}_{0,1,j})\mathbf{R}_{1,0}
+a^2b_j\frac{|\mathbf{R}_{1,0}|^2\mathbf{R}_{0,1,j}}{Q}+\mathcal{O}(a^2b^2+a^3b+ a^3+b^2).
\end{align*}
Note that $\mathbf{R}_{1,0}\cdot\mathbf{Q}=\mathbf{R}_{1,1,j}\cdot\mathbf{R}_{1,0}
=\mathbf{R}_{0,1,j}\cdot\mathbf{R}_{2,0}=0$. We find the equation
\begin{align}\label{3app-7}
L\mathbf{R}_{2,1,j}=&\frac{3}{2}J\mathbf{R}_{1,1,j}+J\Lambda\mathbf{R}_{1,1,j}
-J\partial_j\mathbf{R}_{2,0}+\Re\mathbf{R}_{1,1,j}\mathbf{R}_{1,0}\notag\\
&+(\mathbf{R}_{1,0}\cdot\mathbf{\bar{R}}_{0,1,j})\mathbf{R}_{1,0}
+\frac{|\mathbf{R}_{1,0}|^2\mathbf{R}_{0,1,j}}{Q}.
\end{align}
Using the symmetries of the previously constructed functions, we can check that
\begin{align}\notag
\text{Right-hand side of \eqref{3app-7}} \bot\ker L,
\end{align}
Since $(g,Q)=0$ for any antisymmetry function $g\in L^{2}(\mathbb{R}^2)$. Thus there is a unique solution $\mathbf{R}_{2,1,j}\bot\ker L$ of the equation \eqref{3app-7}, and we set  that $\mathbf{R}_{2,1,j}=[0,S_{2,1,j}]^{\top}$ with some radial function $S_{2,1,j}$, $j=1,2$.

$\mathbf{Step~2}$ Regularity and decay bounds. Let $m \in [0,2]$ be given. First, we recall that $\|Q\|_{H^m}\lesssim1$ and $|Q(x)|\lesssim\langle x\rangle^{-3}$ holds. Since, moreover,  $L_{+}\Lambda Q=-Q$ and $(\Lambda Q,Q)=0$, we can apply lemma \ref{app-lemma-1} to conclude that
\begin{align}
\|\Lambda Q\|_{H^m}\lesssim1,\ \ |\Lambda Q|\lesssim\langle x\rangle^{-3}.
\end{align}
Next, by applying $\Lambda$ to the equation $L_{+}\Lambda Q=-Q$ and using that $[L_{+},\Lambda]=D+2x\cdot\nabla Q$, we deduce that
\begin{align}
L_{+}\{-\Lambda^2Q+\Lambda Q+\alpha Q \}=-(Q^2+x\cdot\nabla Q)\Lambda Q-(1-2\alpha) Q^2.
\end{align}
for $\alpha\in\mathbb{R}$.

Using the previous bounds for $Q$ and $\Lambda Q$ (and hence for $x\cdot\nabla Q$) as well, we can apply lemma \ref{app-lemma-1} again to obtain the bounds
\begin{align}\notag
\|\Lambda^2 Q\|_{H^m}\lesssim1,\ \ |\Lambda^2Q(x)|\lesssim\langle x\rangle^{-3}.
\end{align}
Having these bounds for $\mathbf{Q}=[Q,0]^{\top}$, $\Lambda\mathbf{Q}=[\Lambda Q,0]^{\top}$, and $\Lambda^2\mathbf{Q}=[\Lambda^2 Q,0]^{\top}$ at hand, we can now prove the claimed bound \eqref{3app-regulairty} and \eqref{3app-decay} by iterating the equations satisfied  by the functions $\{\mathbf{R}_{k,l}\}_{0\leq k\leq3,0\leq l\leq1}$ above. For instance, recall that $\mathbf{R}_{1,0}=[0,S_{1,0}]^{\top}$ with $L_{-}S_{1,0}=\Lambda Q$ and hence $\Lambda L_{-}S_{1,0}=\Lambda^2Q$. Then, by using the commutator $[L_{+},\Lambda]$ and the previous estimates for $\{Q,\Lambda Q,\Lambda^2Q\}$, we derive that
\begin{align}
\|\Lambda^kS_{1,0}\|_{H^m}\lesssim1,\ |\Lambda^kS_{1,0}(x)|\lesssim\langle x\rangle^{-3},\ \ \text{for}\ k=0,1,2\ \text{and}\ m\geq0.
\end{align}
Using this and proceeding in the above, we deduce that \eqref{3app-regulairty} and \eqref{3app-decay} holds.

Finally, we mention that the bounds \eqref{3app-regu-decay} for the error term $\mathbf{\Phi}_{\mathcal{P}}$ follow from expanding $|\mathbf{Q}_{\mathcal{P}}|\mathbf{Q}_{\mathcal{P}}$ and using the regularity and decay bounds for the functions  $\{\mathbf{R}_{k,l}\}$. We omit the straightforward details. The proof of lemma \ref{lemma-3app} is now complete.
\end{proof}
\begin{remark}
Note that $L_{-}>0$ on $Q^{\bot}$ and we have $S_{1,0}\bot Q$ and $S_{0,1,j}\bot Q$, $j=1,2$.
\end{remark}
\begin{remark}
The proof of lemma \ref{lemma-3app} will actually show that the functions $\{\mathbf{R}_{k,l}\}$ have the following symmetry structure
\begin{align*}
&\mathbf{R}_{1,0}=\left[\begin{array}{c}0\\symmetry\end{array}\right],\
\mathbf{R}_{0,1,j}=\left[\begin{array}{c}0\\antisymmetry\end{array}\right],\
\mathbf{R}_{1,1,j}=\left[\begin{array}{c}antisymmetry\\0\end{array}\right],\\
&\mathbf{R}_{2,0}=\left[\begin{array}{c}symmetry\\0\end{array}\right],\
\mathbf{R}_{0,2,j}=\left[\begin{array}{c}symmetry\\0\end{array}\right],\
\mathbf{R}_{3,0}=\left[\begin{array}{c}0\\symmetry\end{array}\right],\\
&\mathbf{R}_{2,1,j}=\left[\begin{array}{c}0\\antisymmetry\end{array}\right],\
\mathbf{R}_{4,0}=\left[\begin{array}{c}symmetry\\0\end{array}\right].
\end{align*}
These symmetry properties will be of essential use in the sequel.
\end{remark}

We now turn to some key properties of the approximate blowup profile $\mathbf{Q}_{\mathcal{P}}$ constructed in lemma \ref{lemma-3app}.
\begin{lemma}\label{lemma-3app-2}
The mass, the energy and the linear momentum of $\mathbf{Q}_{\mathcal{P}}$ satisfy
\begin{align*}
\int|\mathbf{Q}_{\mathcal{P}}|^2&=\int  Q^2+\mathcal{O}(a^4+b^2+b\mathcal{P}^2),\\
E(\mathbf{Q}_{\mathcal{P}})&=e_1a^2+\mathcal{O}(a^4+b^2+b\mathcal{P}^2),\\
P(\mathbf{Q}_{\mathcal{P}})&=p_1b+\mathcal{O}(a^4+b^2+b\mathcal{P}^2).
\end{align*}
Here $e_1>0$ and $p_1>0$ are the positive constants given by
\begin{align}\notag
e_1=\frac{1}{2}(L_{-}S_{1,0},S_{1,0}),\ p_1=2\int_{\mathbb{R}^2}L_{-}S_{0,1}\cdot S_{0,1},
\end{align}
where $S_{1,0}$ and $S_{0,1}$ satisfy $L_{-}S_{1,0}=\Lambda Q$ and $L_{-}S_{0,1}=-\nabla Q$, respectively.
\end{lemma}
\begin{proof}
From the proof of lemma \ref{lemma-3app}, we recall that the facts that $\mathbf{R}_{1,0}=[0,S_{1,0}]^{\top}$, $\mathbf{R}_{0,1,j}=[0,S_{0,1,j}]^{\top}$ and $\mathbf{R}_{1,1,j}=[T_{1,1,j},0]^{\top}$ with some antisymmetry function. Hence we have $\int\mathbf{Q}\cdot\mathbf{R}_{0,1,j}
=\int\mathbf{Q}\cdot\mathbf{R}_{1,0}=\int\mathbf{Q}\cdot\mathbf{R}_{1,1,j}=0$, $j=1,2$. Next, we recall that $\mathbf{R}_{2,0}=[T_{2,0},0]^{\top}$ satisfies $(S_{1,0},S_{1,0})+2(Q,T_{2,0})=0$, shown in \eqref{3app-relation} above. In summary,we see that
\begin{align}\notag
\int|\mathbf{Q}_{\mathcal{P}}|^2&=\int  Q^2+\mathcal{O}(a^4+b^2+b\mathcal{P}^2).
\end{align}

For the expansion for the linear momentum functional, we observe that $P(\mathbf{f})=2\int f_1\nabla f_2$, where $\mathbf{f}=[f_1,f_2]^{\top}$. Hence
\begin{align*}
P(\mathbf{Q}_{\mathcal{P}})=&2a\int Q\nabla S_{1,0}+2\sum_{j=1}^2b_j\int Q\nabla S_{0,1,j}+2a^2\sum_{j=1}^2b_j\int T_{1,1,j}\nabla S_{1,0}+2a^3\int T_{2,0}\nabla S_{1,0}\\
&+\mathcal{O}(a^4+b^2+b\mathcal{P}^2)\\
=&-2\sum_{j=1}^2b_j( \nabla Q, S_{0,1,j})+\mathcal{O}(a^4+b^2+b\mathcal{P}^2)\\
=&2\sum_{j=1}^2b_j(L_{-}S_{0,1},S_{0,1,j})+\mathcal{O}(a^4+b^2+b\mathcal{P}^2)\\
=&bp_1+\mathcal{O}(a^4+b^2+b\mathcal{P}^2),
\end{align*}
Since $L_{-}S_{0,1,j}=-\partial_{x_j} Q$, and using that $\int Q\nabla S_{1,0}+\int T_{1,1,j}\nabla S_{1,0}+\int T_{2,0}\nabla S_{1,0}=0$ due to the fact that $Q,S_{1,0},T_{2,0}$ are the radial symmetry functions.

To treat the expansion of the energy, we first recall that $E(\mathbf{Q})=0$ and $DQ+Q-Q^2=0$ and $E'(\mathbf{Q})=-Q$. Moreover, since we have $(Q,S_{1,0})=0$ and $(Q,S_{0,1,j})=0$, $j=1,2$ we obtain
\begin{align*}
E(\mathbf{Q}_{\mathcal{P}})
=&\frac{1}{2}(\mathbf{Q}_{\mathcal{P}},D\mathbf{Q}_{\mathcal{P}})
-\frac{1}{3}(\mathbf{Q}_{\mathcal{P}},|\mathbf{Q}_{\mathcal{P}}|\mathbf{Q}_{\mathcal{P}})\\
=&a^2\left\{\frac{1}{2}(Q,DQ)+\frac{1}{2}(S_{1,0},DS_{1,0})
+(T_{2,0},DQ)\right\}\\
+&a^2\left\{-\frac{1}{3}(Q,Q^2)
-(T_{2,0},Q^2)-\frac{1}{2}(S_{1,0}^2,Q)\right\}
+\mathcal{O}(a^4+b^2+b\mathcal{P}^2)\\
=&a^2\left\{\frac{1}{2}(S_{1,0},DS_{1,0})+(T_{2,0},DQ)\right\}\\
&+a^2\left\{-(T_{2,0},Q^2)-\frac{1}{2}(S_{1,0}^2,Q)\right\}
+\mathcal{O}(a^4+b^2+b\mathcal{P}^2).
\end{align*}
Note that the term $\mathcal{O}(ab)$ vanishes in the expansion for  $E(\mathbf{Q}_{\mathcal{P}})$, since $S_{0,1}$ and $S_{1,0}$ are the antisymmetry and radial symmetry functions, respectively, and hence $(S_{1,0},DS_{0,1})=0$. Using $DQ+Q-Q^2=0$ and \eqref{3app-5}, we obtain that \begin{align*}
E(\mathbf{Q}_{\mathcal{P}})
=&a^2\left\{\frac{1}{2}(S_{1,0},DS_{1,0})-(T_{2,0},Q)-\frac{1}{2}(S_{1,0}^2,Q)\right\}
+\mathcal{O}(a^4+b^2+b\mathcal{P}^2)\\
=&a^2\left\{\frac{1}{2}(S_{1,0},DS_{1,0})+\frac{1}{2}(L_{-}S_{1,0},S_{1,0})-(L_{-}S_{1,0},\Lambda S_{1,0})+\frac{1}{2}(\Lambda Q,S_{1,0}^2)\right\}\\
&-\frac{1}{2}(S_{1,0}^2,Q)+\mathcal{O}(a^4+b^2+b\mathcal{P}^2)\\
=&\frac{a^2}{2}(L_{-}S_{1,0},S_{1,0})+a^2(\frac{1}{2}-\frac{1}{2})(S_{1,0}^2,Q)
+\mathcal{O}(a^4+b^2+b\mathcal{P}^2)\\
=&\frac{a^2}{2}(L_{-}S_{1,0},S_{1,0})+\mathcal{O}(a^4+b^2+b\mathcal{P}^2)\\
=&a^2e_1+\mathcal{O}(a^4+b^2+b\mathcal{P}^2).
\end{align*}
The proof of lemma \ref{lemma-3app-2} is now complete.
\end{proof}

\section{Modulation Estimates}\label{section-mod-estimate}
We start with a general observation: If $u=u(t,x)$ solves \eqref{equ-1-hf-2}, then we define the function $v=v(s,y)$ by setting
\begin{align}\notag
u(t,x)=\frac{1}{\lambda(t)}v\left(s,\frac{x-\alpha(t)}{\lambda(t)}\right)
e^{i\gamma(t)},\ \ \frac{ds}{dt}=\frac{1}{\lambda(t)},
\end{align}
here $s=s(t)$. It is easy to check that $v=v(s,y)$ with $y=\frac{x-\alpha}{\lambda}$ satisfies
\begin{equation}\notag
i\partial_sv-Dv-v+|v|v=i\frac{\lambda_s}{\lambda}\Lambda v+i\frac{\alpha_s}{\lambda}\cdot\nabla v+\tilde{\gamma_s}v,
\end{equation}
where we set $\tilde{\gamma_s}=\gamma_s-1$. Here the operators $D$ and $\nabla$ are understood as $D=D_y$ and $\nabla=\nabla_y$, respectively.
\subsection{Geometrical Decomposition and Modulation Equations}
Let $u\in H^{1/2}(\mathbb{R}^2)$ be a solution of \eqref{equ-1-hf-2} on some time interval $[t_0,t_1]$ with $t_1<0$. Assume that $u(t)$ admits a geometrical decomposition of the form
\begin{align}\label{mod-decomposition}
u(t,x)=\frac{1}{\lambda(t)}[Q_{\mathcal{P}(t)}+\epsilon]
\left(s,\frac{x-\alpha(t)}{\lambda(t)}\right)
e^{i\gamma(t)},\ \ \frac{ds}{dt}=\frac{1}{\lambda(t)},
\end{align}
with $\mathcal{P}(t)=(a(t),b(t))$, and we impose the uniform smallness bound
\begin{align}
a^2(t)+|b(t)|+\|\epsilon\|_{H^{1/2}}^2\ll1.
\end{align}
Furthermore, we assume that $u(t)$ has almost critical mass in the sense that
\begin{align}\label{mod-mass-assume}
\left|\int|u(t)|^2-\int Q^2\right|\lesssim\lambda^3(t),\ \ \forall t\in[t_0,t_1].
\end{align}
To fix the modulation parameters $\{a(t),b(t),\lambda(t),\alpha(t),\gamma(t)\}$ uniquely, we impose the following orthogonality conditions on $\epsilon=\epsilon_1+i\epsilon_2$ as follows:
\begin{align}\label{mod-orthogonality-condition}
(\epsilon_2,\Lambda Q_{1\mathcal{P}})-(\epsilon_1,\Lambda Q_{2\mathcal{P}})=0,\notag\\
(\epsilon_2,\partial_aQ_{1\mathcal{P}})-(\epsilon_1,\partial_aQ_{2\mathcal{P}})=0,\notag\\
(\epsilon_2,\partial_{b_j}Q_{1\mathcal{P}})-(\epsilon_1,\partial_{b_j}Q_{2\mathcal{P}})=0,\notag\\
(\epsilon_2,\partial_{x_j} Q_{1\mathcal{P}})-(\epsilon_1,\partial_{x_j} Q_{2\mathcal{P}})=0,\notag\\
(\epsilon_2,\rho_1)+(\epsilon_1,\rho_2)=0,
\end{align}
 the function $\rho=\rho_1+i\rho_2$ is defined by
\begin{align}\label{mod-definition-rho}
L_{+}\rho_1=S_{1,0}, L_{-}\rho_2=aS_{1,0}\rho_1+a\Lambda\rho_1-2aT_{2,0}
+b\cdot S_{0,1}\rho_1-b\cdot\nabla\rho_1-b\cdot T_{1,1},
\end{align}
where $S_{1,0}$, $T_{2,0}$ and $T_{1,1,j}$ are the functions introduced in the proof of lemma \ref{lemma-3app}. Note that $L_{+}^{-1}$ exists on $L^2_{rad}(\mathbb{R}^2)$ and thus $\rho_1$ is well-defined. Moreover, it is easy to see that the right-hand side in the equation for $\rho_2$ is orthogonality to $Q$. Indeed
\begin{align*}
(Q,S_{1,0}\rho_1+\Lambda\rho_1-2T_{2,0})&=
(QS_{1,0},\rho_1)-(\Lambda Q,\rho_1)-2(Q,T_{2,0})\\
&=(QS_{1,0},\rho_1)-(S_{1,0},L_{-}\rho_1)+(S_{1,0},S_{1,0})\\
&=-(S_{1,0},L_{+}\rho_1)+(S_{1,0},S_{1,0})=0,
\end{align*}
using that $(S_{1,0},S_{1,0})=-2(T_{2,0},Q)$, see \eqref{3app-relation}, and the definition of $\rho_1$. Moreover, we clearly see that $S_{0,1,j}\rho_1-\partial_{x_j}\rho_1-T_{1,1,j}\bot Q$, since $S_{0,1,j}$ and $T_{1,1,j}$ are the antisymmetry functions, whereas $\rho_1$ and $Q$ are radial symmetry functions. Hence $\rho_2$ is well-defined.

In the conditions \eqref{mod-orthogonality-condition}, we use the notation
\begin{align}\notag
Q_{\mathcal{P}}=Q_{1\mathcal{P}}+iQ_{2\mathcal{P}},
\end{align}
which (in terms of the vector notation used in Section 3) means that
\begin{align}\notag
\mathbf{Q}_{\mathcal{P}}={\left[\begin{array}{c}
Q_{1\mathcal{P}}\\Q_{2\mathcal{P}}
\end{array}
\right]}.
\end{align}

We refer to Appendix \ref{section-app-mod-1} for some standard arguments, which show that the orthogonality condition \eqref{mod-orthogonality-condition} imply that the modulation parameters $\{a(t),b(t),\lambda(t),\alpha(t),\gamma(t)\}$ are uniquely determined, provided that $\epsilon=\epsilon_1+i\epsilon_2\in H^{1/2}(\mathbb{R}^2)$ is sufficiently small. Moreover, it follows from the standard arguments that $\{a(t),b(t),\lambda(t),\alpha(t),\gamma(t)\}$ are $C^1$-functions.

If we insert the decomposition \eqref{mod-decomposition} into \eqref{equ-1-hf-2}, we obtain the following system
\begin{align}\label{mod-system-1}
&(a_s+\frac{1}{2}a^2)\partial_aQ_{1\mathcal{P}}+\sum_{j=1}^2((b_j)s+ab_j)\partial_{b_j}Q_{1\mathcal{P}}
+\partial_s\epsilon_1-M_{-}(\epsilon)+a\Lambda\epsilon_1-b\cdot\nabla\epsilon_1\notag\\
=&(\frac{\lambda_s}{\lambda}+a)(\Lambda Q_{1\mathcal{P}}+\Lambda\epsilon_1)
+(\frac{\alpha_s}{\lambda}-b)\cdot(\nabla Q_{1\mathcal{P}}+\nabla\epsilon_1)\notag\\
&+\tilde{\gamma}_s(Q_{2\mathcal{P}}+\epsilon_2)+\Im(\Phi_{\mathcal{P}})
-R_2(\epsilon),\\\label{mod-system-2}
&(a_s+\frac{1}{2}a^2)\partial_aQ_{2\mathcal{P}}+\sum_{j=1}^2((b_j)s+ab_j)\partial_{b_j}Q_{2\mathcal{P}}
+\partial_s\epsilon_2+M_{+}(\epsilon)+a\Lambda\epsilon_2-b\cdot\nabla\epsilon_2\notag\\
=&(\frac{\lambda_s}{\lambda}+a)(\Lambda Q_{2\mathcal{P}}+\Lambda\epsilon_2)
+(\frac{\alpha_s}{\lambda}-b)\cdot(\nabla Q_{2\mathcal{P}}+\nabla\epsilon_2)\notag\\
&-\tilde{\gamma}_s(Q_{1\mathcal{P}}-\epsilon_1)-\Re(\Phi_{\mathcal{P}})+R_1(\epsilon).
\end{align}
Here $\Phi_{\mathcal{P}}$ denotes the error term from lemma \ref{lemma-3app}, and $M=(M_{+},M_{-})$ are the small deformations of the linearized operator $L=(L_{+},L_{-})$ given by
\begin{align}\label{mod-define-M1}
M_{+}(\epsilon)=&D\epsilon_1+\epsilon_1
-\frac{3}{2}|Q_{\mathcal{P}}|\epsilon_1
-\frac{1}{2}|Q_{\mathcal{P}}|^{-1}(Q_{1\mathcal{P}}^2-Q_{2\mathcal{P}}^2)\epsilon_1\notag\\
&-|Q_{\mathcal{P}}|^{-1}Q_{1\mathcal{P}}Q_{2\mathcal{P}}\epsilon_2,\\\label{mod-define-M2}
M_{-}(\epsilon)=&D\epsilon_2+\epsilon_2
-\frac{3}{2}|Q_{\mathcal{P}}|\epsilon_2
-\frac{1}{2}|Q_{\mathcal{P}}|^{-1}(Q_{1\mathcal{P}}^2-Q_{2\mathcal{P}}^2)\epsilon_2\notag\\
&-|Q_{\mathcal{P}}|^{-1}Q_{1\mathcal{P}}Q_{2\mathcal{P}}\epsilon_1.
\end{align}
And $R_1(\epsilon)$, $R_2(\epsilon)$ are the high order terms about $\epsilon$.
\begin{align*}
R_1(\epsilon)=&\frac{5}{4}|Q_{\mathcal{P}}|^{-1}Q_{1\mathcal{P}}|\epsilon|^2
+\frac{3}{8}|Q_{\mathcal{P}}|^{-1}\left(Q_{1\mathcal{P}}(\epsilon_1^2-\epsilon_2^2)
+2Q_{2\mathcal{P}}\epsilon_1\epsilon_2\right)\\
&+\frac{1}{8}|Q_{\mathcal{P}}|^{-3}\left((Q_{1\mathcal{P}}^3-3Q_{1\mathcal{P}}Q_{2\mathcal{P}}^2)(\epsilon_1-\epsilon_2)
+2(3Q_{1\mathcal{P}}^2Q_{2\mathcal{P}}-Q_{2\mathcal{P}}^3)\epsilon_1\epsilon_2\right)+\mathcal{O}(\epsilon^3),\\
R_2(\epsilon)=&\frac{5}{4}|Q_{\mathcal{P}}|^{-1}Q_{2\mathcal{P}}|\epsilon|^2
+\frac{3}{8}|Q_{\mathcal{P}}|^{-1}\left(2Q_{1\mathcal{P}}\epsilon_1\epsilon_2
+Q_{2\mathcal{P}}(\epsilon_1^2-\epsilon_2^2)\right)\\
&+\frac{1}{8}|Q_{\mathcal{P}}|^{-3}\left(2(Q_{1\mathcal{P}}^3-3Q_{1\mathcal{P}}Q_{2\mathcal{P}}^2)
\epsilon_2\epsilon_2+(3Q_{1\mathcal{P}}^2Q_{2\mathcal{P}}-Q_{2\mathcal{P}}^3)
(\epsilon_1^2-\epsilon_2^2)\right)+\mathcal{O}(\epsilon^3).
\end{align*}

 We have the following energy type bound.
\begin{lemma}\label{lemma-mod-1}
For $t\in[t_0,t_1]$ with $t_1<0$, it holds that
\begin{align}\notag
a^2+|b|+\|\epsilon\|_{H^{1/2}}^2\lesssim\lambda(|E_0|+|P_0|)
+\mathcal{O}(\lambda^3+a^4+b^2+b\mathcal{P}^2).
\end{align}
Here $E_0=E(u_0)$ and $P_0=P(u_0)$ denote the conserved energy and linear momentum of $u=u(t,x)$, respectively.
\end{lemma}
\begin{proof}
By the conservation of $L^2$-mass and lemma \ref{lemma-3app-2}, we find that
\begin{align}\notag
\int|u|^2=\int|Q_{\mathcal{P}}+\epsilon|^2=\int|Q|^2+2\Re(\epsilon,Q_{\mathcal{P}})
+\int|\epsilon|^2+\mathcal{O}(a^4+b^2+b\mathcal{P}^2).
\end{align}
By assumption \eqref{mod-mass-assume}, this implies
\begin{align}\label{mod-mass}
2\Re(\epsilon,Q_{\mathcal{P}})+\int|\epsilon|^2
=\mathcal{O}(\lambda^3+a^4+b^2+b\mathcal{P}^2).
\end{align}
Next, we recall that $v=Q_{\mathcal{P}}+\epsilon$ and the assumed form of $u=u(t,x)$. Hence, by energy conservation and scaling, we obtain
\begin{align}\label{mod-energy}
E(v)=\lambda E(u_0).
\end{align}
On the other hand, from lemma \ref{lemma-3app-2} and by expanding the energy functional
\begin{align*}
E(v)=&E(Q_{\mathcal{P}}+\epsilon)\\
=&\frac{1}{2}(Q_{\mathcal{P}}+\epsilon,D(Q_{\mathcal{P}}+\epsilon))
-\frac{1}{3}(Q_{\mathcal{P}}+\epsilon,
|Q_{\mathcal{P}}+\epsilon|(Q_{\mathcal{P}}+\epsilon))\\
=&\frac{1}{2}\left\{(Q_{\mathcal{P}},DQ_{\mathcal{P}})
+(\epsilon,D\epsilon)\right\}
+\Re(\epsilon,DQ_{\mathcal{P}})\\
&-\frac{1}{3}\left\{(Q_{\mathcal{P}},|Q_{\mathcal{P}}+\epsilon|Q_{\mathcal{P}})
+2(\epsilon,|Q_{\mathcal{P}}+\epsilon|Q_{\mathcal{P}})
+(\epsilon,|Q_{\mathcal{P}}+\epsilon|\epsilon)\right\}\\
=&E(Q_{\mathcal{P}})+\frac{1}{2}(\epsilon,D\epsilon)
+\Re(\epsilon,DQ_{\mathcal{P}}-|Q_{\mathcal{P}}|Q_{\mathcal{P}})\\
&-\frac{1}{2}\int|Q_{\mathcal{P}}||\epsilon|^2
+\frac{1}{8}\int|Q_{\mathcal{P}}|^{-1}(2Q_{1\mathcal{P}}\epsilon_1+2Q_{2\mathcal{P}}\epsilon_2)^2\\
&+\mathcal{O}(\|\epsilon\|_{H^{1/2}}^3+\mathcal{P}^2\|\epsilon\|_{H^{1/2}}^2).
\end{align*}
Combining the above equality and \eqref{mod-mass}, \eqref{mod-energy} we find that
\begin{align*}
\lambda E_0=&a^2e_1+\Re(\epsilon,DQ_{\mathcal{P}}
+Q_{\mathcal{P}}-|Q_{\mathcal{P}}|Q_{\mathcal{P}})
+\frac{1}{2}\{M_{+}\epsilon_1+M_{-}\epsilon_2\}\\
&+\mathcal{O}(\|\epsilon\|_{H^{1/2}}^3+\mathcal{P}^2\|\epsilon\|_{H^{1/2}}^2
+a^4+b^2+b\mathcal{P}^2),
\end{align*}
where $e_1=\frac{1}{2}(L_{-}S_{1,0},S_{1,0})>0$.
In the previous equation, we note that the linear term in $\epsilon=\epsilon_1+i\epsilon_2$ satisfies
\begin{align*}
&\Re(\epsilon,DQ_{\mathcal{P}}
+Q_{\mathcal{P}}-|Q_{\mathcal{P}}|Q_{\mathcal{P}})\\
=&\Im(\epsilon,\frac{a^2}{2}\partial_aQ_{\mathcal{P}}+a\sum_{j=1}^2b_j\partial_{b_j}Q_{\mathcal{P}}
-a\Lambda Q_{\mathcal{P}}+b\cdot\nabla Q_{\mathcal{P}})+\mathcal{O}(\epsilon(a^4+b^2+b\mathcal{P}^2))\\
=&\mathcal{O}(a^4+b^2+b\mathcal{P}^2),
\end{align*}
thanks to the orthogonality condition \eqref{mod-orthogonality-condition}. Next, we observe that quadratic form $M=(M_{+},M_{-})$ is a small deformation of the quadratic form given by the linearization $L=(L_{+},L_{-})$ around $Q$. Hence, we deduce
\begin{align}\label{mod-1}
&a^2e_1+\frac{1}{2}\left\{(L_{+}\epsilon_1,\epsilon_1)
+(L_{-}\epsilon_2,\epsilon_2)\right\}\notag\\
=&\lambda E_0+\mathcal{O}(\|\epsilon\|_{H^{1/2}}^3+a^4+b^2+b\mathcal{P}^2)
+o(\|\epsilon\|_{H^{1/2}}^2).
\end{align}
Next, we recall from lemma \ref{lemma-app-coercivity-estimate} the coercivity estimate
\begin{align}\label{mod-coercivity-estimate}
&(L_{+}\epsilon_1,\epsilon_1)+(L_{-}\epsilon_2,\epsilon_2)\notag\\ \geq & c_0\|\epsilon\|_{H^{1/2}}^2-\frac{1}{c_0}\left\{(\epsilon_1,Q)^2+(\epsilon_1,S_{1,0})^2
+|(\epsilon_1,S_{0,1})|^2+|(\epsilon_2,\rho_1)|^2\right\},
\end{align}
with some universal constant $c_0>0$. Note that the orthogonality condition \eqref{mod-orthogonality-condition} imply that
\begin{align*}
(\epsilon_1,S_{1,0})^2=\mathcal{O}(\mathcal{P}\|\epsilon\|_2^2),\ |(\epsilon_1,S_{0,1})|^2=\mathcal{O}(\mathcal{P}\|\epsilon\|_2^2),\
(\epsilon_2,\rho_1)^2=\mathcal{O}(\mathcal{P}\|\epsilon\|_2^2).
\end{align*}
Furthermore, from the relation \eqref{mod-mass} we deduce that
\begin{align}\notag
|(\epsilon_1,Q)|^2=o(\|\epsilon\|_2^2)+\mathcal{O}(\lambda^3+a^4+b^2+b\mathcal{P}^2).
\end{align}
Combining these bounds with \eqref{mod-coercivity-estimate} and the universal smallness assumption for $\mathcal{P}$ and $\|\epsilon\|_{H^{1/2}}$, we obtain that
\begin{align}\notag
(L_{+}\epsilon_1,\epsilon_1)+(L_{-}\epsilon_2,\epsilon_2)
\geq\frac{c_0}{2}\|\epsilon\|_{H^{1/2}}^2+\mathcal{O}(a^4+b^2+b\mathcal{P}^2).
\end{align}
Inserting this bound into \eqref{mod-1} and recall that $e_1=\frac{1}{2}(L_{-}S_{1,0},S_{1,0})>0$ holds, we have
\begin{align}\label{mod-2}
a^2+\|\epsilon\|_{H^{1/2}}^2\lesssim\lambda E_0+\mathcal{O}(\lambda^3+a^4+b^2+b\mathcal{P}^2).
\end{align}

As our final step, we derive the bound for the boost parameter $b$. Here we observe that
\begin{align}\notag
P(v)=\lambda P(u_0),
\end{align}
by scaling and using the conservation of the linear momentum $P(u(t))=P(u_0)$. Hence, by expansion and lemma \ref{lemma-3app-2} and using the orthogonality condition \eqref{mod-orthogonality-condition}, we obtain
\begin{align*}
\lambda P_0=P(v)&=P(Q_{\mathcal{P}})+2\Re(\epsilon,-i\nabla Q_{\mathcal{P}})+\Re(\epsilon,-i\nabla\epsilon)\\
&=p_1b+\mathcal{O}(a^4+b^2+b\mathcal{P}^2+\|\epsilon\|_{H^{1/2}}^2),
\end{align*}
with the constant $p_1=2\int_{\mathbb{R}^2}(L_{-}S_{0,1}\cdot S_{0,1})>0$. Recalling that \eqref{mod-2} holds, we complete the proof of this lemma.
\end{proof}

\subsection{Modulation Estimates}
We continue with estimating the modulation parameters. To this end, we define the vector-valued function
\begin{align}\label{define-mod}
\mathbf{Mod}(t):=(a_s+\frac{1}{2}a^2,\tilde{\gamma}_s,\frac{\lambda_s}{\lambda}+a,
\frac{\alpha_s}{\lambda}-b,b_s+ab).
\end{align}
We have the following result.
\begin{lemma}\label{lemma-mod-2}
For $t\in[t_0,t_1]$ with $t_1<0$, we have the bound
\begin{align}
|\mathbf{Mod}(t)|\lesssim\lambda^3+a^4+b^2+b\mathcal{P}^2+\mathcal{P}^2\|\epsilon\|_2+
\|\epsilon\|_2^2+\|\epsilon\|_{H^{1/2}}^3.
\end{align}
Furthermore, we have the improved bound
\begin{align}\notag
\left|\frac{\lambda_s}{\lambda}+a\right|\lesssim a^5 +b\mathcal{P}^2+\mathcal{P}^2\|\epsilon\|_2
+\|\epsilon\|_2^2+\|\epsilon\|_{H^{1/2}}^3.
\end{align}
\end{lemma}
\begin{proof}
We divide the proof into the following six steps, where we also make use of the estimate \eqref{app-B-1}-\eqref{app-B-5}, which are shown in lemma \ref{lemma-app-mod-estimate}. Now, we recall that
\begin{align*}\notag
\Lambda Q_{1\mathcal{P}}=\Lambda Q+\mathcal{O}(\mathcal{P}^2),\ \Lambda Q_{2\mathcal{P}}=a\Lambda S_{1,0}+b\cdot\Lambda S_{0,1}+\mathcal{O}(\mathcal{P}^2),\\
\partial_aQ_{1\mathcal{P}}=2aT_{2,0}+b\cdot T_{1,1}+\mathcal{O}(\mathcal{P}^2),\ \partial_aQ_{2\mathcal{P}}=S_{1,0}+\mathcal{O}(\mathcal{P}^2),\\
\partial_{j} Q_{1\mathcal{P}}=\partial_{j} Q+\mathcal{O}(\mathcal{P}^2),\
\partial_{j} Q_{2\mathcal{P}}=a\partial_{j} S_{1,0}+\partial_{j}( b\cdot S_{0,1})+\mathcal{O}(\mathcal{P}^2),\\
\partial_{b_j}Q_{1\mathcal{P}}=aT_{1,1,j}+2b_jT_{0,2,j}+\mathcal{O}(\mathcal{P}^2),\ \partial_{b_j}Q_{2\mathcal{P}}=S_{0,1,j}+\mathcal{O}(\mathcal{P}^2),
\end{align*}
where $j=1,2$.

$\mathbf{Step ~1~ Law~ for}$ $a$. We multiply both sides of the equation \eqref{mod-system-1} and \eqref{mod-system-2} by $-\Lambda Q_{2\mathcal{P}}$ and $\Lambda Q_{1\mathcal{P}}$, respectively. Adding this and using \eqref{app-B-1} yields, after some calculation (also using the condition \eqref{mod-orthogonality-condition}), we have
\begin{align*}
&\left(a_s+\frac{1}{2}a^2\right)[(\partial_aQ_{1\mathcal{P}},-\Lambda Q_{2\mathcal{P}})+
(\partial_aQ_{2\mathcal{P}},\Lambda Q_{1\mathcal{P}})]
+\sum_{j=1}^2((b_j)_s+ab_j)\big[(\partial_{b_j}Q_{1\mathcal{P}},-\Lambda Q_{2\mathcal{P}})\\
&+(\partial_{b_j}Q_{2\mathcal{P}},\Lambda Q_{1\mathcal{P}})\big]
+[(\partial_s\epsilon_1,-\Lambda Q_{2\mathcal{P}})+(\partial_s\epsilon_2,\Lambda Q_{1\mathcal{P}})]-\Re(\epsilon,Q_{\mathcal{P}})\\
=&\left(\frac{\lambda_s}{\lambda}+a\right)[(\Lambda Q_{1\mathcal{P}}+\Lambda\epsilon_1,-\Lambda Q_{2\mathcal{P}})+
(\Lambda Q_{2\mathcal{P}}+\Lambda\epsilon_2,\Lambda Q_{1\mathcal{P}})]\\
&+\sum_{j=1}^2\left(\frac{(\alpha_j)_s}{\lambda}-b_j\right)[(\partial_{j} Q_{1\mathcal{P}}+\partial_{j}\epsilon_1,-\Lambda Q_{2\mathcal{P}})+(\partial_{j} Q_{2\mathcal{P}}+\partial_{j}\epsilon_2,\Lambda Q_{1\mathcal{P}})]+\tilde{\gamma}_s[(Q_{2\mathcal{P}}+\epsilon_2,-\Lambda Q_{2\mathcal{P}})\\
&-(Q_{1\mathcal{P}}+\epsilon_1,\Lambda Q_{1\mathcal{P}})]
+(R_2(\epsilon),\Lambda Q_{2\mathcal{P}})+(R_1(\epsilon),\Lambda Q_{1\mathcal{P}})-(\Im(\Phi_{\mathcal{P}}),\Lambda Q_{2\mathcal{P}})
+(\Re(\Phi_{\mathcal{P}}),\Lambda Q_{1\mathcal{P}})\\
&+\mathcal{O}(\mathcal{P}^2\|\epsilon\|_2).
\end{align*}
Hence, we obtain that
\begin{align*}
&-\left(a_s+\frac{1}{2}a^2\right)[(L_{-}S_{1,0},S_{1,0})+\mathcal{O}(\mathcal{P}^2)]
+\sum_{j=1}^2\left(\frac{(\alpha_j)_s}{\lambda}-b_j\right)[-(\partial_{j} Q,S_{0,1,j})+\mathcal{O}(\mathcal{P}^2)]\\
=&\Re(\epsilon,Q_{\mathcal{P}})+(R_2(\epsilon),\Lambda Q_{2\mathcal{P}})+(R_1(\epsilon),\Lambda Q_{1\mathcal{P}})-(\Im(\Phi_{\mathcal{P}}),\Lambda Q_{2\mathcal{P}})
+(\Re(\Phi_{\mathcal{P}}),\Lambda Q_{1\mathcal{P}})\\
&+\mathcal{O}\left((\mathcal{P}^2+|\mathbf{mod}(t)|)(\|\epsilon\|_2+\mathcal{P}^2)\right).
\end{align*}
Here we also used that
\begin{align}\notag
(\partial_{b_j}Q_{1\mathcal{P}},-\Lambda Q_{2\mathcal{P}})
+(\partial_{b_j}Q_{2\mathcal{P}},\Lambda Q_{1\mathcal{P}})=
(S_{0,1,j},\Lambda Q)+\mathcal{O}(\mathcal{P}^2),\,j=1,2,
\end{align}
since $S_{0,1,j}=-L_{-}^{-1}\partial_{x_j} Q$ is antisymmetry and $\Lambda Q$ is radial symmetry, and hence $(S_{0,1,j},\Lambda Q)=0$. Next, from the lemma \ref{lemma-3app-2} the constants
\begin{align}\notag
e_1=\frac{1}{2}(L_{-}S_{1,0},S_{1,0})>0,\ \ p_{1,j}=2(L_{-}S_{0,1,j},S_{0,1,j})>0.
\end{align}
and
\begin{align}\notag
2\Re(\epsilon,Q_{\mathcal{P}})=-\int|\epsilon|^2
+\left(\int|u|^2-\int|Q|^2\right)+\mathcal{O}(a^4+b^2+b\mathcal{P}^2).
\end{align}
We deduce that
\begin{align*}
&-\left(a_s+\frac{1}{2}a^2\right)[2e_1+\mathcal{O}(\mathcal{P}^2)]
+\sum_{j=1}^2\left(\frac{(\alpha_j)_s}{\lambda}-b_j\right)\left[\frac{1}{2}p_{1,j}
+\mathcal{O}(\mathcal{P}^2)\right]\\
=&-\int|\epsilon|^2+(R_2(\epsilon),\Lambda Q_{2\mathcal{P}})+(R_1(\epsilon),\Lambda Q_{1\mathcal{P}})\\
&+\mathcal{O}\left((\mathcal{P}^2+|\mathbf{mod}(t)|)
(\|\epsilon\|_2+\mathcal{P}^2)+|\|u\|_2^2-\|Q\|_2^2|+a^4+b^2+b\mathcal{P}^2\right).
\end{align*}

$\mathbf{Step ~2~ Law~ for}$ $\lambda$. We multiply both sides of the equation \eqref{mod-system-1} and \eqref{mod-system-2} by $-\partial_a Q_{2\mathcal{P}}$ and $\partial_a Q_{1\mathcal{P}}$, respectively. After some calculation (also using the condition \eqref{mod-orthogonality-condition}), we have
\begin{align*}
&\left(a_s+\frac{1}{2}a^2\right)[(\partial_aQ_{1\mathcal{P}},-\partial_a Q_{2\mathcal{P}})+
(\partial_aQ_{2\mathcal{P}},\partial_a Q_{1\mathcal{P}})]
+\sum_{j=1}^2((b_j)_s+ab_j)[(\partial_{b_j}Q_{1\mathcal{P}},-\partial_a Q_{2\mathcal{P}})\\
&+(\partial_{b_j}Q_{2\mathcal{P}},\partial_a Q_{1\mathcal{P}})]
+[(\partial_s\epsilon_1,-\partial_a Q_{2\mathcal{P}})+(\partial_s\epsilon_2,\partial_a Q_{1\mathcal{P}})]\\
=&\left(\frac{\lambda_s}{\lambda}+a\right)[(\Lambda Q_{1\mathcal{P}}+\Lambda\epsilon_1,-\partial_a Q_{2\mathcal{P}})+
(\Lambda Q_{2\mathcal{P}}+\Lambda\epsilon_2,\partial_a Q_{1\mathcal{P}})]\\
&+\left(\frac{\alpha_s}{\lambda}-b\right)\cdot[(\nabla Q_{1\mathcal{P}}+\nabla\epsilon_1,-\partial_a Q_{2\mathcal{P}})+(\nabla Q_{2\mathcal{P}}+\nabla\epsilon_2,\partial_a Q_{1\mathcal{P}})]+\tilde{\gamma}_s[(Q_{2\mathcal{P}}+\epsilon_2,-\partial_a Q_{2\mathcal{P}})\\
&-(Q_{1\mathcal{P}}+\epsilon_1,\partial_a Q_{1\mathcal{P}})]
+(R_2(\epsilon),\partial_a Q_{2\mathcal{P}})+(R_1(\epsilon),\partial_a Q_{1\mathcal{P}})-(\Im(\Phi_{\mathcal{P}}),\partial_a Q_{2\mathcal{P}})\\
&+(\Re(\Phi_{\mathcal{P}}),\partial_a Q_{1\mathcal{P}})
+\mathcal{O}(\mathcal{P}^2\|\epsilon\|_2).
\end{align*}
Hence, we deduce
\begin{align*}
&\sum_{j=1}^2((b_{j})_s+ab_j)\left[(aT_{1,1,j}+2b_jT_{0,2,j},-S_{1,0})
+(S_{0,1,j},2aT_{2,0}+b_jT_{1,1,j})+\mathcal{O}(\mathcal{P}^2)\right]\\
=&\left(\frac{\lambda_s}{\lambda}+a\right)[(\Lambda Q,-S_{1,0})+\mathcal{O}(\mathcal{P}^2)]
+\left(\frac{\alpha_s}{\lambda}-b\right)\cdot[(\nabla Q,-S_{1,0})+\mathcal{O}(\mathcal{P}^2)]\\
&+(R_2(\epsilon),\partial_a Q_{2\mathcal{P}})+(R_1(\epsilon),\partial_a Q_{1\mathcal{P}})+\mathcal{O}\left((\mathcal{P}^2
+|\mathbf{mod}(t)|)(\|\epsilon\|_2+\mathcal{P}^2)\right).
\end{align*}
Therefore, we obtain
\begin{align*}
&\left(\frac{\lambda_s}{\lambda}+a\right)[2e_1+\mathcal{O}(\mathcal{P}^2)]+(b_s+ab)\mathcal{O}(\mathcal{P})\\
=&(R_2(\epsilon),\partial_a Q_{2\mathcal{P}})+(R_1(\epsilon),\partial_a Q_{1\mathcal{P}})+\mathcal{O}\left((\mathcal{P}^2
+|\mathbf{mod}(t)|)(\|\epsilon\|_2+\mathcal{P}^2)+a^5+b^2\mathcal{P}\right).
\end{align*}
Here we used that $(\nabla Q,-S_{1,0})=0$.

Furthermore, by this estimate, we deduce the improved bound for $\left|\frac{\lambda_s}{\lambda}+a\right|$.

$\mathbf{Step ~3~ Law~ for}$ $\tilde{\gamma}_s$. We multiply both sides of the equation \eqref{mod-system-1} and \eqref{mod-system-2} by $-\rho_2$ and $\rho_1$, respectively. Adding this and using \eqref{app-B-3} yields, after some calculation, we have
\begin{align*}
&\left(a_s+\frac{1}{2}a^2\right)[(\partial_aQ_{1\mathcal{P}},-\rho_2)+
(\partial_aQ_{2\mathcal{P}},\rho_1)]
+\sum_{j=1}^2((b_j)_s+ab_j)[(\partial_{b_j}Q_{1\mathcal{P}},-\rho_2)\\
&+(\partial_{b_j}Q_{2\mathcal{P}},\rho_1)]
+[(\partial_s\epsilon_1,-\rho_2)+(\partial_s\epsilon_2,\rho_1)]\\
=&\left(\frac{\lambda_s}{\lambda}+a\right)[(\Lambda Q_{1\mathcal{P}}+\Lambda\epsilon_1,-\rho_2)+
(\Lambda Q_{2\mathcal{P}}+\Lambda\epsilon_2,\rho_1)]\\
&+\left(\frac{\alpha_s}{\lambda}-b\right)\cdot[(\nabla Q_{1\mathcal{P}}+\nabla\epsilon_1,-\rho_2)+(\nabla Q_{2\mathcal{P}}+\nabla\epsilon_2,\rho_1)]
+\tilde{\gamma}_s[(Q_{2\mathcal{P}}+\epsilon_2,-\rho_2)\\
&-(Q_{1\mathcal{P}}+\epsilon_1,\rho_1)]
+(R_2(\epsilon),-\rho_2)+(R_1(\epsilon),\rho_1)
-(\Im(\Phi_{\mathcal{P}}), \rho_2)\\
&+(\Re(\Phi_{\mathcal{P}}),\rho_1)
+\mathcal{O}(\mathcal{P}^2\|\epsilon\|_2).
\end{align*}
Hence, we have
\begin{align*}
&\tilde{\gamma}_s((Q,\rho_1)+\mathcal{O}(\mathcal{P}^2))\\
=&-\left(a_s+\frac{1}{2}a^2\right)
\left((S_{1,0},\rho_1)+\mathcal{O}(\mathcal{P}^2)\right)
-(b_s+ab)\cdot[(S_{0,1},\rho_1)+\mathcal{O}(\mathcal{P}^2)]\\
&+\left(\frac{\lambda_s}{\lambda}+a\right)[(\Lambda Q,\rho_2)+(a\Lambda S_{1,0}+b\cdot\Lambda S_{0,1}, \rho_1)+\mathcal{O}(\mathcal{P}^2)]
+\left(\frac{\alpha_s}{\lambda}-b\right)\cdot[(\nabla Q,\rho_2)\\
&+(a\nabla S_{1,0}+\nabla (b\cdot S_{0,1}),\rho_1)+\mathcal{O}(\mathcal{P}^2)]
+(R_2(\epsilon),\rho_2)+(R_1(\epsilon),\rho_1)
-(\Im(\Phi_{\mathcal{P}}), \rho_2)\\
&+(\Re(\Phi_{\mathcal{P}}),\rho_1)
+\mathcal{O}(\mathcal{P}^2\|\epsilon\|_2)\\
=&-\left(a_s+\frac{1}{2}a^2\right)
\left((S_{1,0},\rho_1)+\mathcal{O}(\mathcal{P}^2)\right)
+\left(\frac{\lambda_s}{\lambda}+a\right)\mathcal{O}(\mathcal{P})
+\left(\frac{\alpha_s}{\lambda}-b\right)\cdot\mathcal{O}(\mathcal{P})\\
&+(R_2(\epsilon),\rho_2)+(R_1(\epsilon),\rho_1)+\mathcal{O}\left((\mathcal{P}^2
+|\mathbf{mod}(t)|)(\|\epsilon\|_2+\mathcal{P}^2)+a^5+b^2\mathcal{P}\right).
\end{align*}
Here we used that $(S_{0,1},\rho_1)=0$ and the definition of $\rho$.

$\mathbf{Step ~4~ Law~ for}$ $b_j$. We multiply both sides of the equation \eqref{mod-system-1} and \eqref{mod-system-2} by $-\partial_{j} Q_{2\mathcal{P}}$ and $\partial_{j} Q_{1\mathcal{P}}$, respectively. Adding this and using \eqref{app-B-4} yields, after some calculation (also using the condition \eqref{mod-orthogonality-condition}), we have
\begin{align*}
&\left(a_s+\frac{1}{2}a^2\right)[(\partial_aQ_{1\mathcal{P}},-\partial_{j} Q_{2\mathcal{P}})+
(\partial_aQ_{2\mathcal{P}},\partial_{j} Q_{1\mathcal{P}})]
+\sum_{j=1}^2((b_j)_s+ab_j)[(\partial_{b_j}Q_{1\mathcal{P}},-\partial_{j} Q_{2\mathcal{P}})\\
&+(\partial_{b_j}Q_{2\mathcal{P}},\partial_{j} Q_{1\mathcal{P}})]
+[(\partial_s\epsilon_1,-\partial_{j} Q_{2\mathcal{P}})+(\partial_s\epsilon_2,\partial_{j} Q_{1\mathcal{P}})]\\
=&\left(\frac{\lambda_s}{\lambda}+a\right)[(\Lambda Q_{1\mathcal{P}}+\Lambda\epsilon_1,-\partial_{j} Q_{2\mathcal{P}})+
(\Lambda Q_{2\mathcal{P}}+\Lambda\epsilon_2,\partial_{j} Q_{1\mathcal{P}})]\\
&+\left[\left(\left(\frac{\alpha_s}{\lambda}-b\right)\cdot(\nabla Q_{1\mathcal{P}}+\nabla\epsilon_1),-\partial_{j} Q_{2\mathcal{P}}\right)+\left(\left(\frac{\alpha_s}{\lambda}-b\right)\cdot(\nabla Q_{2\mathcal{P}}+\nabla\epsilon_2),\partial_{j} Q_{1\mathcal{P}}\right)\right]\\
&+\tilde{\gamma}_s[(Q_{2\mathcal{P}}+\epsilon_2,-\partial_{j} Q_{2\mathcal{P}})
-(Q_{1\mathcal{P}}+\epsilon_1,\partial_{j} Q_{1\mathcal{P}})]
+(R_2(\epsilon),\partial_{j} Q_{2\mathcal{P}})+(R_1(\epsilon),\partial_{j} Q_{1\mathcal{P}})\\
&-(\Im(\Phi_{\mathcal{P}}), \partial_{j} Q_{2\mathcal{P}})
+(\Re(\Phi_{\mathcal{P}}),\partial_{j} Q_{1\mathcal{P}})
+\mathcal{O}(\mathcal{P}^2\|\epsilon\|_2).
\end{align*}
Hence, we have
\begin{align*}
&\left(a_s+\frac{1}{2}a^2\right)[(S_{1,0},\partial_{j} Q)+\mathcal{O}(\mathcal{P}^2)]
+[(b_s+ab)\cdot(S_{0,1},\partial_{j} Q)+\mathcal{O}(\mathcal{P}^2)]\\
=&(R_2(\epsilon),\partial_{j} Q_{2\mathcal{P}})+(R_1(\epsilon),\partial_{j} Q_{1\mathcal{P}})
+\mathcal{O}\left((\mathcal{P}^2
+|\mathbf{mod}(t)|)\|\epsilon\|_2+a^4+b^2+b\mathcal{P}^2\right).
\end{align*}
Therefore, we deduce that
\begin{align*}
&(b_s+ab)\left[-\frac{1}{2}p_1+\mathcal{O}(\mathcal{P}^2)\right]\\
=&(R_2(\epsilon),\nabla Q_{2\mathcal{P}})+(R_1(\epsilon),\nabla Q_{1\mathcal{P}})
+\mathcal{O}\left((\mathcal{P}^2
+|\mathbf{mod}(t)|)\|\epsilon\|_2+a^4+b^2+b\mathcal{P}^2\right).
\end{align*}

$\mathbf{Step ~5~ Law~ for}$ $\alpha_j$. We multiply both sides of the equation \eqref{mod-system-1} and \eqref{mod-system-2} by $-\partial_{b_j} Q_{2\mathcal{P}}$ and $\partial_{b_j} Q_{1\mathcal{P}}$, where $j=1,2$ respectively. Adding this and using \eqref{app-B-5} yields, after some calculation (also using the condition \eqref{mod-orthogonality-condition}), we have
\begin{align*}
&\left(a_s+\frac{1}{2}a^2\right)[(\partial_aQ_{1\mathcal{P}},-\partial_{b_j} Q_{2\mathcal{P}})+
(\partial_aQ_{2\mathcal{P}},\partial_{b_j} Q_{1\mathcal{P}})]
+((b_j)_s+ab_j)[(\partial_{b_j}Q_{1\mathcal{P}},-\partial_{b_j} Q_{2\mathcal{P}})\\
&+(\partial_{b_j}Q_{2\mathcal{P}},\partial_{b_j} Q_{1\mathcal{P}})]
+[(\partial_s\epsilon_1,-\partial_b Q_{2\mathcal{P}})+(\partial_s\epsilon_2,\partial_{b_j} Q_{1\mathcal{P}})]\\
=&\left(\frac{\lambda_s}{\lambda}+a\right)[(\Lambda Q_{1\mathcal{P}}+\Lambda\epsilon_1,-\partial_{b_j} Q_{2\mathcal{P}})+
(\Lambda Q_{2\mathcal{P}}+\Lambda\epsilon_2,\partial_{b_j} Q_{1\mathcal{P}})]\\
&+\left(\frac{(\alpha)_s}{\lambda}-b\right)\cdot[(\nabla Q_{1\mathcal{P}}+\nabla\epsilon_1,-\partial_{b_j} Q_{2\mathcal{P}})+(\nabla Q_{2\mathcal{P}}+\nabla\epsilon_2,\partial_{b_j} Q_{1\mathcal{P}})]
+\tilde{\gamma}_s[(Q_{2\mathcal{P}}+\epsilon_2,-\partial_{b_j} Q_{2\mathcal{P}})\\
&-(Q_{1\mathcal{P}}+\epsilon_1,\partial_{b_j} Q_{1\mathcal{P}})]
+(R_2(\epsilon),\partial_{b_j} Q_{2\mathcal{P}})+(R_1(\epsilon),\partial_{b_j} Q_{1\mathcal{P}})
-(\Im(\Phi_{\mathcal{P}}), \partial_{b_j} Q_{2\mathcal{P}})\\
&+(\Re(\Phi_{\mathcal{P}}),\partial_{b_j} Q_{1\mathcal{P}})
+\mathcal{O}(\mathcal{P}^2\|\epsilon\|_2).
\end{align*}
Hence, we have
\begin{align*}
&\left(a_s+\frac{1}{2}a^2\right)\mathcal{O}(\mathcal{P})\\
=&\left(\frac{\lambda_s}{\lambda}+a\right)[-(\Lambda Q,S_{0,1,j})+\mathcal{O}(\mathcal{P}^2)]
+\left(\frac{\alpha_s}{\lambda}-b\right)\cdot[(\nabla Q,S_{0,1,j})+\mathcal{O}(\mathcal{P}^2)]\\
&+(R_2(\epsilon),\partial_b Q_{2\mathcal{P}})+(R_1(\epsilon),\partial_b Q_{1\mathcal{P}})+\mathcal{O}(\mathcal{P}^2\|\epsilon\|_2).
\end{align*}
Therefore, we deduce that
\begin{align*}
&\left(a_s+\frac{1}{2}a^2\right)\mathcal{O}(\mathcal{P})
+\left(\frac{\alpha_s}{\lambda}-b\right)[p_1 +\mathcal{O}(\mathcal{P}^2)]\\
=&(R_2(\epsilon),\partial_{b_j} Q_{2\mathcal{P}})+(R_1(\epsilon),\partial_{b_j} Q_{1\mathcal{P}})
+\mathcal{O}\left((\mathcal{P}^2
+|\mathbf{mod}(t)|)\|\epsilon\|_2+a^4+b^2+b\mathcal{P}^2\right),
\end{align*}
where we use $(\Lambda Q,S_{0,1})=0$.

$\mathbf{Step~6~Conclusion.}$  We collect the previous equation and estimate the nonlinear terms in $\epsilon$ by Sobolev inequalities. This gives us
\begin{align*}
(A+B)\mathbf{Mod}(t)=&\mathcal{O}\big((\mathcal{P}^2+|\mathbf{Mod}(t)|)\|\epsilon_2\|
+\|\epsilon\|_2^2+\|\epsilon\|_{H^{1/2}}^3\\
&+|\|u\|_2^2-\|Q\|_2^2|+a^4+b^2+b\mathcal{P}^2\big).
\end{align*}
Here $A=O(1)$ is invertible $7\times7$-matrix, and $B=\mathcal{O}(\mathcal{P})$ is some $7\times7$-matrix that is polynomial in $\mathcal{P}=(a,b)$. For $|\mathcal{P}|\ll1$, we can thus invert $A+B$ by Taylor expansion and derive the estimate for $\mathbf{Mod}(t)$ stated in this lemma.
\end{proof}

\section{Refined Energy bounds}\label{section-refined-energy}
In this section, we establish a refined energy estimate, which will be a key ingredient in the compactness argument to construct minimal mass blowup solutions.

Let $u=u(t,x)$ be a solution \eqref{equ-1-hf-2} on the time interval $[t_0,0)$ and suppose that $w$ is an approximate solution to \eqref{equ-1-hf-2} such that
\begin{equation}\label{equ-approximate-hf-N}
iw_t-Dw+|w|w=\psi,
\end{equation}
with the priori bounds
\begin{align}\label{energy-priori-1}
\|w\|_2\lesssim 1,\ \|D^{\frac{1}{2}}w\|_2\lesssim \lambda^{-1},\ \|\nabla w\|_2\lesssim \lambda^{-2}.
\end{align}
We decompose $u=w+\tilde{u}$, and hence $\tilde{u}$ satisfies
\begin{equation}\label{equ-app2-hf-N}
i\tilde{u}_t-D\tilde{u}+(|u|u-|w|w)=-\psi,
\end{equation}
where we assume the priori estimate
\begin{align}\label{energy-priori-2}
\|D^{\frac{1}{2}}\tilde{u}\|_2\lesssim \lambda,\ \|\tilde{u}\|_2\lesssim \lambda^{\frac{3}{2}},
\end{align}
as well as
\begin{align}\label{energy-priori-3}
|\lambda_t+a|\lesssim\lambda^2,\ a\lesssim\lambda^{\frac{1}{2}},\ |a_t|\lesssim1,\ |\alpha_t|\lesssim\lambda.
\end{align}

Next, Let $\phi:\mathbb{R}^2\rightarrow\mathbb{R}$ be a smooth and radial function with the following properties
\begin{align}\label{energy-cutoff-function}
\phi(x)=\begin{cases}x\ \ &\text{for}\ \ 0\leq x\leq1,\\
3-e^{-|x|}\ &\text{for}\ x\geq2,
\end{cases}
\end{align}
and the convexity condition
\begin{align}
\phi''(x)\geq0\ \text{for}\ x\geq0.
\end{align}
Furthermore, we denote
\begin{align}\notag
F(u)=\frac{1}{3}|u|^{3},\ f(u)=|u|u,\ F'(u)\cdot h=\Re(f(u)\bar{h}).
\end{align}
Let $A>0$ be a large constant and define the quantity
\begin{align}\notag
J_A(u):=&\frac{1}{2}\int|D^{\frac{1}{2}}\tilde{u}|^2
+\frac{1}{2}\int\frac{|\tilde{u}|^2}{\lambda}
-\int[F(u)-F(w)-F'(w)\cdot\tilde{u}]\notag\\
&+\frac{a}{2}\Im\left(\int A\nabla\phi\left(\frac{x-\alpha}{A\lambda}\right)\cdot\nabla\tilde{u}\bar{\tilde{u}}\right).
\end{align}
Our strategy will be to use the preceding functional to bootstrap control over $\|\tilde{u}\|_{H^{\frac{1}{2}}}$.
\begin{lemma}\label{lemma-energy-estimate}
(Localized energy estimate) Let $J_A$ be as above. Then we have
\begin{align}
\frac{dJ_A}{dt}
=&\Im\left(\psi,D\tilde{u}+\frac{1}{\lambda}\tilde{u}-f'(w)\tilde{u}\right)
-\frac{1}{\lambda}(\tilde{u},f'(w)\tilde{u})
-\Re\left(\partial_tw,{(f(u)-f(w)-f'(w)\cdot\tilde{u})}\right)\notag\\
&+\frac{a}{2\lambda}\int\frac{|\tilde{u}|^2}{\lambda}
-\frac{2a}{\lambda}\int_0^{+\infty}\sqrt{s}
\int_{\mathbb{R}^2}\Delta\phi(\frac{x-\alpha}{A\lambda})|\nabla\tilde{u}_s|^2dxds\notag\\
&+\frac{a}{2A^2\lambda^3}\int_0^{+\infty}\sqrt{s}\int_{\mathbb{R}^2}
\Delta^2\phi(\frac{x-\alpha}{A\lambda})|\tilde{u}_s|^2dxds\notag\\
&+\Im\left(\int\left[iaA\nabla\phi(\frac{x-\alpha}{A\lambda})\cdot\nabla\psi
+i\frac{a}{2\lambda}\Delta\psi(\frac{x-\alpha}{A\lambda})\psi\right]\bar{\tilde{u}}\right)\notag\\
&+a\Re\left(\int A\nabla\phi\big(\frac{x-\alpha}{A\lambda}\big)
\left(\frac{3}{4}|w|^{-1}|\tilde{u}|^2w
+\frac{1}{4}|w|^{-1}\tilde{u}^2\bar{w}\right)\cdot\overline{\nabla w}\right)\notag\\
&+\mathcal{O}(\lambda^2\|\psi\|_2+\|\tilde{u}\|_{H^{1/2}}^2
+\lambda^{\frac{1}{2}}\|\tilde{u}\|_{H^{1/2}}^2),
\end{align}
where $\tilde{u}_s:=\sqrt{\frac{2}{\pi}}\frac{1}{-\Delta+s}\tilde{u}$ with $s>0$.
\end{lemma}
\begin{proof}
$\mathbf{Step~1:}$ (Estimating the energy part).
Using \eqref{equ-app2-hf-N}, a computation
\begin{align}\label{energy-part-1}
&\frac{d}{dt}\left\{\frac{1}{2}\int|D^{\frac{1}{2}}\tilde{u}|^2
+\frac{1}{2}\int\frac{|\tilde{u}|^2}{\lambda}
-\int[F(u)-F(w)-F'(w)\cdot\tilde{u}]\right\}\notag\\
=&\Re\left(\partial_t\tilde{u},
{D\tilde{u}+\frac{1}{\lambda}\tilde{u}-(f(u)-f(w))}\right)
-\frac{\lambda_t}{2\lambda^2}\int|\tilde{u}|^2\notag\\
&-\Re\left(\partial_tw,{(f(u)-f(w)-f'(w)\cdot\tilde{u})}\right)\notag\\
=&-\Im\left(\psi,{D\tilde{u}+\frac{1}{\lambda}\tilde{u}-(f(u)-f(w))}\right)
-\frac{\lambda_t}{2\lambda^2}\int|\tilde{u}|^2
-\Re\left(\partial_tw,{(f(u)-f(w)-f'(w)\cdot\tilde{u})}\right)\notag\\
&-\Im\left(D\tilde{u}-(f(u)-f(w)),
{D\tilde{u}+\frac{1}{\lambda}\tilde{u}-(f(u)-f(w))}\right)\notag\\
=&-\Im\left(\psi,{D\tilde{u}+\frac{1}{\lambda}\tilde{u}-(f(u)-f(w))}\right)
-\frac{\lambda_t}{2\lambda^2}\int|\tilde{u}|^2
+\Im\left(f(u)-f(w),\frac{1}{\lambda}\tilde{u}\right)\notag\\
&-\Re\left(\partial_tw,{(f(u)-f(w)-f'(w)\cdot\tilde{u})}\right)\notag\\
=&-\Im\left(\psi,D\tilde{u}+\frac{1}{\lambda}\tilde{u}-f'(w)\tilde{u}\right)
-\frac{1}{\lambda}(\tilde{u},f'(w)\tilde{u})-\frac{\lambda_t}{2\lambda^2}\int|\tilde{u}|^2\notag\\
&+\Im\left(\psi-\frac{1}{\lambda}\tilde{u},f(u)-f(w)-f'(w)\cdot\tilde{u}\right)
-\Re\left(\partial_tw,{(f(u)-f(w)-f'(w)\cdot\tilde{u})}\right),
\end{align}
where we denote
\begin{align}\notag
f'(w)\tilde{u}=\frac{3}{2}|w|\tilde{u}
+\frac{1}{2}|w|^{-1}w^2\bar{\tilde{u}}.
\end{align}
From \eqref{energy-priori-3} we obtain that
\begin{align}\label{energy-estimate-1}
-\frac{\lambda_t}{2\lambda^2}\int|\tilde{u}|^2
&=\frac{a}{2\lambda}\int\frac{|\tilde{u}|^2}{\lambda}
-\frac{1}{2\lambda^2}(\lambda_t+a)\|\tilde{u}\|_2^2\notag\\
&=\frac{a}{2\lambda}\int\frac{|\tilde{u}|^2}{\lambda}
+\mathcal{O}(\|\tilde{u}\|_{H^{1/2}}^2).
\end{align}
Next, we estimate
\begin{align}\label{energy-estiamte-2}
&\left|\Im\left(\psi-\frac{1}{\lambda}\tilde{u},
f(u)-f(w)-f'(w)\cdot\tilde{u}\right)\right|\notag\\
\lesssim&\left(\|\psi\|_2+\lambda^{-1}\|\tilde{u}\|_2\right)\|f(u)-f(w)-f'(w)\cdot\tilde{u}\|_2\notag\\
\lesssim&\left(\|\psi\|_2+\lambda^{-1}\|\tilde{u}\|_2\right)\|\tilde{u}\|_4^2\notag\\
\lesssim&\left(\|\psi\|_2+\lambda^{-1}\|\tilde{u}\|_2\right)\|\tilde{u}\|_{\dot{H}^{1/2}}^2\notag\\
\lesssim&\lambda^2\|\psi\|_2+\|\tilde{u}\|_{H^{1/2}}^2
\end{align}
where we used the H\"{o}lder inequality and Sobolev inequality together with the assumed a-priori estimate \eqref{energy-priori-1} and \eqref{energy-priori-2}. Here we also used the following estimate
\begin{align}
|g(u+v)-g(u)-g'(u)\cdot v|\lesssim|v|^{p},
\end{align}
for $1<p\leq2$ and $g(u)=|u|^{p-1}u$.


We now insert \eqref{energy-estimate-1} and \eqref{energy-estiamte-2} into \eqref{energy-part-1}. Combined with the assumed a priori bounds on $\tilde{u}$, we conclude
\begin{align}
&\frac{d}{dt}\left\{\frac{1}{2}\int|D^{\frac{1}{2}}\tilde{u}|^2
+\frac{1}{2}\int\frac{|\tilde{u}|^2}{\lambda}
-\int[F(u)-F(w)-F'(w)\cdot\tilde{u}]\right\}\notag\\
=&-\Im\left(\psi,D\tilde{u}+\frac{1}{\lambda}\tilde{u}-f'(w)\tilde{u}\right)
-\frac{1}{\lambda}(\tilde{u},f'(w)\tilde{u})
-\Re\left(\partial_tw,{(f(u)-f(w)-f'(w)\cdot\tilde{u})}\right)\notag\\
&+\frac{a}{2\lambda}\int\frac{|\tilde{u}|^2}{\lambda}
+\mathcal{O}(\lambda^2\|\psi\|_2+\|\tilde{u}\|_{H^{1/2}}^2).
\end{align}

$\mathbf{Step~2:}$ Estimating the localized virial part. We set
\begin{align}\notag
\nabla\tilde{\phi}(t,x)=aA\nabla\phi(t,\frac{x-\alpha}{A\lambda}).
\end{align}
Then we obtain
\begin{align}
&\frac{1}{2}\frac{d}{dt}\left(a\Im\left(\int A\nabla
\Big(\frac{x-\alpha}{A\lambda}\Big)\cdot\nabla{\tilde{u}\bar{\tilde{u}}}\right)\right)\notag\\
=&\frac{1}{2}\Im\left(\int(\partial_t\nabla{\tilde{\phi}})\cdot\nabla{u}\bar{\tilde{u}}\right)
+\frac{1}{2}\Im\left(\int\nabla{\tilde{\phi}}\cdot(\nabla\partial_t\tilde{u}\bar{\tilde{u}}
+\nabla\tilde{u}\partial_t\bar{\tilde{u}})\right).
\end{align}
Using the bounds \eqref{energy-priori-3}, we estimate
\begin{align}
|\partial_t\nabla\tilde{\phi}|\lesssim|a_t|+\frac{a}{\lambda}\alpha_t
+a\frac{\lambda_t}{\lambda}\lesssim1\\
|\partial_t\Delta\tilde{\phi}|\lesssim\lambda^{-1}
\end{align}
Hence, by \cite[Lemma F.1]{KLR2013}, we deduce that
\begin{align}
\left|\Im\left((\partial_t\nabla{\tilde{\phi}})\cdot\nabla{u}\bar{\tilde{u}}\right)\right|
\lesssim\|\tilde{u}\|_{\dot{H}^{1/2}}^2+\lambda^{-1}\|\tilde{u}\|_2^2.
\end{align}
Using \eqref{equ-app2-hf-N}, a calculation yield that
\begin{align}\label{energy-prior-part-2}
&\frac{1}{2}\Im\left(\int\nabla{\tilde{\phi}}\cdot(\nabla\partial_t\tilde{u}\bar{\tilde{u}}
+\nabla\tilde{u}\partial_t\bar{\tilde{u}})\right)\notag\\
=&-\frac{1}{4}\Re\left(\int\bar{\tilde{u}}\left[-iD,
\nabla\tilde{\phi}\cdot (-i\nabla)+(-i\nabla)\cdot\nabla\tilde{\phi}\right]\right)\notag\\
&-a\Re(\int(|u|u-|w|w)A\nabla\phi(\frac{x-\alpha}{A\lambda})
\cdot\overline{\nabla\tilde{u}})\notag\\
&-\frac{1}{2}\frac{a}{\lambda}\Re\left(\int(|u|u-|w|w)
A\Delta\phi(\frac{x-\alpha}{A\lambda})\cdot\bar{\tilde{u}}\right)\notag\\
&-a\Re\left(\int\psi\nabla\phi(\frac{x-\alpha}{A\lambda})
\cdot\overline{\nabla\tilde{u}}\right)-\frac{1}{2}\frac{a}{\lambda}\Re
\left(\int\psi\Delta(\frac{x-\alpha}{A\lambda})\bar{\tilde{u}}\right).
\end{align}
Next, we rewrite the commutator by using some identities from functional calculus. Here, we recall the known formula
\begin{align}\notag
x^\beta=\frac{\sin(\pi\beta)}{\pi}\int_0^\infty s^{\beta-1}\frac{x}{x+s}ds
\end{align}
for $x>0$ and $0<\beta<1$. From spectral calculus applied to the self-adjoint operator $-\Delta$, we have the Balakrishnan's formula
\begin{align}
(-\Delta)^\beta=\frac{\sin(\pi\beta)}{\pi}\int_0^\infty s^{\beta-1}\frac{-\Delta}{-\Delta+s}ds.
\end{align}
Next, we note the formal identity
\begin{align}
[\frac{A}{A+B},B]=[1-\frac{s}{A+s},B]=-[\frac{s}{A+s},B]=
s\frac{1}{A+s}[A,B]\frac{1}{A+s}
\end{align}
for operators $A\geq0$, and $B$, where $s>0$ is the positive constant. We obtain the formal commutator identity
\begin{align}
[(-\Delta)^m,B]=\frac{\sin(\pi m)}{\pi}\int s^m\frac{1}{-\Delta+s}[-\Delta,B]\frac{1}{-\Delta+s}ds.
\end{align}
In particular, we deduce that
\begin{align}
&[D,\nabla\tilde{\phi}\cdot (-i\nabla)+(-i\nabla)\cdot\nabla\tilde{\phi}]\notag\\
=&\frac{1}{\pi}\int \sqrt{s}\frac{1}{-\Delta+s}
[-\Delta,\nabla\tilde{\phi}\cdot (-i\nabla)+(-i\nabla)\cdot\nabla\tilde{\phi}]\frac{1}{-\Delta+s}ds.
\end{align}
Next, we recall the known formula
\begin{align}\notag
[-\Delta,\nabla\tilde{\phi}\cdot(-i\nabla)+(-i\nabla)\cdot\nabla\tilde{\phi}]
=-4\nabla\cdot(\Delta\tilde{\phi}\cdot(-i\nabla))+i\Delta^2\tilde{\phi},
\end{align}
for any smooth function $\phi$.

We now define the auxiliary function
\begin{align}
\tilde{u}_s(t,x):=\sqrt{\frac{2}{\pi}}\frac{1}{-\Delta+s}\tilde{u}(t,x).\ \text{for}\ s>0.
\end{align}
Hence, by construction, we have that $\tilde{u}_s$ solves the elliptic equation
\begin{equation}
-\Delta\tilde{u}_s+s\tilde{u}_s=\sqrt{\frac{2}{\pi}}\tilde{u}.
\end{equation}
Note that the integral kernel for the resolvent $(-\Delta+s)^{-1}$ is explicitly given by
\begin{align}\notag
G^s(x)=\int_0^{\infty}(4\pi t)^{-1}exp\left\{-\frac{|x|^2}{4t}-st\right\}dt.
\end{align}
Hence, we remark that we have the convolution formula
\begin{align}\notag
\tilde{u}_s=\sqrt{\frac{2}{\pi}}G^s(x)*\tilde{u}(t,x).
\end{align}
Recalling that $\nabla\tilde{\phi}(t,x)=aA\nabla\phi(\frac{x-\alpha}{A\lambda})$ and using that $(-\Delta+s)^{-1}$ is self-adjoint and the definition of $\tilde{u}_s$, as well as Fubini's theorem, we conclude that
\begin{align}\label{energy-part2-4}
&\Re\left(\int\bar{\tilde{u}}\left[-iD,
\nabla\tilde{\phi}\cdot (-i\nabla)
+(-i\nabla)\cdot\nabla\tilde{\phi}\right]\tilde{u}\right)\notag\\
=&\Re\left(\int\bar{\tilde{u}}\frac{1}{\pi}\int \sqrt{s}\frac{1}{-\Delta+s}
[-i(-\Delta),\nabla\tilde{\phi}\cdot (-i\nabla)
+(-i\nabla)\cdot\nabla\tilde{\phi}]\frac{1}{-\Delta+s}ds\tilde{u}dx\right)\notag\\
=&\frac{1}{\pi}\Re\left(\int\int \sqrt{s}\bar{\tilde{u}}\frac{1}{-\Delta+s}
(4\nabla\cdot(\Delta\tilde{\phi}\cdot\nabla)+\Delta^2\tilde{\phi})
\frac{1}{-\Delta+s}\tilde{u}dxds\right)\notag\\
=&\frac{1}{2}\Re\left(\int\int \sqrt{s}\bar{\tilde{u}}_s
(4\nabla\cdot(\Delta\tilde{\phi}\cdot\nabla)+\Delta^2\tilde{\phi})
\tilde{u}_sdxds\right)\\ \notag
=&-\frac{2a}{\lambda}\int_0^{+\infty}\sqrt{s}
\int_{\mathbb{R}^2}\Delta\phi(\frac{x-\alpha}{A\lambda})|\nabla\tilde{u}_s|^2dxds
+\frac{a}{2A^2\lambda^3}\int_0^{+\infty}\sqrt{s}\int_{\mathbb{R}^2}
\Delta^2\phi(\frac{x-\alpha}{A\lambda})|\tilde{u}_s|^2dxds.
\end{align}
Next, we estimate the other term in \eqref{energy-prior-part-2}. Using fractional Leibniz rule as well as the bound \eqref{energy-priori-1}, \eqref{energy-priori-2} and \eqref{energy-priori-3}, we find that
\begin{align}\label{energy-part2-1}
&\Big|-a\Re\left(\int A\nabla\phi(\frac{x-\alpha}{A\lambda})
(f(u)-f(w)-f'(w)\tilde{u})\cdot\overline{\nabla\tilde{u}}\right)\notag\\
&-\frac{1}{2}\frac{a}{\lambda}\Re\left(\int\Delta\phi(\frac{x-\alpha}{A\lambda})
(f(u)-f(w)-f'(w)\tilde{u})\cdot\overline{\tilde{u}}\right)\Big|\notag\\
\lesssim&\left|-a\Re\left(\int A\nabla\phi(\frac{x-\alpha}{A\lambda})
|\tilde{u}|\tilde{u})\cdot\overline{\nabla\tilde{u}}\right)\right|
+\left|\frac{1}{2}\frac{a}{\lambda}\Re\left(\int\Delta\phi(\frac{x-\alpha}{A\lambda})
|\tilde{u}|\tilde{u})\cdot\overline{\tilde{u}}\right)\right|\notag\\
\lesssim&\left|-a\Re\left(\int A\nabla\phi(\frac{x-\alpha}{A\lambda})
\nabla(\tilde{u})^3\right)\right|
+\left|\frac{1}{2}\frac{a}{\lambda}\Re\left(\int\Delta\phi(\frac{x-\alpha}{A\lambda})
|\tilde{u}|\tilde{u})\cdot\overline{\tilde{u}}\right)\right|\notag\\
\lesssim&\left|\frac{1}{2}\frac{a}{\lambda}\Re\left(\int\Delta\phi(\frac{x-\alpha}{A\lambda})
|\tilde{u}|\tilde{u})\cdot\overline{\tilde{u}}\right)\right|\notag\\
\lesssim&\lambda^{-\frac{1}{2}}\int|\tilde{u}|^3
\lesssim\lambda^{\frac{1}{2}}\|\tilde{u}\|_{H^{1/2}}^2.
\end{align}
We consider the term in \eqref{energy-prior-part-2} that are quadratic in $\tilde{u}$. Integrating by parts, we obtain
\begin{align}\label{energy-part2-2}
&-a\Re\left(\int\psi\nabla\phi\left(\frac{x-\alpha}{A\lambda}\right)
\cdot\overline{\nabla\tilde{u}}\right)-\frac{1}{2}\frac{a}{\lambda}\Re
\left(\int\psi\Delta\left(\frac{x-\alpha}{A\lambda}\right)\bar{\tilde{u}}\right)\notag\\
&=\Im\left(\int\left[iaA\nabla\phi(\frac{x-\alpha}{A\lambda})\cdot\nabla\psi
+i\frac{a}{2\lambda}\Delta\psi(\frac{x-\alpha}{A\lambda})\psi\right]\bar{\tilde{u}}\right).
\end{align}
Moreover, an integration by parts yields that
\begin{align}\label{energy-part2-3}
&-a\Re\left(\int\psi A\nabla\phi\left(\frac{x-\alpha}{A\lambda}\right)
\left(\frac{3}{2}|w|\tilde{u}
+\frac{1}{2}|w|^{-1}w^2\bar{\tilde{u}}\right)\cdot\overline{\nabla \tilde{u}}\right)\notag\\
&-\frac{1}{2}\frac{a}{\lambda}\Re\left(\int\Delta\phi\left(\frac{x-\alpha}{A\lambda}\right)
\left(\frac{3}{2}|w|\tilde{u}
+\frac{1}{2}|w|^{-1}w^2\bar{\tilde{u}}\right)\cdot\overline{\tilde{u}}\right)\notag\\
=&a\Re\left(\int A\nabla\phi\left(\frac{x-\alpha}{A\lambda}\right)
\left(\frac{3}{4}|w|^{-1}|\tilde{u}|^2w
+\frac{1}{4}|w|^{-1}\tilde{u}^2\bar{w}\right)\cdot\overline{\nabla w}\right).
\end{align}
Note that $\Delta\phi$ is not present on the right-hand side of the previous equation and that the term is different from those appearing on the left-hand side.

Finally, we insert \eqref{energy-part2-4},\eqref{energy-part2-1}, \eqref{energy-part2-2}, and \eqref{energy-part2-3} into \eqref{energy-prior-part-2}. This yield that
\begin{align}
&\frac{1}{2}\Im\left(\nabla{\tilde{\phi}}\cdot(\nabla\partial_t\tilde{u}\bar{\tilde{u}}
+\nabla\tilde{u}\partial_t\bar{\tilde{u}})\right)\notag\\
=&-\frac{2a}{\lambda}\int_0^{+\infty}\sqrt{s}
\int_{\mathbb{R}^2}\Delta\phi\left(\frac{x-\alpha}{A\lambda}\right)|\nabla\tilde{u}_s|^2dxds\notag\\
&+\frac{a}{2A^2\lambda^3}\int_0^{+\infty}\sqrt{s}\int_{\mathbb{R}^2}
\Delta^2\phi\left(\frac{x-\alpha}{A\lambda}\right)|\tilde{u}_s|^2dxds\notag\\
&+\Im\left(\int\left[iaA\nabla\phi\left(\frac{x-\alpha}{A\lambda}\right)\cdot\nabla\psi
+i\frac{a}{2\lambda}\Delta\psi\left(\frac{x-\alpha}{A\lambda}\right)\psi\right]\bar{\tilde{u}}\right)\notag\\
&+a\Re\left(\int A\nabla\phi\left(\frac{x-\alpha}{A\lambda}\right)
\left(\frac{3}{4}|w|^{-1}|\tilde{u}|^2w
+\frac{1}{4}|w|^{-1}\tilde{u}^2\bar{w}\right)\cdot\overline{\nabla w}\right)\notag\\
&+\mathcal{O}(\lambda^{\frac{1}{2}}\|\tilde{u}\|_{H^{1/2}}^2).
\end{align}
This completes the proof of lemma.
\end{proof}
\section{Backwards Propagation of Smallness}
We now apply the energy estimate of the previous section in order to establish a bootstrap argument that will be needed in the construction of minimal mass blowup solution.

Let $u=u(t,x)$ ba a solution to \eqref{equ-1-hf-2} defined in $[t_0,0)$. Assume that $t_0<t_1<0$ and suppose that $u$ admits on $[t_0,t_1]$ a geometrical decomposition of the form
\begin{align}\label{back-decomposition}
u(t,x)=\frac{1}{\lambda(t)}[Q_{\mathcal{P}(t)}+\epsilon]
\left(s,\frac{x-\alpha(t)}{\lambda(t)}\right)
e^{i\gamma(t)},
\end{align}
where $\epsilon=\epsilon_1+i\epsilon_2$ satisfies the orthogonality condition \eqref{mod-orthogonality-condition} and $a^2+|b|+\|\epsilon\|_{H^{1/2}}^2\ll1$ holds. We set
\begin{align}\label{back-small-part}
\tilde{u}(t,x)=\frac{1}{\lambda(t)}\epsilon
\left(s,\frac{x-\alpha(t)}{\lambda(t)}\right)
e^{i\gamma(t)}.
\end{align}
Suppose that the energy satisfies $E_0=E(u_0)>0$ and define the constant
\begin{align}\label{back-define-1}
A_0=\sqrt{\frac{e_1}{E_0}},
\end{align}
with the vector $e_1=\frac{1}{2}(L_{-}S_{1,0},S_{1,0})>0$. Moreover, Let $P_0=P(u_0)$ be the linear momentum and define the constant
\begin{align}\label{back-define-2}
B_0=\frac{P_0}{p_1},
\end{align}
where $p_1=2\int_{\mathbb{R}^2}L_{-}S_{0,1}\cdot S_{0,1}>0$ is a constant.

Now we claim that the following backwards propagation estimate holds.
\begin{lemma}(Backwards propagation of smallness)\label{lemma-back}
Assume that, for some $t_1<0$ sufficiently close to $0$, we have the bounds
\begin{align*}
&\left|\|u\|_2^2-\|Q\|_2^2\right|\lesssim\lambda^3(t_1),\\
&\|D^{\frac{1}{2}}\tilde{u}(t_1)\|_2^2+\frac{\|\tilde{u}\|_2^2}{\lambda(t_1)}
\lesssim\lambda^2(t_1),\\
&\left|\lambda(t_1)-\frac{t_1^2}{4A_0^2}\right|\lesssim\lambda^{3}(t_1),\
\left|\frac{a(t_1)}{\lambda^{\frac{1}{2}}(t_1)}\right|\lesssim\lambda(t_1),\
\left|\frac{b(t_1)}{\lambda(t_1)}-B_0\right|\lesssim\lambda(t_1),
\end{align*}
where $A_0$ and $B_0$ are defined in \eqref{back-define-1} and \eqref{back-define-2}, respectively. Then there exists a time $t_0<t_1$ depending on $A_0$ and $B_0$ such that $\forall t\in[t_0,t_1]$, it holds that
\begin{align*}
&\|D^{\frac{1}{2}}\tilde{u}(t)\|_2^2+\frac{\|\tilde{u}\|_2^2}{\lambda(t)}\lesssim
\|D^{\frac{1}{2}}\tilde{u}(t_1)\|_2^2+\frac{\|\tilde{u}\|_2^2}{\lambda(t_1)}\lesssim\lambda^2(t),\\
&\left|\lambda(t)-\frac{t^2}{4A_0^2}\right|\lesssim\lambda^{3}(t),\
\left|\frac{a(t)}{\lambda^{\frac{1}{2}}(t)}-\frac{1}{A_0}\right|\lesssim\lambda(t),\
\left|\frac{b(t)}{\lambda(t)}-B_0\right|\lesssim\lambda(t).
\end{align*}
\end{lemma}
\begin{proof}
By assumption, we have $u\in C^0([t_0,t_1];H^{1/2}(\mathbb{R}^2))$. Hence, by this continuity and the continuity of the functions $\{\lambda(t),a(t),b(t),\alpha(t)\}$, there exists a time $t_0<t_1$ such that for any $t\in[t_0,t_1]$ we have the bounds
\begin{align}\label{back-claim-1}
&\|\tilde{u}\|_2^2\lesssim K\lambda^3(t),\ \|\tilde{u}(t)\|_{H^{1/2}}\lesssim K\lambda(t),\\\label{back-claim-2}
&\left|\lambda(t)-\frac{t^2}{4A_0^2}\right|\lesssim K\lambda^{3}(t),\
\left|\frac{a(t)}{\lambda^{\frac{1}{2}}(t)}-\frac{1}{A_0}\right|\lesssim K\lambda(t),\
\left|\frac{b(t)}{\lambda(t)}-B_0\right|\lesssim K\lambda(t),
\end{align}
with some constant $K>0$. We now claim that the bounds stated in this lemma hold on $[t_0,t_1]$, hence improving \eqref{back-claim-1} and \eqref{back-claim-2} on $[t_0,t_1]$ for $t_0=t_0(C_0)<t_1$ small enough but independent of $t_1$. We divide the proof into the following steps.

$\mathbf{Step~1~Bounds~on~energy~and~L^2-norm}$. We set
\begin{align}\label{back-1-1}
w(x,t)=\tilde{Q}(t,x)=\frac{1}{\lambda(t)}Q_{\mathcal{P}(t)}
\left(\frac{x-\alpha(t)}{\lambda(t)}\right)e^{i\gamma(t)}.
\end{align}
Let $J_A$ be given by above section. Applying lemma \ref{lemma-energy-estimate}, we claim that we obtain the following coercivity estimate:
\begin{align}\label{back-1-claim}
\frac{J_A}{dt}\geq\frac{a}{\lambda^{2}}\int|\epsilon|^2
+\mathcal{O}\left(\|\tilde{u}\|_{H^{1/2}}^2+K^4\lambda^{\frac{5}{2}}\right),
\end{align}
Assume \eqref{back-1-claim} holds. By the Sobolev embedding and small of $\epsilon$, we deduce the upper bound
\begin{align}\label{back-1-2}
|J_A|\lesssim\|D^{\frac{1}{2}}\tilde{u}\|_2^2+\frac{1}{\lambda}\|\tilde{u}\|_2^2
\end{align}
Here we use the following inequality
\begin{align}\notag
\left|\Im\left(\int A\nabla\phi(\frac{x-\alpha}{A\lambda})
\cdot\nabla\tilde{u}\bar{\tilde{u}}\right)\right|
\lesssim\|D^{\frac{1}{2}}\tilde{u}\|_2^2+\frac{1}{\lambda}\|\tilde{u}\|_2^2,
\end{align}
where we can see \cite[Lemma F.1]{KLR2013}. Furthermore, due to the proximity of $Q_{\mathcal{P}}$ to $Q$, we derive the lower bound
\begin{align}\label{back-1-3}
J_A=&\frac{1}{2}\int|D^{\frac{1}{2}}\tilde{u}|^2
+\frac{1}{2}\int\frac{|\tilde{u}|^2}{\lambda}
-\int[F(u)-F(w)-F'(w)\cdot\tilde{u}]\notag\\
&+\frac{a}{2}\Im\left(\int A\nabla\phi
\left(\frac{x-\alpha}{A\lambda}\right)\cdot\nabla\tilde{u}\bar{\tilde{u}}\right)\notag\\
=&\frac{1}{2\lambda}\left[(L_{+}\epsilon_1,\epsilon_1)+(L_{-}\epsilon_2,\epsilon_2)
+o(\|\epsilon\|_{H^{1/2}}^2)\right]\notag\\
\geq&\frac{C_0}{\lambda}\left[\|\epsilon\|_{H^{1/2}}^2-(\epsilon_1,Q)^2\right],
\end{align}
using the orthogonality conditions \eqref{mod-orthogonality-condition} satisfied by $\epsilon$ and the coercivity estimate for the linearized operator $L=(L_{+},L_{-})$. On the other hand, using the conservation of the $L^2-$mass and applying lemma \ref{lemma-mod-1}, we combine the assumed bounds to conclude that
\begin{align}\notag
|\Re(\epsilon,Q_{\mathcal{P}})|\lesssim \|\epsilon\|_2^2+\lambda^2(t)+
\left|\int|u|^2-\int|Q|^2\right|\lesssim \|\epsilon\|_2^2+K^2\lambda^2(t).
\end{align}
This implies
\begin{align}\label{back-1-4}
(\epsilon_1,Q)^2\lesssim o(\|\epsilon\|_2^2)+K^4\lambda^4(t).
\end{align}
Next, we define
\begin{align}\notag
H(t):=\|D^{\frac{1}{2}}\tilde{u}(t)\|_2^2
+\frac{1}{\lambda(t)}\|\tilde{u}(t)\|_2^2.
\end{align}
By integrating \eqref{back-1-claim} in time and using \eqref{back-1-2}, \eqref{back-1-3} and \eqref{back-1-4}, we find that
\begin{align*}
H(t)\lesssim&H(t_1)+K^4\lambda^3(t)+\int_{t}^{t_1}\left(\|\tilde{u}\|_{H^{1/2}}^2
+K^4\lambda^{\frac{5}{2}}(\tau)\right)d\tau\\
\lesssim&H(t_1)+K^4\lambda^{3}(t)+\int_{t}^{t_1}H(\tau)d\tau
\end{align*}
for $t\in[t_0,t_1]$ with some $t_0=t_0(C_0)<t_1$ close enough to $t_1<0$. By Gronwall's inequality, we deduce the desired bound for $H(t)$. In particular, we obtain
\begin{align}
H(t):=\|D^{\frac{1}{2}}\tilde{u}(t)\|_2^2
+\frac{1}{\lambda(t)}\|\tilde{u}(t)\|_2^2\lesssim\lambda^2(t),\ \text{for}\ t\in[t_0,t_1],
\end{align}
and closes the bootstrap for \eqref{back-claim-1}.

$\mathbf{Step~2~Controlling~the~law~for~the~parameters.}$ From lemma \ref{lemma-mod-2} and using \eqref{back-claim-2}, we deduce
\begin{align}\label{back-2-1}
\left|a_s+\frac{1}{2}a^2\right|+\left|\frac{\lambda_s}{\lambda}+a\right|\lesssim\lambda^2.
\end{align}
As a direct consequence of this bound, we obtain that
\begin{align}\notag
\left(\frac{a}{\lambda^{\frac{1}{2}}}\right)_s
=\frac{a_s+\frac{1}{2}a^2}{\lambda^{\frac{1}{2}}}-\frac{a}{2\lambda^{\frac{1}{2}}}
\left(\frac{\lambda_s}{\lambda}+a\right)\lesssim \lambda^{\frac{3}{2}}.
\end{align}
Hence, for any $s<s_1$, we have
\begin{align}\label{back-2-2}
\frac{1}{A_0}-\frac{a}{\lambda^{\frac{1}{2}}}(s)
\lesssim\frac{1}{A_0}-\frac{a}{\lambda^{\frac{1}{2}}}(s_1)
+\int_{s}^{s_1}\lambda^{\frac{3}{2}}(\tau)d\tau\lesssim\lambda(s).
\end{align}
Here we used $\lambda(t)\sim t^2$ and the relation $dt=\lambda^{-1}ds$, as well as the assumed initial bound for $\left|\frac{a(t)}{\lambda^{\frac{1}{2}}(t)}-\frac{1}{A_0}\right|$ at time $t=t_1$. Next, by following the calculations int proof of lemma \ref{lemma-mod-1} and recalling that $a^2+|b|\sim \lambda^2$ thanks to \eqref{back-claim-2} and $\|\epsilon\|_{H^{1/2}}^2\lesssim\lambda^3$, we deduce
\begin{align}\notag
a^2e_1=\lambda E_0+\left(\int|u|^2-\int|Q|^2\right)+\mathcal{O}(\lambda^2),
\end{align}
where $e_1=\frac{1}{2}(L_{-}S_{1,0},S_{1,0})>0$ is a constant. Since $\int|u|^2-\int|Q|^2=\mathcal{O}(\lambda^3)$ and recalling the definition of \eqref{back-define-1}, we deduce that
\begin{align}\notag
\frac{a^2}{\lambda}-\frac{1}{A_0^2}
=\left(\frac{a}{\lambda^{\frac{1}{2}}}-\frac{1}{A_0}\right)
\left(\frac{a}{\lambda^{\frac{1}{2}}}+\frac{1}{A_0}\right)=\mathcal{O}(\lambda).
\end{align}
Furthermore, from \eqref{back-2-2} we see that $\frac{a}{\lambda^{\frac{1}{2}}}\geq1$. Hence, we obtain the desired bound
\begin{align}\notag
\left|\frac{a}{\lambda^{\frac{1}{2}}}-\frac{1}{A_0}\right|\lesssim\lambda.
\end{align}
From \eqref{back-claim-1} and \eqref{back-2-1}, we conclude that
\begin{align}\notag
-\lambda_t=a+\mathcal{O}(\lambda^2)=\frac{\lambda^{\frac{1}{2}}}{A_0}
+\mathcal{O}(\lambda^{\frac{3}{2}}+t^4)=\frac{\lambda^{\frac{1}{2}}}{A_0}
+\mathcal{O}(t^3).
\end{align}
Dividing the above equality by $\lambda^{\frac{1}{2}}$, and integrating in $[t,t_1]$ and using the boundary value at $t_1$, we have
\begin{align}\notag
\left|\lambda^{\frac{1}{2}}(t)-\frac{t}{2A_0}\right|\lesssim
\left|\lambda^{\frac{1}{2}}(t_1)-\frac{t_1}{2A_0}\right|+\mathcal{O}(t^3)\lesssim t^2,
\end{align}
and the bound for $\lambda$ is obtained.

Next, we improve the bound \eqref{back-claim-2}. In fact, by following the calculations in the proof of lemma \ref{lemma-mod-1} for the linear momentum $P(u_0)$ and recalling that $a^2+|b|\sim\lambda$, we deduce that
\begin{align}\notag
bp_1=\lambda P_0+\mathcal{O}(\lambda^2),
\end{align}
where $p_1=2\int_{\mathbb{R}^2}L_{-}S_{0,1}\cdot S_{0,1}$ is a positive constant. Here we also used the fact that $\|\epsilon\|_{H^{\frac{1}{2}}}^2\lesssim\lambda^3$. Recalling the definition of $B_0=\frac{P_0}{p_1}$ see \eqref{back-define-2}, we thus obtain
\begin{align}\notag
\left|\frac{b(t)}{\lambda(t)}-B_0\right|\lesssim\lambda(t).
\end{align}
This completes the proof of Step 2.

$\mathbf{Step~3~Proof~of~the~coercivity~estimate}$. Recalling that $w=\tilde{Q}$. Let $\mathcal{K}_A(\tilde{u})$ denote the terms in $\tilde{u}$ on the righthand side in lemma \ref{lemma-energy-estimate}, that is, we have
\begin{align}
\mathcal{K}_A(\tilde{u})&=-\frac{1}{\lambda}(\tilde{u},f'(w)\tilde{u})
+\frac{a}{2\lambda}\int\frac{|\tilde{u}|^2}{\lambda}\notag\\
&-\frac{2a}{\lambda}\int_0^{+\infty}\sqrt{s}
\int_{\mathbb{R}^2}\Delta\phi\left(\frac{x-\alpha}{A\lambda}\right)|\nabla\tilde{u}_s|^2dxds\notag\\
&+\frac{a}{2A^2\lambda^3}\int_0^{+\infty}\sqrt{s}\int_{\mathbb{R}^2}
\Delta^2\phi\left(\frac{x-\alpha}{A\lambda}\right)|\tilde{u}_s|^2dxds\notag\\
&+a\Re\left(\int A\nabla\phi\left(\frac{x-\alpha}{A\lambda}\right)
\left(\frac{3}{4}|w|^{-1}|\tilde{u}|^2w
+\frac{1}{4}|w|^{-1}\tilde{u}^2\bar{w}\right)\cdot\overline{\nabla w}\right).
\end{align}
Recalling that the function $\tilde{u}_s=\tilde{u}_s(t,x)$ with the parameter $s>0$ was defined in lemma \ref{lemma-energy-estimate} to be $\tilde{u}_s=\sqrt{\frac{2}{\pi}}\frac{1}{-\Delta+s}\tilde{u}$ and $\tilde{u}=\frac{1}{\lambda}\epsilon\left(t,\frac{x}{\lambda}\right)$, we now claim that the following estimate holds:
\begin{align}
\mathcal{K}_A(\tilde{u})\geq\frac{C}{\lambda^{\frac{3}{2}}}\int|\epsilon|^2
+\mathcal{O}(K^4\lambda^{\frac{5}{2}}),
\end{align}
where $C>0$ is some positive constant.

Indeed, from the lemma \ref{lemma-mod-2} and the estimate \eqref{back-claim-1} we obtain that
\begin{align}\label{back-mod-estimate}
|\mathbf{Mod}(t)|\lesssim K^2\lambda^4(t).
\end{align}
We find that $w=\tilde{Q}$ satisfies
\begin{align*}
\partial_t \tilde{Q}
=&e^{i\gamma(t)}\frac{1}{\lambda}\left[-\frac{\lambda_t}{\lambda}\Lambda Q_{\mathcal{P}}+i\gamma_tQ_{\mathcal{P}}+a_t\frac{\partial Q_{\mathcal{P}}}{\partial_a}-\frac{\alpha_t}{\lambda}\cdot \nabla Q_{\mathcal{P}}\right]\left(\frac{x-\alpha}{\lambda}\right)\\
=&\left(\frac{i}{\lambda}+\frac{a}{\lambda}\right)\tilde{Q}
+a\left(\frac{x-\alpha}{\lambda}\right)\cdot\nabla\tilde{Q}
+\mathcal{O}(K\lambda^{-1}),
\end{align*}
where we use the uniform bounds $\|\partial_aQ_{\mathcal{P}}\|\lesssim1$, $\|\partial_bQ_{\mathcal{P}}\|\lesssim1$ and the fact that $|a_t|\lesssim K$, $|b_t|\lesssim K$, which can be seen from \eqref{back-mod-estimate} and \eqref{back-claim-2}. Hence
\begin{align*}
-\Re\int\partial_t\tilde{Q}|\tilde{u}|^2
=&-\frac{1}{\lambda}\Im\int\tilde{Q}|\tilde{u}|^2
-\frac{a}{\lambda}\Re\int\tilde{Q}|\tilde{u}|^2\\
&-a\Re\int\left(\frac{x-\alpha}{\lambda}\right)|\tilde{u}|^2\cdot\nabla\tilde{Q}
+\mathcal{O}(K\lambda^{-1}\|\epsilon\|_2^2).
\end{align*}
By the definition of $\mathcal{K}_A(\tilde{u})$ and expressing everything in terms of $\epsilon(t,x)=\lambda\tilde{u}(t,\lambda x+\alpha)$, we conclude that
\begin{align}
\mathcal{K}_A(\tilde{u})\geq&\frac{a}{2\lambda^2}\Big\{\int_0^{\infty}\sqrt{s}
\int\Delta\left(\frac{x}{A}\right)|\nabla\epsilon_s|^2dxds+\int|\epsilon|^2
-2\int Q_{1\mathcal{P}}(\epsilon_1^2+\epsilon_2^2)\notag\\
&-\frac{1}{2A^2}\int_0^{\infty}\sqrt{s}\int\Delta^2\phi\left(\frac{x}{A}\right)
|\epsilon_s|^2dxds\notag\\
&+2\Re\left(\int A\nabla\phi\big(\frac{x}{A}\big)
\left(\frac{3}{4}|Q_{\mathcal{P}}|^{-1}|\epsilon|^2Q_{\mathcal{P}}
+\frac{1}{4}|Q_{\mathcal{P}}|^{-1}\epsilon^2\bar{Q}_{\mathcal{P}}\right)
\cdot\overline{\nabla Q_{\mathcal{P}}}\right)\notag\\
&-2\int x\cdot(|\epsilon|^2\nabla Q_{\mathcal{P}})\Big\}
+\mathcal{O}(K\lambda^{-1}\|\epsilon\|_2^2).
\end{align}
Next, we note that the definition of $\phi$ and we estimate
\begin{align*}
&\left|\Re\int A\nabla\phi\big(\frac{x}{A}\big)
\left(\frac{3}{4}|Q_{\mathcal{P}}|^{-1}|\epsilon|^2Q_{\mathcal{P}}
+\frac{1}{4}|Q_{\mathcal{P}}|^{-1}\epsilon^2\bar{Q_{\mathcal{P}}}\right)\cdot\overline{\nabla Q_{\mathcal{P}}}-x\cdot|\epsilon|^2\nabla Q_{\mathcal{P}}\right|\\
\lesssim&\|(A+|x|)\nabla Q_{\mathcal{P}}\|_{L^{\infty}(|x|\geq A)}\|\epsilon\|_2^2\lesssim\frac{1}{A}\|\epsilon\|_2^2,
\end{align*}
where we use the uniform decay estimate of $\nabla Q_{\mathcal{P}}$. Furthermore, thanks to lemma \ref{lemma-app-c-3}, we have
\begin{align}
\left|\int_{s=0}^{+\infty}\sqrt{s}\int\Delta^2\phi_{A}|\epsilon_s|^2dxds\right|
\lesssim\frac{1}{A}\|\epsilon\|_2^2.
\end{align}
Recalling the definitions of $L_{+,A}$ and $L_{-,A}$ in \eqref{app-c-define-1} and \eqref{app-c-define-2}, we deduce that
\begin{align}
K_{A}(\tilde{u})=\frac{a}{2\lambda^2}\left\{(L_{+,A}\epsilon_1,\epsilon_2)
+(L_{-,A}\epsilon_2,\epsilon_2)+\mathcal{O}\left(\frac{1}{A}\|\epsilon\|_2^2\right)\right\}
+\frac{1}{\lambda^{\frac{3}{2}}}\mathcal{O}(K\lambda^{-\frac{1}{2}}\|\epsilon\|_2^2).
\end{align}
Next, we recall that $a\sim\lambda^{\frac{1}{2}}$ due to the above. Hence, by lemma \ref{lemma-app-c-2}  and choosing the $A>0$ sufficiently large, we deduce from previous estimates that
\begin{align}
K_A(\tilde{u})\geq\frac{1}{\lambda^{\frac{3}{2}}}\left\{\int|\epsilon|^2-(\epsilon_1,Q)^2\right\}
\gtrsim\frac{1}{\lambda^{\frac{3}{2}}}\int|\epsilon|^2+\mathcal{O}(K^4\lambda^{\frac{5}{2}}).
\end{align}

$\mathbf{Step~4~Controlling~the~remainder~terms~in~\frac{d}{dt}J_A}$. We now control the terms that appear in lemma \ref{lemma-energy-estimate} and contain $\psi$. Here we recall that $w=\tilde{Q}$ and \eqref{equ-app2-hf-N}, which yields
\begin{align*}
\psi&=\frac{1}{\lambda^2}\Big[i(a_s+\frac{1}{2}a^2)\partial_aQ_{\mathcal{P}}
-i(\frac{\lambda_s}{\lambda}+a)\Lambda Q_{\mathcal{P}}
+i(b_s+ab)\partial_vQ_{\mathcal{P}}\\
&-i(\frac{\alpha_s}{\lambda}-b)\cdot\nabla Q_{\mathcal{P}}+
\tilde{\gamma}_sQ_{\mathcal{P}}+\Phi_{\mathcal{P}}\Big]
\left(\frac{x-\alpha}{\lambda}\right)e^{i\gamma}.
\end{align*}
Here $\Phi_{\mathcal{P}}$ is the error term given in lemma \ref{lemma-3app}. In fact, by the estimate for $Q_{\mathcal{P}}$ and $\Phi_{\mathcal{P}}$ from lemma \ref{lemma-3app} and recalling \eqref{back-mod-estimate}, we deduce the rough pointwise bounds:
\begin{align}
\left|\nabla^k\psi(x)\right|\lesssim\frac{1}{\lambda^{2+k}}
\left\langle\frac{x-\alpha}{\lambda}\right\rangle^{-3}K^2\lambda^2,\ \text{for}\ k=0,1.
\end{align}
Hence
\begin{align}
\|\nabla^k\psi\|_2\lesssim K^2\lambda^{1-k},\ \text{for}\ k=0,1.
\end{align}
In particular, we obtain the following bounds
\begin{align}
&\|\psi\|_2^2\lesssim K^4\lambda^{6},\\
&\left|\Im\left(\int\left[iaA\nabla\phi\left(\frac{x-\alpha}{A\lambda}\right)\cdot\nabla\psi
+i\frac{a}{2\lambda}\Delta\left(\frac{x-\alpha}{A\lambda}\right)\psi\right]\right)\right|\notag\\
\lesssim&\lambda^{\frac{1}{2}}\|\nabla\psi\|_2\|\tilde{u}\|_2+\lambda^{-\frac{1}{2}}\|\psi\|_2\|\tilde{u}\|_2\notag\\
\lesssim&K^2\lambda^{\frac{1}{2}}\|\epsilon\|_2\lesssim o\left(\frac{\|\epsilon\|_2^2}{\lambda^{\frac{3}{2}}}\right)+K^4\lambda^{\frac{5}{2}}.
\end{align}

Write $\psi=\psi_1+\psi_2$ with $\psi_2=\mathcal{O}(\mathcal{P}|\mathbf{Mod}|+a^5)=\mathcal{O}(\lambda^{\frac{13}{2}})$, that is, we denote
\begin{align*}
\psi_1&=\frac{1}{\lambda^2}\Big[-(a_s+\frac{1}{2}a^2)S_{1,0}
-i(\frac{\lambda_s}{\lambda}+a)\Lambda Q-(b_s+ab)S_{0,1}\\
&-i(\frac{\alpha_s}{\lambda}-b)\cdot\nabla Q+
\tilde{\gamma}_sQ\Big](\frac{x-\alpha}{\lambda})e^{i\gamma}.
\end{align*}
Let us first deal with estimating the contributions coming from $\psi_2$. Indeed, since $a^2+|b|\sim\lambda$ we note that $\psi_2=\mathcal{O}(\lambda^{\frac{5}{2}})$ satisfies the pointwise bound
\begin{align}
|\nabla^{k}\psi_2(x)|\lesssim \frac{1}{\lambda^{2+k}}
\left\langle\frac{x-\alpha}{\lambda}\right\rangle^{-3} K^2\lambda^{\frac{5}{2}},\ \text{for}\ k=0,1.
\end{align}
Hence
\begin{align}
\|\nabla^k\psi_2\|_2\lesssim K^2\lambda^{\frac{3}{2}-k},\ \text{for}\ k=0,1.
\end{align}
Therefore, we obtain that
\begin{align}
&\left|\Re\left(\int\left[-D\psi_2-\frac{\psi_2}{\lambda}+
\frac{3}{2}|w|\psi_2
+\frac{1}{2}|w|^{-1}w^2\bar{\psi_2}\right]\bar{\tilde{u}}\right)\right|\notag\\
\lesssim&\left(\|\nabla\psi_2\|_2+\lambda^{-1}\|\psi_2\|_2+\|w\|_2^{2/3}\|\nabla w\|_2^{1/3}\|\psi_2\|_{2}^{2/3}\|\nabla\psi_2\|_2^{1/3}\right)\|\epsilon\|_2\notag\\
\lesssim&K^2\lambda^{\frac{1}{2}}\|\epsilon\|_2\lesssim o\left(\frac{\|\epsilon\|_2^2}{\lambda^{\frac{3}{2}}}\right)+K^4\lambda^{\frac{5}{2}},
\end{align}
which is acceptable. Here we used the H\"{o}lder inequality and Sobolev inequality. We finally use the fact that $\psi_1$ belongs to the generalized null space of $L=(L_+,L_-)$ and hence an extra factor of $\mathcal{O}(\mathcal{P})$ is gained using the orthogonality conditions obeyed by $\epsilon=\epsilon_1+\epsilon_2$. Indeed, we find the following bound
\begin{align}
&\left|\Re\left(\int\left[-D\psi_1-\frac{\psi_1}{\lambda}+
\frac{3}{2}|w|\psi_1
+\frac{1}{2}|w|^{-1}w^2\bar{\psi_1}\right]\bar{\tilde{u}}\right)\right|\notag\\
\lesssim&\frac{|\mathbf{Mod}(t)|}{\lambda^2}\left[|(\epsilon_2,L_-S_{1,0})|+|(\epsilon_2,L_-S_{0,1})|
+|(\epsilon_2,L_-Q)|+\mathcal{O}(\mathcal{P}\|\epsilon\|_2)\right]\notag\\
&+\frac{1}{\lambda^2}\left|\frac{\lambda_s}{\lambda}+a\right||(\epsilon_1,L_+\Lambda Q)|+\frac{1}{\lambda^2}\left|\frac{\alpha_s}{\lambda}-b\right||(\epsilon_1,L_+\nabla Q)|\notag\\
\lesssim&K^2\lambda^{\frac{1}{2}}\|\epsilon\|_2
+\frac{K^2\lambda\|\epsilon\|_2}{\lambda^2}
(\lambda^{\frac{1}{2}}\|\epsilon\|_2+K^2\lambda^2)\notag\\
\lesssim&o\left(\frac{\|\epsilon\|_2^2}{\lambda^{\frac{3}{2}}}\right)+K^4\lambda^{\frac{5}{2}},
\end{align}
Here we used \eqref{back-mod-estimate} once again and $|\mathcal{P}|\lesssim\lambda^{\frac{1}{2}}$, as well as $(\epsilon_2,L_-S_{1,0})=(\epsilon_2,\Lambda Q)=\mathcal{O}(\mathcal{P}\|\epsilon\|_2)$ and $(\epsilon_2,L_-S_{0,1})=-(\epsilon_2,\nabla Q)=\mathcal{O}(\mathcal{P}\|\epsilon\|_2)$, thanks to the orthogonality conditions for $\epsilon$. Moreover, we used that $L_+\nabla Q=0$ and $L_+\Lambda Q=-Q$ together with the improved bound in lemma \ref{lemma-mod-2}, combined with the fact that $|(\epsilon_1,Q)|\lesssim\lambda^{\frac{1}{2}}\|\epsilon\|_2+K^2\lambda^2$, which follows from $\|\epsilon\|_2\lesssim\lambda^\frac{3}{2}$ and the conservation of $L^2$-norm.
 And the proof of this lemma is complete.
\end{proof}

\section{Existence of minimal mass blowup solutions}
In this section, we prove the following result.
\begin{thm}
Let $\gamma_0$, $P_0\in\mathbb{R}^2$, $x_0\in\mathbb{R}^2$ and $E_0>0$ be given. Then there exist a time $t_0<0$ and a solution $u\in C^0([t_0,0); H^{\frac{1}{2}}(\mathbb{R}^2))$ of \eqref{equ-1-hf-2} such that $u$ blowup  at time $T=0$ with
\begin{align}\notag
E(u)=E_0,\ P(u)=P_0,\ \text{and}\ \|u||_2^2=\|Q\|_2^2.
\end{align}
Furthermore, we have $\|D^{\frac{1}{2}}u\|_2\sim t^{-2}$ as $t\rightarrow0^{-}$, and $u$ is of the form
\begin{align}\notag
u(t,x)=\frac{1}{\lambda(t)}[Q_{\mathcal{P}(t)}+\epsilon]
\left(t,\frac{x-\alpha}{\lambda}\right)e^{i\gamma(t)}=\tilde{Q}+\tilde{u},
\end{align}
where $\mathcal{P}(t)=(a(t),b(t))$, and $\epsilon$ satisfies the orthogonality condition \eqref{mod-orthogonality-condition}. Finally, the following estimate hold:
\begin{align*}
&\|\tilde{u}\|_2\lesssim\lambda^{\frac{3}{2}},\ \|\tilde{u}\|_{H^{1/2}}\lesssim\lambda,\\
&\lambda(t)-\frac{t^2}{4A_0^2}=\mathcal{O}(\lambda^5),\ \frac{a}{\lambda^{\frac{1}{2}}}(t)-\frac{1}{A_0}=\mathcal{O}(\lambda^2),\
\frac{b}{\lambda}(t)-B_0=\mathcal{O}(\lambda^2),\\
&\gamma(t)=-\frac{4A_0^2}{t}+\gamma_0+\mathcal{O}(\lambda^{\frac{1}{2}}),\
\alpha(t)=x_0+\mathcal{O}(\lambda^{\frac{3}{2}}).
\end{align*}
 for $t\in[t_0,0)$ and $t$ sufficiently close to $0$.
Here $A_0>0$ and $B_0\in\mathbb{R}^2$ are the constant and the vector defined in \eqref{back-define-1} and \eqref{back-define-2}, respectively.
\end{thm}
\begin{proof}
Let $t_n\rightarrow0^{-}$ be a sequence of negative times and let $u_n$ be the solution to \eqref{equ-1-hf-2} with initial data at $t=t_n$ given by
\begin{align}\label{existence-1}
u_n(t_n,x)=\frac{1}{\lambda_n(t_n)}Q_{\mathcal{P}_n(t_n)}
\left(\frac{x-\alpha_n(t_n)}{\lambda_n(t_n)}\right)e^{i\gamma_n(t_n)},
\end{align}
where the sequence $\mathcal{P}_n(t_n)=(a_n(t_n),b_n(t_n))$ and $\{\lambda_n(t_n),\alpha_n(t_n)\}$ are given by
\begin{align}\label{existence-2}
a_n(t_n)&=-\frac{t_n}{2A_0},\ \lambda_n(t_n)=\frac{t_n^2}{4A_0^2},\
\gamma_n(t_n)=\gamma_0-\frac{4A_0^2}{t_n},\\\label{existence-3}
b_n(t_n)&=\frac{B_0t_n^2}{2A_0},\ \alpha_n(t_n)=x_0.
\end{align}
By lemma \ref{lemma-3app-2}, we have
\begin{align}
\int|u_n(t_n)|^2=\int|Q|^2+\mathcal{O}(t_n^6),
\end{align}
and $\tilde{u}(t_n)=0$ by construction. Thus $u_n$ satisfies the assumptions of lemma \ref{lemma-back}. Hence we can find a backwards time $t_0$ independent of $n$ such that for al $t\in[t_0,t_1)$ we have the geometric decomposition
\begin{align}
u_n(t,x)=\frac{1}{\lambda_n(t)}Q_{\mathcal{P}_n(t)}
\left(\frac{x-\alpha_n(t)}{\lambda_n(t)}\right)+\tilde{u}_n(t,x),
\end{align}
with the uniform bounds given by
\begin{align}
&\|D^{\frac{1}{2}}\tilde{u}_n\|_2^2
+\frac{\|\tilde{u}_n\|_2^2}{\lambda_n(t)}\lesssim\lambda_n^2(t),\\
&\left|\lambda_n(t)-\frac{t^2}{4A_0^2}\right|\lesssim K\lambda_n^{5}(t),\
\left|\frac{a_n(t)}{\lambda_n^{\frac{1}{2}}(t)}-\frac{1}{A_0}\right|\lesssim K\lambda_n^2(t),\
\left|\frac{b_n(t)}{\lambda_n(t)}-B_0\right|\lesssim K\lambda_n^2(t).
\end{align}
Next, we conclude that $\{u_n(t_0)\}_{n=1}^{\infty}$ converges strongly in $H^{\frac{1}{2}}(\mathbb{R}^2)$ (after passing to a subsequence if necessary). Indeed, from the uniform bound $\|\tilde{u}_n(t_0)\|_{H^{1/2}}\lesssim1$ we can assume (after passing to a subsequence if necessary) that $u_n(t_0)\rightharpoonup u_0$  weakly in $H^s(\mathbb{R}^2)$ for any $s\in[0,\frac{1}{2}]$. Moreover, we note the uniform bound
\begin{align}
\left|\frac{d}{dt}\int\chi_R|u_n|^2\right|
=&\left|\int\chi_R\left((u_n)_t\bar{u}_n+u_n(\bar{u}_n)_t\right)\right|\notag\\
=&\left|\int u_n\left[\chi_R,iD\right]\bar{u}_n\right|\notag\\
\lesssim&\|\nabla\chi_R\|_{\infty}\|u_n\|_2^2\lesssim\frac{1}{R},
\end{align}
with a smooth cutoff function $\chi_R(x)=\chi(\frac{x}{R})$ where $\chi(x)\equiv0$ for $|x|\leq 1$ and $\chi(x)\equiv1$ for $|x|\geq2$. Note that we used the commutator estimate (which we can see \cite{stein-1993})
\begin{align}\notag
\|[\chi_R,D]\|_{L^2\rightarrow L^2}\lesssim\|\nabla \chi_R\|_{\infty}.
\end{align}
By integrating the previous bound from $t_1$ to $t_0$ and using the previous estimate \eqref{existence-1}, \eqref{existence-2} and \eqref{existence-3}, we derive that for every $\tau>0$ there is a radius  $R>0$ such that
\begin{align}\notag
\int_{|x|\geq R}|u_n((t_0)|^2\leq\tau\ \text{for all}\ n\geq1.
\end{align}
Combining this fact with the weak convergence of $\{u_n(t_0)\}_{n=1}^{\infty}$ in $H^{s}(\mathbb{R}^2)$, we deduce that
\begin{align}
u_n(t_0)\rightarrow u_0\ \text{strongly in}\ H^s(\mathbb{R}^2)\ \text{for every}\ s\in[0,\frac{1}{2}].
\end{align}
Thus, by local well -- posedness, we can solve the Cauchy problem  \eqref{equ-1-hf-2} and find
$$ u \in C([t_0,T);H^{\frac{1}{2}}(\mathbb{R}^2)) $$
and
obtain
\begin{align}
u_n(t)\rightarrow u(t)\ \text{strongly in}\ H^{\frac{1}{2}}(\mathbb{R}^2)\ \text{for}\ t\in[t_0,T),
\end{align}
where $T>t_0$ is the lifetime of $u$ on the right. Moreover, $u$  for $t<\min\{T,0\}$ a geometrical decomposition of the form state in above with
\begin{align}
a_n(t)\rightarrow a(t),\ b_n(t)\rightarrow b(t),\ \lambda_n(t)\rightarrow\lambda(t),\ \gamma_n(t)\rightarrow\gamma(t),\ \alpha_n(t)\rightarrow\alpha(t).
\end{align}
Furthermore, we deduce that
\begin{align}\notag
\|\tilde{u}(t)\|_2\lesssim\lambda^\frac{3}{2}\ \text{and}\ \|\tilde{u}(t)\|_{H^{1/2}}\lesssim\lambda.
\end{align}
 for $t \in [t_0,T).$
In particular, this implies that $u(t)$ blows up at time $T=0$ such that
\begin{align}\notag
\|D^{\frac{1}{2}}u\|_2^2\sim\lambda^{-2}(t)\sim|t|^{-4}\ \text{as}\ t\rightarrow0^{-}.
\end{align}
In addition, we deduce from $L^2$-mass conversation and the strong convergence that
\begin{align}\notag
\|u\|_2=\lim_{n\rightarrow\infty}\|u_n(t_n)\|_2=\|Q\|_2.
\end{align}
As for the energy, we note that
\begin{align}\notag
E(u(t))=\frac{a^2}{\lambda}e_1+o(1)\rightarrow E_0\ \text{as}\ t\rightarrow0^{-},
\end{align}
by the choice of $A_0$, $a_n(t_n)$ and $\lambda_n(t_n)$. By energy conversation, this implies that
\begin{align}\notag
E(u)=E_0
\end{align}
Also, we observe that
\begin{align}\notag
P(u(t))=\frac{b}{\lambda}p_1+o(1)\rightarrow P_0\ \text{as}\ t\rightarrow0^{-},
\end{align}
by our choice of $B_0$, $a_n(t_n)$ and $\lambda_n(t_n)$. By momentum conversation, this shows that
\begin{align}\notag
P(u)=P_0.
\end{align}
Next, we recall that rough bound
\begin{align}\notag
|\tilde{\gamma}_s|\lesssim\lambda_n.
\end{align}
Therefore, using that $\frac{ds}{dt}=\lambda^{-1}$ and the estimates for $\lambda_n$
\begin{align}\notag
\left|\frac{d}{dt}\left(\gamma_n+\frac{4A_0^2}{t}\right)\right|
=\frac{1}{\lambda_n}\left|(\gamma_n)_s-\frac{4A_0^2\lambda_n}{t^2}\right|
=\frac{1}{\lambda_n}\left|(\tilde{\gamma}_n)_s-\frac{4A_0^2\lambda_n}{t^2}+1\right|\lesssim1.
\end{align}
Integrating this bound and using \eqref{existence-2} and $\lambda\sim t^2$, we find
\begin{align}\notag
\gamma_n(t)+\frac{4A_0^2}{t}=\gamma_0+\mathcal{O}(\lambda^{\frac{1}{2}}),
\end{align}
whence the claim for $\gamma$ follows, since we have $\lambda\sim t^2$. Finally, we recall the rough bound $\left|\frac{(\alpha_n)_s}{\lambda_n}+b_n\right|\lesssim\lambda_n$. Integrating this and bounds for $b_n$ and $\lambda_n$, we deduce that
\begin{align}\notag
\left|\frac{d}{dt}(\alpha_n-x_0)\right|=\left|\frac{(\alpha_n)_s}{\lambda_n}\right|
\lesssim\lambda_n+|b_n|\lesssim\lambda_n.
\end{align}
Integrating this and using \eqref{existence-3}, we find that
\begin{align}\notag
\alpha_n(t)=x_0+\mathcal{O}(\lambda^{\frac{3}{2}}),
\end{align}
which shows that the claims for $\alpha(t)$ holds.
\end{proof}

\appendix

\section{Appendix A}

In this section, we collect some regularity and decay estimates concerning the linearized operators $L_{+}$ and $L_{-}$.
\begin{lemma}\label{app-lemma-1}
Let $f,g\in H^k(\mathbb{R}^2)$ for some $k\geq0$ and suppose $f\bot Q$ and $g\bot \partial_{x_j}Q$, where $j=1,2$. Then we have the regularity bounds
\begin{align}\notag
\|L_{+}^{-1}g\|_{H^{k+1}}\lesssim\|g\|_{H^k},\ \ \|L_{-}^{-1}f\|_{H^{k+1}}\lesssim\|f\|_{H^k},
\end{align}
and the decay estimates
\begin{align}\notag
\|\langle x\rangle^{3}L_{-}^{-1}g\|_{\infty}\lesssim\|\langle x\rangle^{3}g\|_{\infty},\ \
\|\langle x\rangle^{3}L_{+}^{-1}f\|_{\infty}\lesssim\|\langle x\rangle^{3}f\|_{\infty}.
\end{align}
\end{lemma}
\begin{proof}
It suffices to prove the lemma for $L_{-}^{-1}g$, since the estimate for  $L_{+}^{-1}f$ follow in the same fashion.

To show the regularity bound, we can assume that $k\in\mathbb{N}$ is an integer. Let  $L_{-}^{-1}g=h$, and thus
\begin{align}\notag
Dh+h=Qh+g.
\end{align}
Note that $Q\in H^2(\mathbb{R}^2)\cap C^{\infty}(\mathbb{R}^2)$ for any $k\in\mathbb{N}$ by Sobolev embeddings and the fact that $Q\in H^{1/2}(\mathbb{R}^2)$. Applying $\nabla^k+1$ to the equation above and using the Leibniz rule and H\"{o}lder, we find that
\begin{align}\label{appendix-A-1}
\|h\|_{H^{k+1}}\sim\|(\nabla^k+1)(Dh+h)\|_2
\lesssim\|Q\|_{W^{k,\infty}}\|h\|_{H^k}+\|g\|_{H^k}.
\end{align}
Note, in particular, that $\|h\|_2=\|L_{-}^{-1}g\|_2\lesssim\|g\|_2$ holds, since $L_{-}$ has a bound inverse on $Q^{\bot}$. Hence \eqref{appendix-A-1} shows that the desired regularity estimates are true for $k=0$. By induction, we obtain the desired estimate $\|L_{-}^{-1}g\|_{H^{k+1}}\lesssim\|g\|_{H^k}$ for any integer $k\in\mathbb{N}$.

To show the decay estimate, we argue as follows. Assume that $\|\langle x\rangle^{-3}g\|_{\infty}<+\infty$, because otherwise there is nothing to prove. As above, let $L_{-}^{-1}g=h$ and rewrite the equation satisfied by $h$ in resolvent form:
\begin{align}\notag
h=\frac{1}{D}Qh+\frac{1}{D}g.
\end{align}
Let $R(x-y)=\mathcal{F}^{-1}(\frac{1}{|\xi|+1})(x-y)$ denote the associated kernel of the resolvent $(D+1)^{-1}$. From \cite{FrankLS2016}  we recall the standard fact that $R\in L^{p}(\mathbb{R}^2)$ for any $p\in[1,\infty]$ with $1-\frac{1}{p}<\frac{1}{2}$. Since $h\in L^2(\mathbb{R}^2)$, this implies that $(R*h)(x)$ is continuous and vanishes as $|x|\rightarrow\infty$. Moreover we have the pointwise bound
\begin{align}\notag
0<R(z)\lesssim\frac{1}{|x|^{3}},\ \ \text{for}\ |x|\geq1.
\end{align}
Using this bound and our decay assumption on $h(x)$, it is elementary to check that
\begin{align}\notag
|R*g(x)|\lesssim\min\{1,|x|^{-3}\}.
\end{align}
Using this bound, we can bootstrap the equation for $f$, using that $Q$ is continuous and vanishes at infinity; we refer to \cite{FJLenzmann-non-2007} for details on a similar decay estimate. This shows that $|f(x)|\lesssim\langle x\rangle^{-3}$ as desired.
\end{proof}

\section{On the modulation equations}

Here we collect some results and estimates regarding the modulation theory used in section \ref{section-mod-estimate}.
\subsection{Uniqueness of modulation parameters}\label{section-app-mod-1}
First, we show that the parameters $\{a,b,\lambda,\alpha,\gamma\}$ are uniquely determined if $\epsilon=\epsilon_1+i\epsilon_2\in H^{1/2}(\mathbb{R}^2)$ is sufficiently small and satisfies the orthogonality conditions \eqref{mod-orthogonality-condition}. Indeed, this follows from an implicit function argument, which we detail here.

For $\delta>0$, let $W_{\delta}=\{w\in H^{1/2}(\mathbb{R}^2):\|w-Q\|_{H^{1/2}}<\delta\}$. Consider approximate blowup profiles $Q_{\mathcal{P}}$ with $|\mathcal{P}|=|(a,b)|<\eta$, where $\eta>0$ is a small constant. For $w\in W_{\delta}$, $\lambda_1>0$, $y_1\in\mathbb{R}^2$, $\gamma\in\mathbb{R}$ and $|\mathcal{P}|<\eta$, we define
\begin{align}\notag
\epsilon_{\lambda_1,y_1,\gamma_1,a,b}(y)
=e^{i\gamma_1}\lambda_1w(\lambda_1y-y_0)-Q_{\mathcal{P}}.
\end{align}
Consider the map $\mathbf{\sigma}=(\sigma^1,\sigma^2,\sigma^3,\sigma^4,\sigma^5,\sigma^6,\sigma^7)$ define by
\begin{align*}
\sigma^1&=((\epsilon_{\lambda_1,y_0,\gamma_1,a,b})_2,\Lambda Q_{1\mathcal{P}})-
((\epsilon_{\lambda_1,y_1,\gamma_1,a,b})_1,\Lambda Q_{2\mathcal{P}}),\\
\sigma^2&=((\epsilon_{\lambda_1,y_1,\gamma_1,a,b})_2,\partial_aQ_{1\mathcal{P}})-
((\epsilon_{\lambda_1,y_1,\gamma_1,a,b})_1,\partial_aQ_{2\mathcal{P}}),\\
\sigma^3&=((\epsilon_{\lambda_1,y_1,\gamma_1,a,b})_1,\rho_2)-
((\epsilon_{\lambda_1,y_1,\gamma_1,a,b})_2,\rho_1),\\
\sigma^4&=((\epsilon_{\lambda_1,y_1,\gamma_1,a,b})_2,\partial_{1} Q_{1\mathcal{P}})-
((\epsilon_{\lambda_1,y_1,\gamma_1,a,b})_1,\partial_{1} Q_{2\mathcal{P}}),\\
\sigma^5&=((\epsilon_{\lambda_1,y_1,\gamma_1,a,b})_2,\partial_{2} Q_{1\mathcal{P}})-
((\epsilon_{\lambda_1,y_1,\gamma_1,a,b})_1,\partial_{2} Q_{2\mathcal{P}}),\\
\sigma^6&=((\epsilon_{\lambda_1,y_1,\gamma_1,a,b})_2,\partial_{b_1} Q_{1\mathcal{P}})-
((\epsilon_{\lambda_1,y_1,\gamma_1,a,b})_1,\partial_{b_1} Q_{2\mathcal{P}}),\\
\sigma^7&=((\epsilon_{\lambda_1,y_1,\gamma_1,a,b})_2,\partial_{b_2} Q_{1\mathcal{P}})-
((\epsilon_{\lambda_1,y_1,\gamma_1,a,b})_1,\partial_{b_2} Q_{2\mathcal{P}}).
\end{align*}
Recall that $\rho=\rho_1+i\rho_2$ was defined in \eqref{mod-definition-rho}. Taking the partial derivatives at $(\lambda_1,y_1,\gamma_1,a,b_1,b_2)=(1,0,0,0,0,0)$ yields that
\begin{align*}
\frac{\partial\epsilon_{\lambda_1,y_1,\gamma_1,a,b}}{\partial\lambda_1}=\Lambda w, \frac{\partial\epsilon_{\lambda_1,y_1,\gamma_1,a,b}}{\partial y_{1,j}}=-\partial_{j} w,
\frac{\partial\epsilon_{\lambda_1,y_1,\gamma_1,a,b}}{\partial\gamma_1}=iw,\\
\frac{\partial\epsilon_{\lambda_1,y_1,\gamma_1,a,b}}{\partial a}=-\partial_aQ_{\mathcal{P}}|_{\mathcal{P}=(0,0)}=-iS_{1,0},\\
\frac{\partial\epsilon_{\lambda_1,y_1,\gamma_1,a,b}}{\partial b_j}=-\partial_{b_j}Q_{\mathcal{P}}|_{\mathcal{P}=(0,0)}=-iS_{0,1,j},
\end{align*}
where we recall that $L_{-}S_{1,0}=\Lambda Q$ and $L_{-}S_{0,1}=-\nabla Q$. Note that $S_{1,0}$ is an radial function, whereas $S_{0,1}$ is antisymmetry. At $(\lambda_1,y_1,\gamma_1,a,b_1,b_2,w)=(1,0,0,0,0,0,0,Q)$, the Jacobian of the map $\sigma$ is hence given by
\begin{align*}
\frac{\partial \sigma^1}{\partial\lambda_1}=0,\ \frac{\partial\sigma^1}{\partial y_{1,1}}=0,\ \frac{\partial\sigma^1}{\partial y_{1,2}}=0,\ \frac{\partial \sigma^1}{\partial\gamma_1}=0,\ \frac{\partial \sigma^1}{\partial a}=-(S_{1,0},L_{-}S_{1,0}),\ \frac{\partial \sigma^1}{\partial b_1}=0,\ \frac{\partial \sigma^1}{\partial b_2}=0,\\
\frac{\partial \sigma^2}{\partial\lambda_1}=-(S_{1,0},L_{-}S_{1,0}),\ \frac{\partial\sigma^2}{\partial y_{1,1}}=0,\ \frac{\partial\sigma^2}{\partial y_{1,2}}=0,\ \frac{\partial \sigma^2}{\partial\gamma_1}=0,\ \frac{\partial \sigma^2}{\partial a}=0,\ \frac{\partial \sigma^2}{\partial b_1}=0,\ \frac{\partial \sigma^2}{\partial b_2}=0,\\
\frac{\partial \sigma^3}{\partial\lambda_1}=0,\ \frac{\partial\sigma^3}{\partial y_{1,1}}=0,\ \frac{\partial\sigma^3}{\partial y_{1,2}}=0,\ \frac{\partial \sigma^3}{\partial\gamma_1}=-(Q,\rho_1),\ \frac{\partial \sigma^3}{\partial a}=0,\ \frac{\partial \sigma^3}{\partial b_1}=0,\ \frac{\partial \sigma^3}{\partial b_2}=0,\\
\frac{\partial \sigma^4}{\partial\lambda_1}=0,\ \frac{\partial\sigma^4}{\partial y_{1,1}}=0,\ \frac{\partial\sigma^4}{\partial y_{1,2}}=0,\ \frac{\partial \sigma^4}{\partial\gamma_1}=0,\ \frac{\partial \sigma^4}{\partial a}=0,\ \frac{\partial \sigma^4}{\partial b_1}=-(L_{-}S_{0,1,1},S_{0,1,1}),\ \frac{\partial \sigma^4}{\partial b_2}=0,\\
\frac{\partial \sigma^5}{\partial\lambda_1}=0,\ \frac{\partial\sigma^5}{\partial y_{1,1}}=0,\ \frac{\partial\sigma^5}{\partial y_{1,2}}=0,\ \frac{\partial \sigma^5}{\partial\gamma_1}=0,\ \frac{\partial \sigma^5}{\partial a}=0,\ \frac{\partial \sigma^5}{\partial b_1}=0,\ \frac{\partial \sigma^5}{\partial b_2}=-(L_{-}S_{0,1,2},S_{0,1,2}),\\
\frac{\partial \sigma^6}{\partial\lambda_1}=0,\ \frac{\partial\sigma^6}{\partial y_{1,1}}=(L_{-}S_{0,1,1},S_{0,1,1}),\ \frac{\partial\sigma^6}{\partial y_{1,2}}=0,\  \frac{\partial \sigma^6}{\partial\gamma_1}=0,\ \frac{\partial \sigma^6}{\partial a}=0,\ \frac{\partial \sigma^6}{\partial b_1}=0\ \frac{\partial \sigma^6}{\partial b_2}=0,\\
\frac{\partial \sigma^7}{\partial\lambda_1}=0,\ \frac{\partial\sigma^7}{\partial y_{1,1}}=0,\ \frac{\partial\sigma^7}{\partial y_{1,2}}=(L_{-}S_{0,1,2},S_{0,1,2}),\  \frac{\partial \sigma^7}{\partial\gamma_1}=0,\ \frac{\partial \sigma^7}{\partial a}=0,\ \frac{\partial \sigma^7}{\partial b_1}=0\ \frac{\partial \sigma^7}{\partial b_2}=0.
\end{align*}
Note that we used here that $Q$ and $S_{1,0}$ are the radial symmetry functions, whereas $S_{0,1}$ is antisymmetry, for example $(Q,S_{0,1})=0$. Moreover, we note
\begin{align}\notag
-(Q,\rho_1)=(L_{+}\Lambda Q,\rho_1)=-(\Lambda Q,L_{+}\rho_1)=-(\Lambda Q,S_{1,0})=-(L_{-}S_{1,0},S_{1,0}).
\end{align}
Since $(L_{-}S_{1,0},S_{1,0})>0$ and $(L_{-}S_{0,1,j},S_{0,1,j})>0$, hence the determinant of the functional matrix is nonzero. By the implicit function theorem, we obtain existence and uniqueness for $(\lambda_1,y_1,\gamma_1,a,b_1,b_2,w)$ in some neighborhood around $(1,0,0,0,0,0,0,Q)$.

\subsection{Estimates for the modulation equations}
To conclude this section, we collect some estimates needed in the discussion of the modulation equations in section \ref{section-mod-estimate}.
\begin{lemma}\label{lemma-app-mod-estimate}
The following estimates hold
\begin{align}\label{app-B-1}
&(M_{-}(\epsilon)-a\Lambda\epsilon_1+b\cdot\nabla\epsilon_1,\Lambda Q_{2\mathcal{P}})+(M_{+}(\epsilon)+a\Lambda\epsilon_2-b\cdot\nabla\epsilon_2,\Lambda Q_{1\mathcal{P}})\notag\\
=&-\Re(\epsilon,Q_{\mathcal{P}})
+\mathcal{O}(\mathcal{P}^2\|\epsilon\|_2),\\\label{app-B-2}
&(M_{-}(\epsilon)-a\Lambda\epsilon_1+b\cdot\nabla\epsilon_1,\partial_a Q_{2\mathcal{P}})+(M_{+}(\epsilon)+a\Lambda\epsilon_2-b\cdot\nabla\epsilon_2,
\partial_a Q_{1\mathcal{P}})
=\mathcal{O}(\mathcal{P}^2\|\epsilon\|_2),\\\label{app-B-3}
&(M_{-}(\epsilon)-a\Lambda\epsilon_1+b\cdot\nabla\epsilon_1,\rho_2)
+(M_{+}(\epsilon)+a\Lambda\epsilon_2-b\cdot\nabla\epsilon_2,\rho_1 )
=\mathcal{O}(\mathcal{P}^2\|\epsilon\|_2),\\\label{app-B-4}
&(M_{-}(\epsilon)-a\Lambda\epsilon_1+b\cdot\nabla\epsilon_1,\partial_j Q_{2\mathcal{P}})+(M_{+}(\epsilon)+a\Lambda\epsilon_2-b\cdot\nabla\epsilon_2,\partial_{j} Q_{1\mathcal{P}})
=\mathcal{O}(\mathcal{P}^2\|\epsilon\|_2),\\\label{app-B-5}
&(M_{-}(\epsilon)-a\Lambda\epsilon_1+b\cdot\nabla\epsilon_1,\partial_{b_j} Q_{2\mathcal{P}})+(M_{+}(\epsilon)+a\Lambda\epsilon_2-b\cdot\nabla\epsilon_2,\partial_{b_j} Q_{1\mathcal{P}})
=\mathcal{O}(\mathcal{P}^2\|\epsilon\|_2).
\end{align}
\end{lemma}
\begin{proof}
First, we recall that
\begin{align*}
M_{+}(\epsilon)=&L_{+}\epsilon_1-|Q_{\mathcal{P}}|^{-1}
Q_{1\mathcal{P}}Q_{2\mathcal{P}}\epsilon_2+\mathcal{O}(\mathcal{P}\epsilon),\\
M_{-}(\epsilon)=&L_{-}\epsilon_2-|Q_{\mathcal{P}}|^{-1}
Q_{1\mathcal{P}}Q_{2\mathcal{P}}\epsilon_1+\mathcal{O}(\mathcal{P}\epsilon).
\end{align*}
We have notice the identity
\begin{align}\label{app-B-6}
L_{-}\Lambda S_{1,0}=-S_{1,0}+2(\Lambda Q)QS_{1,0}+\Lambda Q+\Lambda^2Q.
\end{align}
To see this relation, we recall that $L_{-}S_{1,0}=\Lambda Q$ and hence
\begin{align*}
L_{-}\Lambda S_{1,0}&=[L_{-},\Lambda]S_{1,0}+\Lambda L_{-}S_{1,0}=DS_{1,0}+(x\cdot\nabla Q)S_{1,0}+\Lambda^2Q\\
&=-S_{1,0}+|Q|S_{1,0}+\Lambda Q+(x\cdot\nabla Q)S_{1,0}+\Lambda^2Q\\
&=-S_{1,0}+\Lambda Q+(\Lambda Q)S_{1,0}+\Lambda^2Q.
\end{align*}
Hence, the identity is hold. Similarly, we deduce from $L_{-}S_{0,1}=-\nabla Q$ that
\begin{align}\label{app-B-7}
L_{-}\Lambda S_{0,1}=-S_{0,1}-\nabla Q+(\Lambda Q)S_{0,1}-\Lambda\nabla Q.
\end{align}
Next, we recall that
\begin{align}\notag
\Lambda Q_{1\mathcal{P}}=\Lambda Q+\mathcal{O}(\mathcal{P}^2),\ \Lambda Q_{2\mathcal{P}}=a\Lambda S_{1,0}+b\cdot\Lambda S_{0,1}+\mathcal{O}(\mathcal{P}^2).
\end{align}
Combining \eqref{app-B-6} and \eqref{app-B-7} with the fact and using that $L_{+}\Lambda Q=-Q$, we find that
\begin{align*}
&\text{left-hand side of \eqref{app-B-1}}\\
=&a(\epsilon_2,L_{-}\Lambda S_{1,0})+(\epsilon_2, b\cdot L_{-}\Lambda S_{0,1})-
|Q_{\mathcal{P}}|^{-1}
(Q_{1\mathcal{P}}Q_{2\mathcal{P}}\epsilon_1,a\Lambda S_{1,0}+b\cdot\Lambda S_{0,1})\\
&+(-a\Lambda\epsilon_1+b\cdot\Lambda\nabla\epsilon_1,a\Lambda S_{1,0}+b\cdot\Lambda S_{0,1})+(\epsilon_1,L_{+}\Lambda Q)+(a\Lambda\epsilon_2-b\cdot\nabla\epsilon_2,\Lambda Q)\\&-|Q_{\mathcal{P}}|^{-1}
(Q_{1\mathcal{P}}Q_{2\mathcal{P}}\epsilon_2,\Lambda Q)+\mathcal{O}(\mathcal{P}^2\|\epsilon\|_2)\\
=&-(\epsilon_1, Q)-a(\epsilon_2,S_{1,0})-(\epsilon_2,b\cdot S_{0,1})+a(\epsilon_2,\Lambda Q)
-(\epsilon_2,b\cdot\nabla Q)+\mathcal{O}(\mathcal{P}^2\|\epsilon\|_2)\\
=&-\Re(\epsilon, Q_{\mathcal{P}})+\mathcal{O}(\mathcal{P}^2\|\epsilon\|_2).
\end{align*}
Here we used that $a(\epsilon_2,\Lambda Q)=\mathcal{O}(\mathcal{P}^2\|\epsilon\|_2)$ and $(\epsilon_2,b\cdot\nabla Q)=\mathcal{O}(\mathcal{P}^2\|\epsilon\|_2)$, which follows from the orthogonality condition \eqref{mod-orthogonality-condition}.

$\mathbf{Estimate~\eqref{app-B-2}}$. From lemma \ref{lemma-3app-2} we recall that
\begin{align*}
\partial_aQ_{1\mathcal{P}}=2aT_{2,0}+b\cdot T_{1,1}+\mathcal{O}(\mathcal{P}^2),\ \partial_aQ_{2\mathcal{P}}=S_{1,0}+\mathcal{O}(\mathcal{P}^2),
\end{align*}
where
\begin{align*}
L_{+}T_{2,0}=\frac{1}{2}S_{1,0}-\Lambda S_{1,0}+\frac{1}{2}|S_{1,0}|^2,\
L_{+}T_{1,1}=S_{0,1}-\Lambda S_{0,1}+\nabla S_{1,0}+S_{1,0}S_{0,1}.
\end{align*}
Using this fact, we have
\begin{align*}
&\text{left-hand side of \eqref{app-B-2}}\\
=&(\epsilon_2,L_{-}S_{1,0})-|Q_{\mathcal{P}}|^{-1}
(Q_{1\mathcal{P}}Q_{2\mathcal{P}}\epsilon_1,S_{1,0})+a(\epsilon_1,\Lambda S_{1,0})\\
&-(\epsilon_1,b\cdot\nabla S_{1,0})+2a(\epsilon_1,L_{+}T_{2,0})+(\epsilon_1,b\cdot L_{+}T_{1,1})+
\mathcal{O}(\mathcal{P}^2\|\epsilon\|_2)\\
=&(\epsilon_2,\Lambda Q)-|Q_{\mathcal{P}}|^{-1}
((aQS_{1,0}+b\cdot QS_{0,1})\epsilon_1,S_{1,0})+a(\epsilon_1,\Lambda S_{1,0})\\
&-(\epsilon_1,b\cdot\nabla S_{1,0})+2a\left(\epsilon_1,\frac{1}{2}S_{1,0}-\Lambda S_{1,0}+\frac{1}{2}|S_{1,0}|^2\right)\\
&+(\epsilon_1,b\cdot S_{0,1}-b\cdot\Lambda S_{0,1})+\nabla S_{1,0}+b\cdot S_{1,0}S_{0,1})+
\mathcal{O}(\mathcal{P}^2\|\epsilon\|_2)\\
=&(\epsilon_2,\Lambda Q)-a(\epsilon_1,\Lambda S_{1,0})-(\epsilon_1,b\cdot\Lambda S_{0,1})+(\epsilon_1,b\cdot S_{0,1})+\mathcal{O}(\mathcal{P}^2\|\epsilon\|_2)\\
=&(\epsilon_2,\Lambda Q_{1\mathcal{P}})-(\epsilon_1,\Lambda Q_{2\mathcal{P}})+
\mathcal{O}(\mathcal{P}^2\|\epsilon\|_2),
\end{align*}
where in the last step we also used that $(\epsilon_1,b\cdot S_{0,1})=\mathcal{O}(\mathcal{P}^2\|\epsilon\|_2)$, thanks to the orthogonality condition \eqref{mod-orthogonality-condition}.

$\mathbf{Estimate~\eqref{app-B-3}}$. Indeed, by the definition of $\rho=\rho_1+i\rho_2$, we have
\begin{align*}
&\text{left-hand side of \eqref{app-B-3}}\\
=&(\epsilon_2,L_{-}\rho_2)+(\epsilon_1,L_{+}\rho_1)-
|Q_{\mathcal{P}}|^{-1}((aQS_{1,0}+b\cdot QS_{0,1})\epsilon_2,\rho_1)\\
&-a(\epsilon_2,\Lambda\rho_1)+(\epsilon_2,b\cdot\nabla\rho_1)
+\mathcal{O}(\mathcal{P}^2\|\epsilon\|_2)\\
=&a(\epsilon_2,S_{1,0}\rho_1)+a(\epsilon_2,\Lambda\rho_1)
-2a(\epsilon_2,T_{2,0})+(\epsilon_2,b\cdot S_{0,1}\rho_1)\\
&-(\epsilon_2,b\cdot\nabla\rho_1)-(\epsilon_2,b\cdot T_{1,1})+(\epsilon_1,S_{1,0})-
|Q_{\mathcal{P}}|^{-1}((aQS_{1,0}+b\cdot QS_{0,1})\epsilon_2,\rho_1)\\
&-a(\epsilon_2,\Lambda\rho_1)+(\epsilon_2,b\cdot\nabla\rho_1)
+\mathcal{O}(\mathcal{P}^2\|\epsilon\|_2)\\
=&-2a(\epsilon_2,T_{2,0})-(\epsilon_2,b\cdot T_{1,1})+(\epsilon_1,S_{1,0})
+\mathcal{O}(\mathcal{P}^2\|\epsilon\|_2)\\
=&-(\epsilon_2,\partial_aQ_{1\mathcal{P}})+(\epsilon_1,\partial_aQ_{2\mathcal{P}})
+\mathcal{O}(\mathcal{P}^2\|\epsilon\|_2)\\
=&\mathcal{O}(\mathcal{P}^2\|\epsilon\|_2),
\end{align*}
where we use the orthogonality condition \eqref{mod-orthogonality-condition}.

$\mathbf{Estimate~\eqref{app-B-4}}$. First, we note that
\begin{align*}
\nabla Q_{1\mathcal{P}}=\nabla Q+\mathcal{O}(\mathcal{P}^2),\
\nabla Q_{2\mathcal{P}}=a\nabla S_{1,0}+\sum_{j=1}^2b_j\nabla S_{0,1,j}+\mathcal{O}(\mathcal{P}^2).
\end{align*}
Moreover, we have the relation
\begin{align*}
L_{+}\nabla Q&=0,\ L_{-}\nabla S_{1,0}=(\nabla Q)S_{1,0}+\nabla\Lambda Q,\\
L_{-}\nabla S_{0,1}&=(\nabla Q)\cdot S_{0,1}-\nabla \cdot\nabla Q.
\end{align*}
Indeed, since $[D,\nabla]=0$, we have
\begin{align*}
L_{+}\nabla Q=(D+1-2Q)\nabla Q=\nabla(DQ+Q-Q^{2})=0.
\end{align*}
Due to the fact that $[L_{-},\nabla]=[Q,\nabla]$, we obtain
\begin{align*}
L_{-}\nabla S_{1,0}=[L_{-},\nabla]S_{1,0}+\nabla L_{-}S_{1,0}=(\nabla Q)S_{1,0}+\nabla\Lambda Q,
\end{align*}
where we use $L_{-}S_{1,0}=\Lambda Q$. Similarly, we can obtain $L_{-}\nabla S_{0,1,j}=(\nabla Q)\cdot S_{0,1,j}-\nabla(\partial_{x_j} Q),\,j=1,2$. Thus, we deduce
\begin{align*}
&\text{left-hand side of \eqref{app-B-4}}\\
=&a(\epsilon_2,L_{-}\partial_j S_{1,0})+b_j(\epsilon_2,L_{-}\partial_j S_{0,1,j})+(\epsilon_1,L_{+}\partial_j Q)-a(\epsilon_2QS_{1,0},\partial_j Q)\\
&-b_j(\epsilon_2QS_{0,1,j},\partial_j Q)-a(\epsilon_2,\Lambda\partial_j Q)
+b_j(\epsilon_2,\partial_j(\partial_{x_j}Q))+\mathcal{O}(\mathcal{P}^2\|\epsilon\|_2)\\
=&a(\epsilon_2,\partial_j QS_{1,0}+\partial_j\Lambda Q)
+b_j(\epsilon_2,\partial_j QS_{0,1,j}-\partial_j(\partial_{x_j}Q))\\
&-a(\epsilon_2QS_{1,0},\partial_j Q)-b_j(\epsilon_2QS_{0,1,j},\partial_j Q)
-a(\epsilon_2,\Lambda\partial_j Q)\\
&+b_j(\epsilon_2,\partial_j(\partial_{x_j}Q))+\mathcal{O}(\mathcal{P}^2\|\epsilon\|_2)\\
=&a(\epsilon_2,[\partial_j,\Lambda]Q)+\mathcal{O}(\mathcal{P}^2\|\epsilon\|_2)\\
=&a(\epsilon_2,\partial_j Q)+\mathcal{O}(\mathcal{P}^2\|\epsilon\|_2)=
\mathcal{O}(\mathcal{P}^2\|\epsilon\|_2),
\end{align*}
since $a(\epsilon_2,\partial_jQ)=\mathcal{O}(\mathcal{P}^2\|\epsilon\|_2)$ due to the orthogonality condition \eqref{mod-orthogonality-condition}.

$\mathbf{Estimate~\eqref{app-B-5}}$. We note that
\begin{align*}
\partial_{b_j}Q_{1\mathcal{P}}=aT_{1,1,j}+2b_jT_{0,2,j}+\mathcal{O}(\mathcal{P}^2),\ \partial_{b_j}Q_{2\mathcal{P}}=S_{0,1,j}+\mathcal{O}(\mathcal{P}^2),
\end{align*}
where
\begin{align*}
L_{+}T_{0,2,j}=\partial_{x_j} S_{0,1,j}+\frac{1}{2}|S_{0,1,j}|^2.
\end{align*}
Using the above relations, we obtain that
\begin{align*}
&\text{left-hand side of \eqref{app-B-5}}\\
=&(\epsilon_2,L_{-}S_{0,1,j})-a(\epsilon_1QS_{1,0},S_{0,1,j})
-b_j(\epsilon_1S_{0,1,j}Q,S_{0,1})+a(\epsilon_1,\Lambda S_{0,1,j})\\
&-b_j(\epsilon_1,\partial_{x_j} S_{0,1,j})+a(\epsilon_1,L_{+}T_{1,1,j})+2b_j(\epsilon_2,L_{+}T_{0,2,j})
+\mathcal{O}(\mathcal{P}^2\|\epsilon\|_2)\\
=&-(\epsilon_2,\partial_{x_j} Q)-a(\epsilon_1QS_{1,0},S_{0,1,j})
-b_j(\epsilon_1S_{0,1,j}Q,S_{0,1,j})+a(\epsilon_1,\Lambda S_{0,1,j})\\
&-b_j(\epsilon_1,\partial_{x_j} S_{0,1,j})
+a(\epsilon_1,S_{0,1,j}-\Lambda S_{0,1,j}+\nabla S_{1,0}+S_{0,1,j}S_{1,0})\\
&+2b_j(\epsilon_1,\partial_{x_j} S_{0,1,j}+\frac{1}{2}|S_{0,1,j}|^2))
+\mathcal{O}(\mathcal{P}^2\|\epsilon\|_2)\\
=&-(\epsilon_2,\partial_{x_j} Q)+a(\epsilon_1,\partial_{x_j} S_{1,0})+b_j(\epsilon_1,\partial_{x_j} S_{0,1,j})
+\mathcal{O}(\mathcal{P}^2\|\epsilon\|_2)\\
=&-(\epsilon_2,\partial_{x_j} Q_{1\mathcal{P}})+(\epsilon_1,\partial_{x_j} Q_{2\mathcal{P}})
+\mathcal{O}(\mathcal{P}^2\|\epsilon\|_2)\\
=&\mathcal{O}(\mathcal{P}^2\|\epsilon\|_2)
\end{align*}
where in the last step we use the orthogonality condition \eqref{mod-orthogonality-condition} and hence we proven  this lemma.
\end{proof}
\section{Coercivity estimate for the localized energy}
In the following, we assume that $A>0$ is a sufficiently large constant. Let $\phi:\mathbb{R}^2\rightarrow\mathbb{R}$ be the smooth cutoff function introduced in \eqref{energy-cutoff-function}, Section \ref{section-refined-energy}. For $\epsilon=\epsilon_1+i\epsilon_2\in H^{1/2}(\mathbb{R}^2)$, we consider the quadratic forms
\begin{align}\label{app-c-define-1}
L_{+,A}(\epsilon_1):&=\int_{s=0}^{\infty}\sqrt{s}\int\Delta \phi_A|\nabla(\epsilon_1)_{s}|^2dxds+\int|\epsilon_{1}|^2-2\int Q|\epsilon_1|^2\\\label{app-c-define-2}
L_{-,A}(\epsilon_2):&=\int_{s=0}^{\infty}\sqrt{s}\int\Delta \phi_A|\nabla(\epsilon_2)_{s}|^2dxds+\int|\epsilon_{2}|^2-\int Q|\epsilon_2|^2,
\end{align}
where $\Delta \phi_A=\Delta(\phi(\frac{x}{A}))$. As in lemma \ref{lemma-energy-estimate}, we denote
\begin{align}\label{app-c-define-3}
u_s=\sqrt{\frac{2}{\pi}}\frac{1}{-\Delta+s}u.\  \ \text{for}\ s>0.
\end{align}
We start with the following simple identity.

For $u\in H^{1/2}(\mathbb{R}^2)$, we have
\begin{align}\label{app-c-identity}
\int_0^{\infty}\sqrt{s}\int_{\mathbb{R}^2}|\nabla u_s|^2dxds=\|D^{1/2}u\|_2^2.
\end{align}
Indeed, by applying Fubini's theorem and using Fourier transform, we find that
\begin{align}\notag
\int_0^{\infty}\sqrt{s}\int_{\mathbb{R}^2}|\nabla u_s|^2dxds= \frac{1}{4\pi^2}\frac{2}{\pi}
\int_{\mathbb{R}^2}\int_0^{\infty}\frac{\sqrt{s}ds}{(\xi^2+s)^2}|\xi|^2|\hat{u}(\xi)|^2d\xi
=\|D^{1/2}u\|_2^2.
\end{align}
In general, we have
\begin{align}\label{app-c-identity-2}
\frac{2}{\pi}\int_0^{\infty}\sqrt{s}\int_{\mathbb{R}^2}|(-\Delta)^{\alpha/2}u_s|^2dxds
=\|D^{\alpha-\frac{1}{2}}u\|_2^2.
\end{align}
Next, we establish a technical result, which show that, when taking the limit $A\rightarrow+\infty$, the quadratic form $\int_0^{\infty}\sqrt{s}\int\Delta\phi_A|\nabla u_s|^2dxds+\|u\|_2^2$ defines a weak topology that serves as a useful substitute for weak convergence in $H^{1/2}(\mathbb{R}^2)$. The precise statement reads as follows.
\begin{lemma}\label{app-lemma-c-1}
Let $A_n\rightarrow\infty$ and suppose that $\{u_n\}_{n=1}^{\infty}$ is a sequence in $H^{1/2}(\mathbb{R}^2)$ such that
\begin{align}\notag
\int_0^{\infty}\sqrt{s}\int\Delta\phi_{A_{n}}|\nabla (u_n)_s|^2dxds+\|u_n\|_2^2\leq C,
\end{align}
for some constant $C>0$ independent of $n$. Then, after possibly passing to a subsequence of $\{u_n\}_{n=1}^{\infty}$, we have that
\begin{align}\notag
u_n\rightharpoonup u\ \text{weakly in}\ L^2(\mathbb{R}^2)\ \text{and}\ u_n\rightarrow u\ \text{strongly in}\ L^2_{loc}(\mathbb{R}^2),
\end{align}
and $u\in H^{1/2}(\mathbb{R}^2)$. Moreover, we have the bound
\begin{align}\notag
\|D^{1/2}u\|_2^2\leq\liminf_{n\rightarrow\infty}
\int_0^{\infty}\sqrt{s}\int\Delta\phi_{A_n}|\nabla (u_n)_s|^2dxds.
\end{align}
\end{lemma}
\begin{proof}
Let $\eta\in\mathcal{S}(\mathbb{R}^2)$ be a smooth cutoff function in Fourier space that satisfies
\begin{align}\notag
\hat{\eta}(\xi)=\begin{cases}1\ \ \text{for}\ |\xi|\leq1,\\
0\ \ \text{for}\ |\xi|\geq2.
\end{cases}
\end{align}
For any $u\in H^{1/2}(\mathbb{R}^2)$, we write $u=u^l+u^h$ with
\begin{align}\notag
\hat{u}^l=\hat{\eta}\hat{u},\ \hat{u}^h=(1-\hat{\eta})\hat{u}.
\end{align}
Recall the definition of $u_s$, we readily notice that the relations
\begin{align}\notag
(u^l)_s=(u_s)^l,\ (u^h)_s=(u_s)^h.
\end{align}

$\mathbf{Step~1~control~of~u^h}$. Let $\chi\in\mathcal{C}_0^{\infty}(\mathbb{R}^2)$ be a smooth cutoff function such that
\begin{align}\notag
\chi(x)=\begin{cases}1\ \ \text{for}\ |x|\leq1,\\
0\ \ \text{for}\ |x|\geq2.
\end{cases}
\end{align}
For any $R>0$ given, we set
\begin{align}\notag
\chi_R(x)=\chi\left(\frac{x}{R}\right).
\end{align}
We now claim the following control: For any $R>0$, there exist constants $C_R>0$ and $A_0=A_0(R)>0$ such that for any $A\geq A_0$ and $ u\in H^{1/2}(\mathbb{R}^2)$, we have
\begin{align}\label{app-c-claim}
\int|D^{1/2}(\chi_Ru^h)|^2\leq C_R\left[\int_0^{\infty}\sqrt{s}\int\Delta\phi_A|\nabla u_s^h|^2dxds+\|u\|_2^2\right].
\end{align}
Indeed, from definition \eqref{app-c-define-3} we see that
\begin{align}\label{app-c-1}
-\Delta(\chi_Ru^h)_s+s(\chi_Ru^h)_s=\sqrt{\frac{2}{\pi}}\chi_Ru^h.
\end{align}
On the other hand, an elementary calculation show that
\begin{align*}
-\Delta(\chi_R(u_s)^h)+s\chi_R(u_s)^h=&\chi_R(-\Delta (u_s)^h+s(u_s)^h)-2\nabla\chi_R\cdot\nabla (u_s)^h-(u_s)^h\Delta\chi_R\\
=&\sqrt{\frac{2}{\pi}}\chi_R(u_s)^h-2\nabla\chi_R\cdot\nabla(u_s)^h-(u_s)^h\Delta\chi_R.
\end{align*}
Therefore, the function
\begin{align}\notag
w:=\sqrt{\frac{\pi}{2}}\left\{(\chi_Ru^h)_s-\chi_R(u_s)^h\right\}
\end{align}
satisfies the equation
\begin{align}\notag
-\Delta w+sw=\sqrt{\frac{\pi}{2}}\left\{2\nabla\chi_R\cdot \nabla (u_s)^h+(u_s)^h\Delta\chi_R\right\}.
\end{align}
Hence, we deduce the bound
\begin{align}\notag
\int_{\mathbb{R}^2}|\nabla w|^2+s\int_{\mathbb{R}^2}|w|^2\lesssim
\int_{\mathbb{R}^2}\left\{|\nabla\chi_R||\nabla (u_s)^h|+|(u_s)^h||\Delta\chi_R|\right\}|w|.
\end{align}
By using the Cauchy-Schwarz inequality, we conclude that
\begin{align*}
\int_{\mathbb{R}^2}|\nabla w|^2+s\int_{\mathbb{R}^2}|w|^2\leq C_R
\left\{\int_{|x|\leq2R}|\nabla (u_s)^h|^2+\int|(u_s)^h|^2\right\}\ &\text{for}\ s\geq1,\\
\int_{\mathbb{R}^2}|\nabla w|^2+s\int_{\mathbb{R}^2}|w|^2\leq\frac{C_R}{s}
\left\{\int|\nabla (u_s)^h|^2+\int|(u_s)^h|^2\right\}\ &\text{for}\ 0<s<1.
\end{align*}
Next, we apply identity \eqref{app-c-identity-2} and note that $\hat{u}^h(\xi)=0$ for $|\xi|\leq1$. For some sufficiently large $A>A_0(R)$, we obtain
\begin{align*}
\int_1^{\infty}\sqrt{s}\int_{\mathbb{R}^2}|\nabla w|^2dxds&\leq
C_R\int_0^{\infty}\sqrt{s}\left\{\int_{|x|\leq2R}|\nabla (u_s)^h|^2+\int|(u_s)^h|^2\right\}ds\\
&\leq C_R\int_0^{\infty}\sqrt{s}\int\Delta\phi_A|\nabla (u_s)^h|^2dxds+\|D^{-\frac{1}{2}}u^h\|_2^2\\
&\leq C_R\int_0^{\infty}\sqrt{s}\int\Delta\phi_A|\nabla (u_s)^h|^2dxds+\|u\|_2^2
\end{align*}
and
\begin{align*}
\int_0^1\sqrt{s}\int_{\mathbb{R}^2}|\nabla w|^2dxds
\leq &C_R\int_0^1\frac{\sqrt{s}}{s}
\int\frac{(1+|\xi|^2)|\hat{u}^h|^2}{(s+|\xi|^2)^2}d\xi ds\\
\leq&C_R\int_0^1\frac{1}{\sqrt{s}}
\int\frac{1+|\xi|^2}{|\xi|^4}|\hat{u}^h|^2d\xi ds\\
=&C_R\int_0^1\frac{1}{\sqrt{s}}\int_{|\xi|\geq1}
\frac{1+|\xi|^2}{|\xi|^4}|\hat{u}|^2d\xi ds\\
\leq&C_R\|u\|_2^2.
\end{align*}
Using \eqref{app-c-identity} and the previous bounds, we find that
\begin{align*}
&\|D^{\frac{1}{2}}(\chi_Ru^h)\|_2^2\\
=&\int_0^{\infty}\sqrt{s}\int|\nabla(\chi_Ru^h)_s|^2dxds\\
\leq&\int_0^{\infty}\sqrt{s}\int|\nabla w|^2dxds+\int_0^{\infty}\sqrt{s}\int|\nabla(\chi_R(u^h))_s|^2dxds\\
\leq&C_R\left(\int_0^{\infty}\sqrt{s}\int\Delta\phi_A|\nabla (u_s)^h|^2dxds+\|u\|_2^2\right)+
\int_0^{\infty}\sqrt{s}\int|u^h|^2dxds\\
\leq&C_R\left(\int_0^{\infty}\sqrt{s}\int\Delta\phi_A|\nabla (u_s)^h|^2dxds+\|u\|_2^2\right),
\end{align*}
which shows the claim \eqref{app-c-claim}.

$\mathbf{Step~2~Conclusion}$ Let $\{u_n\}_{n=1}^{\infty}$ satisfy the assumption in this lemma. By \eqref{app-c-identity}, we have for all $A>0$ that
\begin{align*}
\int_0^{\infty}\sqrt{s}\int\Delta\phi_A|\nabla (u_n^l)_s|^2dxds
&\leq\int_0^{\infty}\sqrt{s}\int|\nabla (u^l_n)_s|^2dxds\\
&=\|D^{\frac{1}{2}}u^l_n\|_2^2
\leq C\|u_n\|_2^2\leq C.
\end{align*}
Here we used the definition of $u^l$. Thus the assumed bound in this lemma ensures that
\begin{align}\notag
\int_0^{\infty}\sqrt{s}\int\Delta\phi_{A_n}|\nabla(u_n^h)_s|^2dxds\lesssim C.
\end{align}
Therefore we conclude from \eqref{app-c-claim} that, for all $R>0$, the $\{u_n\}_{n=0}^{\infty}$ is a bounded sequence in $H^{1/2}(B_R)$ and $L^2(\mathbb{R}^2)$. Hence, by passing to a subsequence if necessary, we can find that
\begin{align}\notag
u_n\rightharpoonup u\ \text{in}\ L^2(\mathbb{R}^2)\ \text{and}\ u_n\rightharpoonup u\ \text{in}\ H^{1/2}(\mathbb{R}^2)\ \text{for all}\ R>0.
\end{align}
By the compactness of the Sobolev embedding $H^{1/2}(\mathbb{R}^2)\hookrightarrow L^2_{loc}(\mathbb{R}^2)$, we also have
\begin{align}\notag
u_n\rightarrow u\ \text{in}\ L^2_{loc}(\mathbb{R}^2).
\end{align}
It remains to show the ``weak lower semicontinuity property" given by
\begin{align}\notag
\|D^{1/2}u\|_2^2=\int_0^{\infty}\sqrt{s}\int|\nabla u_s|^2dxds
\leq\liminf_{n\rightarrow+\infty}\int_0^{\infty}\sqrt{s}\int\Delta\phi_{A_n}|\nabla (u_n)_s|^2dxds.
\end{align}
Indeed, we first note that
\begin{align}\notag
\nabla(u_n)_s(x)=\sqrt{\frac{2}{\pi}}\int \nabla(G^s(x-y))u_n(y)dy.
\end{align}
Since $u_n\rightharpoonup u$ weakly in $L^2(\mathbb{R}^2)$ and $\nabla G^s(x-y)\in L^2_y(\mathbb{R}^2)$ for any $x\in\mathbb{R}^2$, we thus obtain
\begin{align}\notag
\nabla(u_n)_s(x)\rightarrow\nabla u_s(x)\ \text{pointwise on }\ \mathbb{R}^2\ \text{for any} \ s>0.
\end{align}
Next, by the Cauchy-Schwarz inequality, we derive the uniform pointwise bound
\begin{align}\notag
|\nabla (u_n)_s|\lesssim\|\nabla G^s\|_2\|(u_n)_s\|_2\leq C.
\end{align}
Let $0<\epsilon<1$ and $B>0$ be given. By the dominated convergence theorem, we deduce that
\begin{align*}
\int_{s=\epsilon}^{1/\epsilon}\sqrt{s}\int_{|x|\leq B}|\nabla u_s|dxds
=&\lim_{n\rightarrow\infty}\int_{s=\epsilon}^{1/\epsilon}\sqrt{s}\int_{|x|\leq B}|\nabla (u_n)_s|dxds\\
\leq&\liminf_{n\rightarrow+\infty}\int_0^{\infty}\sqrt{s}\int\Delta\phi_{A_n}|\nabla (u_n)_s|dxds,
\end{align*}
where in the last step we used Fatou's lemma and the fact that $\Delta\phi_{A_n}\geq0$ satisfies $\lim_{n\rightarrow+\infty}\phi_{A_n} =1$ for all $x\in\mathbb{R}^2$. Since the previous bound holds for arbitrary $0<\epsilon<1$ and $B>0$, we conclude that
\begin{align}\notag
\|D^{1/2}u\|_2^2=\int_0^{\infty}\sqrt{s}\int|\nabla u_s|^2dxds\leq
\liminf_{n\rightarrow+\infty}\int_0^{\infty}\sqrt{s}\int\Delta\phi_{A_n}|\nabla (u_n)_s|dxds.
\end{align}
The proof of lemma \ref{app-lemma-c-1} is complete.
\end{proof}

\begin{lemma}\label{lemma-app-c-2}
Let $L_{+,A}(\epsilon_1)$ and $L_{-,A}(\epsilon_2)$ be the quadratic forms defined \eqref{app-c-define-1} and \eqref{app-c-define-2}, respectively. Then there exist  constants $C_0>0$ and $A_0>0$ such that, for all $A\geq A_0$ and all $\epsilon=\epsilon_1+i\epsilon_2\in H^{1/2}(\mathbb{R}^2)$, we have the coercivity estimate
\begin{align}\notag
(L_{+,A}\epsilon_1,\epsilon_1)+(L_{-,A}\epsilon_2,\epsilon_2)\geq C_0\int|\epsilon|^2-\frac{1}{C_0}\left\{(\epsilon_1,Q)^2+(\epsilon_1,S_{1,0})^2
+|(\epsilon_1,S_{0,1})|^2+|(\epsilon_2,\rho_1)|^2\right\}.
\end{align}
Here $S_{1,0}$ and $S_{0,1}$ are the unique functions such that $L_{-}S_{1,0}=\Lambda Q$ with $(S_{1,0},\partial_{j}Q)=0$, where $j=1,2$ and $L_{-}S_{0,1}=-\nabla Q$ with $(S_{0,1},Q)=0$, respectively, and the function $\rho_1$ is defined in \eqref{mod-definition-rho}.
\end{lemma}
\begin{proof}
It suffices to prove the coercivity bound
\begin{align}\label{app-c-2}
(L_{-,A}\epsilon_2,\epsilon_2)\geq C_0\int|\epsilon|^2-\frac{1}{C_0}|(\epsilon_2,\rho_1)|^2,
\end{align}
since the corresponding estimate for $L_{+,A}$ follows by the same argument.

To prove \eqref{app-c-2}, we argue by contradiction as follows. Suppose that there exists a sequence of functions $\{u_n\}_{n=1}^{\infty}\in H^{1/2}(\mathbb{R}^2)$ with
\begin{align}\notag
\|u_n\|_2^2=1,\ (u_n,\rho_1)=0,
\end{align}
as well as a sequence $A_n\rightarrow+\infty$ such that
\begin{align}\label{app-c-3}
\int_0^{+\infty}\sqrt{s}\int\Delta\phi_{A_n}|u_n|^2dxds+\int|u_n|^2
-\int|Q||u_n|^2\leq o(1)\int|u_n|^2,
\end{align}
where $o(1)\rightarrow0$ as $n\rightarrow+\infty$. By applying lemma \ref{app-lemma-c-1}, we find that (after passing to a subsequence if necessary)
\begin{align}\notag
u_n\rightharpoonup u\ \text{weakly in}\ L^2(\mathbb{R}^2)\ \text{and}\
u_n\rightarrow u\ \text{strongly in}\ L^2_{loc}(\mathbb{R}^2).
\end{align}
But since $Q(x)\rightarrow0$ as $|x|\rightarrow+\infty$, we easily check that $\int|Q||u_n|^2\rightarrow\int|Q||u|^2$. Moreover, from \eqref{app-c-3} and $\|u_n\|_2^2=1$, we deduce that $\int|Q||u|^2\geq1$ must hold. In particular, the weak limit $u\not\equiv0$ is nontrivial. However, by the weak lower semicontinuity inequality in lemma \ref{app-lemma-c-1} and the fact that $\liminf_{n\rightarrow\infty}\int|u_n|^2\geq\int|u|^2$, we deduce that
\begin{align}\notag
(L_{-}u,u)=\int|D^{\frac{1}{2}}u|^2+\int|u|^2-\int Q|u|^2\leq0,\ (u,\rho_1)=0.
\end{align}
Since $u\not\equiv0$, this bound contradicts the coercivity estimate for $L_{-}$ stated in below.
\end{proof}
\begin{lemma}\label{lemma-app-c-3}
For any $u\in L^2(\mathbb{R}^2)$, we have the bound
\begin{align}\notag
\left|\int_{s=0}^{+\infty}\sqrt{s}\int\Delta^2\phi_{A}|u_s|^2dxds\right|
\lesssim\frac{1}{A}\|u\|_2^2.
\end{align}
\end{lemma}
\begin{proof}
First, we recall that $\Delta\phi_A(x)=\Delta\phi(\frac{x}{A})$ and hence $\Delta^2\phi(x)=\frac{1}{A^2}\Delta^2(\frac{x}{A})$. Now we consider the following integral
\begin{align}\notag
\frac{1}{A^2}\int_{0}^{\infty}\int\sqrt{s}\Delta^2\phi(\frac{x}{A})|u_s|^2dxds=:I+II,
\end{align}
where
\begin{align}\notag
I=\frac{1}{A^2}\int_{0}^{\tau}\sqrt{s}\int\Delta^2\phi(\frac{x}{A})|u_s|^2dxds,\
II=\frac{1}{A^2}\int_{\tau}^{\infty}\sqrt{s}\int\Delta^2\phi(\frac{x}{A})|u_s|^2dxds.
\end{align}
Here $\tau>0$ is some given number. We can integral by parts twice and use the H\"{o}lder inequality to deduce that
\begin{align*}
|I|&\lesssim\|\Delta\phi\|_{\infty}\int_{0}^{\tau}\sqrt{s}(\|\Delta u_s\|_2\|u_s\|_2+\|\nabla u_s\|_2^2)ds\\
&\lesssim\int_{0}^{\tau}\sqrt{s}\left(\left\|\frac{-\Delta}{-\Delta+s}u\right\|_2
\left\|\frac{1}{-\Delta+s}u\right\|_2+\left\|\frac{\nabla}{-\Delta+s}u\right\|_2^2\right)ds\\
&\lesssim\int_0^{\tau}\frac{1}{s^{1/2}}ds\|u\|_2^2\lesssim\sqrt{\tau}\|u\|_2^2.
\end{align*}
where we use the bound $\|u_s\|_2\lesssim s^{-1}\|u\|_2$. To estimate $II$.
\begin{align}\notag
|II|\lesssim\frac{1}{A^2}\|\Delta^2\phi\|_{\infty}
\int_{\tau}^{+\infty}\frac{1}{s^{\frac{3}{2}}}ds\|u\|_2^2
\lesssim\frac{1}{A^2\sqrt{\tau}}\|u\|_2^2.
\end{align}
Thus, we have show that
\begin{align}\notag
\frac{1}{A^2}\int_{0}^{\infty}\int\sqrt{s}\Delta^2\phi(\frac{x}{A})|u_s|^2dxds
\lesssim\left(\sqrt{\tau}+\frac{1}{A^2\sqrt{\tau}}\right)\|u\|_2^2
\leq\frac{1}{A}\|u\|_2^2.
\end{align}
In the last step, we minimizing this bound with respect to $\tau$.
\end{proof}

\begin{lemma}\label{lemma-app-coercivity-estimate}
There exist a constant $C_1>0$ such that for all $\epsilon=\epsilon_1+i\epsilon_2\in H^{1/2}(\mathbb{R}^2)$, we have the coercivity estimate
\begin{align}\notag
(L_{+}\epsilon_1,\epsilon_1)+(L_{-}\epsilon_2,\epsilon_2)\geq C_1\int|\epsilon|^2-\frac{1}{C_1}\left\{(\epsilon_1,Q)^2+(\epsilon_1,S_{1,0})^2
+|(\epsilon_1,S_{0,1})|^2+|(\epsilon_2,\rho_1)|^2\right\}.
\end{align}
Here $S_{1,0}$ and $S_{0,1,j}$ are the unique functions such that $L_{-}S_{1,0}=\Lambda Q$ with $(S_{1,0},Q)=0$ and $L_{-}S_{0,1}=-\nabla Q$ with $(S_{0,1,j},Q)=0$, respectively, and the function $\rho_1$ is defined in \eqref{mod-definition-rho}.
\end{lemma}
\begin{proof}
From \cite{FrankLS2016} we recall that the key fact that the null space of $L_{+}$ and $L_{-}$ are given by
\begin{align}\label{eq.ke1}
\ker L_{+}=span\{\nabla Q\}=span\{\partial_1Q,\partial_2Q\},\ \ker L_{-}=span\{Q\}.
\end{align}
Moreover, $L_+$ has a unique negative eigenvalue, while $L_- \geq 0.$

If we consider the minimization problem
$$ \inf_{g \in H^{1/2}, g \perp Q} ( L_- g, g )_{L^2},$$
then we have two possibilities: or there exists $ g \neq 0, g \in H^{1/2}, g \perp Q$ such that
$$ L_- g =  a Q$$
or
\begin{equation}\label{eq.se1}
(L_- g,g ) \geq C \|g\|_{H^{1/2}}^2, \ \ \forall  g \perp Q.
\end{equation}
The first possibility leads easy to a contradiction. Indeed, multiplying by $Q$ we find $a=0$ and then we arrive at a contradiction with \eqref{eq.ke1}.
Therefore, remains \eqref{eq.se1} and this estimate implies
\begin{equation}\label{eq.se2}
   (L_- \epsilon_2,\epsilon_2 ) \geq C \|\epsilon_2\|_{H^{1/2}}^2 - \frac{1}{C} |(\epsilon_2,Q)|^2.
\end{equation}
In a similar way, we consider the minimization problem
$$ \inf_{g \in H^{1/2}, g \perp \varphi, g \perp \nabla Q} ( L_+ g, g )_{L^2},$$
and using the argument of sections 7.1 and 7.2 in deduce
\begin{equation}\label{eq.se3}
(L_+ g,g ) \geq C \|g\|_{H^{1/2}}^2, \ \ \forall  g \perp \varphi, g \perp \nabla Q .
\end{equation}
As before this estimate implies
\begin{equation}\label{eq.se4}
   (L_+ \epsilon_1,\epsilon_1 ) \geq C \|\epsilon_1\|_{H^{1/2}}^2 - \frac{1}{C} \left( |(\epsilon_1,\varphi)|^2+ |(\epsilon_1,\nabla Q)|^2 \right).
\end{equation}

{From \eqref{eq.se2} and \eqref{eq.se4} we find }

\begin{align}\label{app-c-4}
(L_{+}\epsilon_1,\epsilon_1)+(L_{-}\epsilon_2,\epsilon_2)\geq C_1\|\epsilon\|_{H^{1/2}}^2-\frac{1}{C_1}\{(\epsilon_1,\varphi)^2+|(\epsilon_1,\nabla Q)|^2+|(\epsilon_2,Q)|^2\}
\end{align}
for all $\epsilon=\epsilon_1+i\epsilon_2\in H^{1/2}(\mathbb{R}^2)$, where $C_1>0$ is some constant. Here $\varphi=\varphi(x)>0$ with $\|\varphi\|_2^2=1$ denotes the unique ground state eigenfunction of $L_{+}$, and we have $L_{+}\varphi=e\varphi$ with some $e<0$.

To derive the coercivity estimate in this lemma from an estimate of the form \eqref{app-c-4}, we can use some arguments that can be found, for example, in \cite{MerleR2006}. For the reader's convenience, we provide the details of the adaptation to our case. To prove the desired coercivity estimate, we can that assume $\epsilon=\epsilon_1+i\epsilon_2\in H^{1/2}(\mathbb{R}^2)$ satisfies
\begin{align}\notag
(\epsilon_1,S_{1,0})=(\epsilon_1,S_{0,1,j})=(\epsilon_2,\rho_1)=0.
\end{align}
and let the auxiliary function $\tilde{\epsilon}=\tilde{\epsilon_1}+i\tilde{\epsilon_2}$ satisfy
\begin{align}\notag
\tilde{\epsilon}=\epsilon-a\Lambda Q-ibQ-\sum_{j=1}^2c_j\partial_{x_j} Q,\ \text{for}\ j=1,2
\end{align}
where $a,b,c$ are chosen such that
\begin{align}\notag
(\tilde{\epsilon}_1,\varphi)=(\tilde{\epsilon}_2,Q)=(\tilde{\epsilon}_1,\partial_{x_j} Q)=0,\ \text{for}\ j=1,2.
\end{align}
That is
\begin{align}\notag
a=\frac{(\epsilon_1,\varphi)}{(\Lambda Q,\varphi)},\ b=\frac{(\epsilon_2,Q)}{(Q,Q)},\ c_j=\frac{(\epsilon_1,\partial_{x_j}Q)}{(\nabla Q,\nabla Q)}, \ \text{for}\ j=1,2
\end{align}
where we also used that $(\Lambda Q,\nabla Q)=0$ holds and $(\varphi,\nabla Q)=0$ since $\nabla Q\in\ker L_{+}$ and $\varphi\in ran L_{+}$. Next, recall that $L_{+}\varphi=e\varphi$ with $e<0$ and $L_{+}\Lambda Q=-Q$. Hence $(\Lambda Q,\varphi)=-\frac{1}{e}(Q,\varphi)>0$, by the strict positivity of $Q>0$ and $\varphi>0$. On the other hand, the orthogonality conditions satisfied by $\epsilon=\epsilon_1+i\epsilon_2$ imply that
\begin{align}\notag
a=\frac{(\tilde{\epsilon}_1,S_{1,0})}{(\Lambda Q,S_{1,0})},\ b=-\frac{\tilde{\epsilon}_2,\rho_1}{(Q,\rho_1)},\ c_j=-\frac{(\tilde{\epsilon}_1,S_{0,1,j})}{(\nabla Q,S_{0,1})},\,j=1,2,
\end{align}
where we use that $(\Lambda Q,S_{0,1})=(\nabla Q,S_{1,0})=0$, since $Q$ and $S_{1,0}$ are radial symmetry and $S_{0,1}$ is antisymmetry. Note that $L_{-}S_{1,0}=\Lambda Q$ and hence $(\Lambda Q,S_{1,0})=(L_{-}S_{1,0},S_{1,0})\neq0$, and $(\nabla Q,S_{0,1})=-(L_{-}S_{0,1},S_{0,1})<0$ because of $L_{-}S_{0,1}=-\nabla Q$. Furthermore, recall that $L_{+}\rho_1=S_{1,0}$. Thus $(Q,\rho_1)=-(\Lambda Q,S_{1,0})=(L_{-}S_{1,0},S_{1,0})>0$ again. In summary, we find that
\begin{align}\notag
\frac{1}{K}\|\epsilon\|_{H^{1/2}}\leq\|\tilde{\epsilon}\|_{H^{1/2}}\leq K\|\epsilon\|_{H^{1/2}},
\end{align}
with some constant $K>0$. Now, since $(\Lambda Q,Q)=(\nabla Q,Q)=0$ and $L_{+}\Lambda Q=-Q$ as well as $L_{+}\nabla Q=0$ and $L_{-}Q=0$, we obtain
\begin{align*}
&(\tilde{\epsilon}_1,Q)=(\epsilon_1,Q),\ (L_{-}\tilde{\epsilon}_2,\tilde{\epsilon}_2)=(L_{-}\epsilon_2,\epsilon_2)\\ &(L_{+}\tilde{\epsilon}_1,\tilde{\epsilon}_1)=(L_{+}\epsilon_1,\epsilon_1)+a(\epsilon_1,Q).
\end{align*}
By the previous relations and estimate \eqref{app-c-4}, we conclude
\begin{align*}\notag
(L_{+}\epsilon_1,\epsilon_1)+(L_{-}\epsilon_2,\epsilon_2)
=&(L_{+}\tilde{\epsilon_1},\tilde{\epsilon_1}+(L_{-}\tilde{\epsilon_2},\tilde{\epsilon_2}))-
a(\epsilon_1,Q)\\
\geq&C_1\|\tilde{\epsilon}\|_{H^{1/2}}^2-a(\epsilon_1,Q)\geq C_0\|\epsilon\|_{H^{1/2}}^2-\frac{1}{C_0}(\epsilon_1,Q)^2,
\end{align*}
with some sufficiently small constant $C_0>0$.
\end{proof}

\textbf{Acknowledgements}

The authors are grateful to the anonymous referee for the careful reading and valuable suggestions. YL was supported by the China Scholarship Council (201906180041).

\end{document}